\documentclass[final,leqno,onefignum,onetabnum]{article}

\usepackage{amsmath}
\usepackage{amsthm}
\usepackage{latexsym}
\usepackage{amssymb}
\usepackage{lipsum}
\usepackage{amsfonts}
\usepackage{graphicx}
\usepackage{epstopdf}
\usepackage{algorithmic}
\usepackage{hyperref}
\usepackage{geometry}
\newcommand{\esp}{\mathbb{E}}
\newcommand{\prob}{\mathbb{P}}

\newcommand{\alg}{\mathcal{F}}
\newcommand{\alge}{\mathcal{G}}

\newcommand{\re}{\mathbb{R}}
\newcommand{\polar}{\mathbb{N}}
\DeclareMathOperator*{\dom}{dom}
\DeclareMathOperator*{\interior}{int}
\newcommand{\constr}{\mathcal{I}}
\DeclareMathOperator*{\dist}{d}
\DeclareMathOperator*{\diam}{\mathcal{D}}
\DeclareMathOperator*{\argmin}{argmin}
\DeclareMathOperator*{\cl}{\mbox{cl}}
\newcommand{\tang}{\mathbb{T}}
\newtheorem{assump}{Assumption}
\newtheorem{definition}{Definition}
\newtheorem{algo}{Algorithm}
\newtheorem{rem}{Remark}
\newtheorem{lemma}{Lemma}
\newtheorem{theorem}{Theorem}
\newtheorem{corollary}{Corollary}
\newtheorem{proposition}{Proposition}

\begin{document}

\title{Incremental constraint projection methods for monotone stochastic variational inequalities} 
		
\author{A. N. Iusem, Instituto Nacional de Matem\'atica Pura e Aplicada (IMPA), \href{}{iusp@impa.br},\and
Alejandro Jofr\'e, Center for Mathematical Modeling (CMM) \& DIM, \href{}{ajofre@dim.uchile.cl},\and
Philip Thompson, Instituto Nacional de Matem\'atica Pura e Aplicada (IMPA), \href{}{philip@impa.br}
}
\date{}
\maketitle

\begin{abstract}
We consider stochastic variational inequalities with  monotone operators defined as the expected value of a random operator. 
We assume the feasible set is the intersection of a large family of convex sets. We propose a method that combines stochastic 
approximation with incremental constraint projections meaning that at each iteration, a step similar to some variant of a deterministic 
projection method is taken after the random operator is sampled and a component of the intersection defining the feasible set is 
chosen at random. Such sequential scheme is well suited for applications involving large data sets, online optimization and distributed 
learning. 
First, we assume that the variational inequality is weak-sharp. We provide asymptotic convergence, 
feasibility rate of $O(1/k)$ in terms of the mean squared distance to the feasible set and solvability rate of $O(1/\sqrt{k})$ 
(up to first order logarithmic terms) in terms of the mean distance to the solution set for a bounded or unbounded feasible set. 
Then,  we assume just monotonicity of the operator and introduce an explicit iterative Tykhonov regularization to the method.
 We consider Cartesian variational inequalities so as to encompass the distributed solution of stochastic Nash games or multi-agent 
optimization problems under a limited coordination. We provide asymptotic convergence, feasibility rate of $O(1/k)$ in terms of the 
mean squared distance to the feasible set and, in the case of a compact set, we provide a near-optimal solvability convergence 
rate of $O\left(\frac{k^\delta\ln k}{\sqrt{k}}\right)$ in terms of the mean dual gap-function of the SVI for arbitrarily small $\delta>0$.
\end{abstract}

\section{Introduction}\label{s1}

The standard (deterministic) variational inequality problem, which we will denote as VI($T,X)$ or simply VI, is defined as follows: 
given a closed and convex set $X\subset\re^n$ and a single-valued operator $T:\mathbb{R}^n\rightarrow\mathbb{R}^n$, 
find $x^*\in X$ such that, for all $x\in X$,
\begin{equation}\label{VI}
\langle T(x^*),x-x^*\rangle\ge0.
\end{equation}
We shall denote by $X^*$ the solution set of VI$(T,X)$. The variational inequality problem includes many interesting special 
classes of variational problems 
with applications in economics, game theory and engineering. The basic prototype is smooth convex optimization, where 
$T$ is the gradient of a smooth function.  
Other classes of problems are posed as variational inequalities which are not equivalent to optimization problems, 
such as \emph{complementarity} problems (with $X=\re^n_+$), \emph{system of equations} (with $X=\re^n$), \emph{saddle-point} 
problems and many different classes of \emph{equilibrium problems}. 

In the stochastic case, we start with a measurable space $(\Xi,\alge)$, a measurable (random) operator 
$F:\Xi\times\re^n\to\re^n$ and a random variable $v:\Omega\rightarrow\Xi$ defined on a probability space 
$(\Omega,\alg,\prob)$ which induces an expectation $\esp$ and a distribution $\prob_v$ of $v$. When no confusion arises, we 
use $v$ to also denote a random sample  $v\in\Xi$. We assume 
that for every $x\in\re^n$, $F(v,x):\Omega\rightarrow\re^n$ is an integrable random vector.
The solution criterion analyzed in this paper consists
of solving VI($T,X$) as defined by \eqref{VI}, where $T:\re^n\to\re^n$ is the expected value of $F(v,\cdot)$, i.e., for any $x\in\re^n$,
\begin{equation}\label{expected}
T(x)=\esp[F(v,x)].
\end{equation} 
Precisely, the definition of \emph{stochastic variational inequality} problem (SVI) is:

\quad

\begin{definition}[SVI]\label{SVI.def}
Assuming that $T:\re^n\rightarrow\re^n$ is given by $T(x)=\esp[F(\xi,x)]$ for all $x\in\re^n$, the SVI problem consists 
of finding $x^*\in X$, such that $\langle T(x^*),x-x^*\rangle\ge0$ for all $x\in X$.
\end{definition}	

\quad

Such formulation of SVI is also called \emph{expected value} formulation. It goes back to G\"urkan et al. \cite{robinson2}, 
as a natural generalization of stochastic optimization (SP) problems. Recently, 
a more general definition of stochastic variational inequality was considered 
in Chen et al. \cite{chen&wets} where the feasible set is also affected by randomness, that is, $X:\Xi\rightrightarrows\re^n$ 
is a random set-valued function.

Methods for the deterministic VI($T,X$) have been extensively studied (see Facchinei and Pang \cite{facchinei}). If $T$ is fully available 
then SVI can be solved by these methods. As in the case of SP, the SVI in Definition \ref{SVI.def} becomes 
very different from the deterministic setting when $T$ is \emph{not available}. This is often the case in practice due to 
expensive computation of the expectation in \eqref{expected}, unavailability of $\prob_v$ or no close form of $F(v,\cdot)$. 
This requires sampling the random variable $v$ and the use of values of $F(\eta,x)$ given a sample $\eta$ of $v$ and a current 
point $x\in\re^n$ (a procedure often called ``stochastic oracle'' call).
In this context, there are two current methodologies for solving the SVI problem: \emph{sample average approximation} (SAA)  
and \emph{stochastic approximation} (SA). In this paper we focus on the SA approach.

The SA methodology for SP or SVI can be seen as a projection-type method where the exact mean operator $T$ is replaced along the iterations by a 
random sample of $F$. This approach induces an stochastic error $F(v,x)-T(x)$ for $x\in X$ along the trajectory of 
the method. When $X=\re^n$, Definition \ref{SVI.def} 
becomes the stochastic equation problem (SE): under \eqref{expected}, almost surely find $x^*\in\re^n$ such that $T(x^*)=0$.	
The SA methodology was first proposed by Robbins and Monro in \cite{rob.monro} for the SE problem in the case in which 
$T$ is the gradient of a strongly convex function under specific conditions. Since this fundamental work, 
SA approaches to SP and, more recently for SVI, have been carried out in
Jiang and Xu \cite{Xu}, Juditsky et al. \cite{nem}, Yousefian et al. \cite{uday0}, Koshal et al. \cite{Uday}, 
Wang and Bertsekas \cite{wang}, Chen et al. \cite{lan}, Yousefian et al. \cite{uday2}, Kannan and Shanbhag \cite{uday4}, 
Yousefian et al. \cite{uday5}. 
See Bach and Moulines \cite{bach} for the stochastic approximation procedure in machine learning and online optimization.

A frequent additional difficulty is the possibly complicated structure 
of the \emph{feasible set} $X$. Often, the feasible set takes the form 
$$
X=\cap_{i\in\constr}X_i,
$$ 
where $\{X_i:i\in\constr\}$ is an arbitrarily family of closed convex sets. There are different motivations for considering 
the design of algorithms which, at every iteration, use only a component $X_i$ rather than the whole feasible set $X$. 
First,  in the case of projection methods, 
when the orthogonal projection onto each $X_i$, namely $\Pi_i:\re^n\to X_i$, is much easier to compute 
than the projection onto $X$, namely $\Pi:\re^n\to X$, a natural idea consists of replacing, at iteration $k$, 
$\Pi$ by one of the $\Pi_i$'s, say $\Pi_{i_k}$, or even by an approximation of $\Pi_i$.  
This occurs, for instance, when $X$ is a polyhedron and the $X_i$'s are halfspaces. 
This procedure is the basis
of the so called sequencial or parallel \emph{row action methods} for solving systems of equations (see Censor \cite{cen}) 
and methods for the \emph{feasibility problem}, useful
in many applications, including image restoration and tomography (see, e.g., Bauschke et al. \cite{luke}, Cegielski and Suchocka \cite{suchocka}). 
Second, in some cases $X$ is not known a priori, but is rather revealed through 
the random realizations of its components $X_i$. 
Such problems arise in fair rate allocation problems in wireless networks 
where the channel state is unknown but the channel states $X_i$ are observed 
in time (see e.g. Nedi\'c \cite{nedich} and Huang et al. \cite{huang}). Third, 
in some cases $X$ is known but the number of constraints  
is prohibitively very large (e.g., in machine learning and signal processing).

\subsection{Projection methods}\label{ss1.1}
In the deterministic setting \eqref{VI}, the classical projection method for VI$(T,X)$,
akin to the projected gradient method for convex optimization, is
\begin{equation}\label{algorithm.bertsekas}
x^{k+1}=\Pi[x^k-\alpha_k T(x^k)],
\end{equation}
where $\Pi$ is the projection operator onto $X$ and $\{\alpha_k\}$ is an 
exogenous sequence of positive stepsizes. Convergence of this method is guaranteed assuming $T$
is strongly monotone, Lipschitz continuous and the stepsizes satisfy 
$\alpha_k\in(0,2\sigma/L^2)$ and $\inf_k\alpha_k>0$, where $\sigma>0$ is the modulus of 
strong monotonicity and $L$ is the Lipschitz constant, see e.g. Facchinei and Pang \cite{facchinei}.

The strong monotonicity assumption is quite demanding, and convergence of \eqref{algorithm.bertsekas} is not guaranteed
when the operator is just monotone. In order to deal with this situation, Korpelevich \cite{korpelevich} proposed 
the extra-gradient algorithm 
\begin{eqnarray}\label{algorithm:extra:gradient}
z^{k}=\Pi[x^k-\alpha_k T(x^k)],&&\nonumber\\
x^{k+1}=\Pi[x^k-\alpha_k T(z^k)],
\end{eqnarray}
in which an additional auxiliary projection step is introduced. Convergence of the method is guaranteed when  
the stepsizes satisfy $\alpha_k\equiv\alpha\in (0,1/L)$. 
In Nemirovski \cite{nem2}, the extra-gradient method was generalized and convergence rates were established assuming compactness of 
the feasible set.

Observe that the projection method \eqref{algorithm.bertsekas} and the 
extra-gradient method \eqref{algorithm:extra:gradient} are \emph{explicit}, i.e., 
the formula for obtaining $x^{k+1}$ is an explicit one, up to the computation of the orthogonal projection $\Pi$. 
An \emph{implicit} approach for the solution of monotone variational inequalities consists of 
a Tykhonov or proximal regularization scheme (see Facchinei and Pang \cite{facchinei}, Chapter 12). 
In these methods, a sequence of regularized variational inequality problems are approximately solved at each iteration.

As commented before, a typical case occurs when the feasible set takes the form  
$
X=\cap_{i=1}^m X_i,
$
where all the $X_i$'s are closed and convex. Row action methods and alternate (or cyclic) 
projection algorithms for convex feasibility problems exploit the computation of projections onto the 
components iteratively (see Bauschke \cite{bauschke}). In such case, the order in which the sets $X_i$ 
are used along the iterations, i.e. the so called \emph{control sequence} $\{\omega_k\}\subset \{1, \dots , m\}$, must be specified. 
Several options have been considered in the literature (such as cyclic control, almost cyclical control, 
most violated constraint control and random control). A negative consequence 
of the use of approximate projections 
is the need to use \emph{small stepsizes}, i.e., satisfying $\sum_k\alpha^2<\infty$ and $\sum_k\alpha_k=\infty$, 
which significantly reduces the efficiency of the method. We thus have a trade-off between easier projection computation 
and slower convergence. Additionally, the use of approximate projections require some condition on the feasible 
set, so that the projections onto the sets $X_i$'s are reasonable approximations of the projection onto $X$. 
For this, some form of error bound, 
linear regularity or Slater-type conditions 
on the sets $X_i$ 
must be assumed 
(e.g., Assumption \ref{assump:constraint:reg} 
in Subsection \ref{ss2.2} and the comments following it). See Bauschke and Borwein \cite{bau.bor}, 
Deutsch and Hundal \cite{deu.hun} and Pang \cite{pang}.
Explicit methods for monotone variational inequalities using approximate projections were studied e.g. in Fukushima \cite{fukushima}
and Censor and Gibali \cite{censor}, imposing rather demanding coercivity assumptions on $T$, 
in Bello Cruz and Iusem \cite{iusbel1} assuming paramonotonicity
of $T$, and then in Bello Cruz and Iusem \cite{iusbel2} assuming just monotonicity of $T$. Another method of this type,  using an Armijo
search as in Iusem and Svaiter \cite{iusem} for determining the stepsizes, and approximate projections with the most violated constraint control,
can be found in Bello Cruz and Iusem \cite{yunier}.  

Related to row-action and alternate projective methods are the so called \emph{incremental} methods, introduced in Kibardin \cite{kibardin} 
(see also Luo and Tseng \cite{luo.tseng}, Bertsekas \cite{bertsekas.2}, Nedi\'c \cite{nedich} and references therein). 
These methods are used for the minimization of a \emph{large} sum of convex functions, e.g.
in machine learning applications. In such a context, instead of using the gradient of the sum, 
the gradient of one of the terms is selected iteratively under different control rules. In Polyak \cite{polyak.2}, 
Polyak \cite{polyak.3} and Nedi\'c \cite{nedich}, \emph{incremental constraint} methods with random control rules 
were proposed for minimizing a convex function over an intersection of a \emph{large} number convex sets.
The feasible set takes the form
\begin{equation}\label{assump:constr:intro}
X=X_0\cap\left(\cap_{i\in\constr}X_i\right),
\end{equation}
where $\{X_0\}\cup\{X_i:i\in\constr\}$ is a collection of closed and convex subsets of $\re^n$. 
The \emph{hard} constraint $X_0$ is assumed to have easy computable projections. The \emph{soft} 
constraints $\{X_i:i\in\constr\}$, for a given $i\in\constr$, has the form:
\begin{equation}\label{assump:constr:intro:repres}
X_i=\{x\in\re^n:g_i(x)\le0\},
\end{equation}
for some convex function $g_i$ with positive part $g_i^+(x):=\max\{g_i(x),0\}$ and easy computable subgradients. 
The method on Nedi\'c \cite{nedich} is given by:
\begin{equation}\label{algo:nedich:eq1}
y^{k}=\Pi_{X_0}\left[x^k-\alpha_k\nabla f(x^k)\right],
\end{equation}
\begin{equation}\label{algo:nedich:eq2}
x^{k+1}=\Pi_{X_0}\left[y^{k}-\beta_k\frac{g^+_{\omega_k}(y^{k})}{\Vert d^{k}\Vert^2}d^{k}\right],
\end{equation}
where $\{\alpha_k,\beta_k\}$ are positive stepsizes, $d^{k}\in\partial g^+_{\omega_k}(y^{k})\setminus\{0\}$ if $g^+_{\omega_k}(y^k)>0$, and 
$d^k=d$ for any $d\in\re^{n}\setminus\{0\}$ if $g^+_{\omega_k}(y^k)=0$. In the 
method \eqref{algo:nedich:eq1}-\eqref{algo:nedich:eq2}, $\{\omega_k\}$ is a random control sequence 
taking values in $\constr$ and satisfying certain conditions and $f:\re^n\rightarrow\re$ is a convex smooth function 
(the non-smooth case is also analyzed). Together with row-action and alternate projection methods, 
incremental constraint projection methods can be viewed as the dual version of (standard) incremental methods. 
More recently, stochastic approximation was incorporated to incremental constraint projections methods 
for stochastic convex minimization problems in Wang and Bertsekas \cite{wang2}.

\subsection{Stochastic approximation methods}\label{ss1.2}

The first SA method for SVI was analyzed in Jiang and Xu \cite{Xu}. Their method is:
\begin{eqnarray}\label{algorithm.xu1}
x^{k+1}=\Pi[x^k-\alpha_k F(v^k,x^k)],
\end{eqnarray} 
where $\Pi$ is the Euclidean projection onto $X$, $\{v^k\}$ is a sample of $v$ and $\{\alpha_k\}$ is 
a sequence of positive steps. The a.s. convergence is proved assuming $L$-Lipschitz continuity of $T$, 
strong monotonicity or strict monotonicity of $T$, stepsizes satisfying
$
\sum_k\alpha_k=\infty,\sum_k\alpha_k^2<\infty
$
(with $0<\alpha_k<2\rho/L^2$ in the case where $T$ is $\rho$-strongly monotone) and an unbiased oracle with 
uniform variance, i.e., there exists $\sigma>0$ such that for all $x\in X$,
\begin{equation}\label{noise.bound2}
\esp\left[\Vert F(v,x)-T(x)\Vert^2\right]\le\sigma^2.
\end{equation}

After the above mentioned work, recent research on SA methods for SVI have been developed 
in Juditsky et al. \cite{nem}, Yousefian et al. \cite{uday5, uday0, uday2}, Koshal et al. \cite{Uday}, 
Chen et al. \cite{lan}, Kannan and Shanbhag \cite{uday4}. Two of the main concerns in these papers were the extension of the SA 
approach to the general monotone case and the derivation of (optimal) convergence rate and complexity results with 
respect to known metrics associated to the VI problem. In order to analyze the monotone case, SA methodologies based on 
the extragradient method of Korpelevich \cite{korpelevich}, the mirror-prox algorithm of Nemirovski \cite{nem2} and iterative 
Tykhonov and proximal regularization procedures (see Kannan and Shandbag \cite{uday1}), were used in these works. 
Other objectives were the use of incremental constraint projections in the case of difficulties accessing the feasible set 
in Wang and Bertsekas \cite{wang}, the convergence analysis in the absence of the Lipschitz constant in 
Yousefian et al. \cite{uday5, uday0, uday2}, and the distributed solution of Cartesian variational 
inequalities in Yousefian et al. \cite{uday0}, Koshal et al. \cite{Uday}.

We finally make some comments on two recent methods upon which we make substantial improvements.

In Wang and Bertsekas \cite{wang}, method \eqref{algorithm.xu1} is improved by incorporating an incremental 
projection scheme, instead of exact ones. They take $X=\cap_{i\in\constr}X_i$, where $\constr$ is a finite index set, 
and use a random control sequence, where both the random map $F$ and the control sequence $\{\omega_k\}$ are jointly sampled, 
giving rise to the following algorithm:
\begin{eqnarray}\label{algorithm.wang}
y^k &=& x^k-\alpha_kF(v^k,x^k)\nonumber\\
x^{k+1} &=& y^k-\beta_k(y^k-\Pi_{\omega_k}(y^k)),
\end{eqnarray}
where $\{\alpha_k,\beta_k\}$ are positive stepsizes and $\{v^k\}$ are samples. When $\beta_k\equiv1$, 
the method is the version of method \eqref{algorithm.xu1} with incremental constraint projections. 
For convergence, the operator is assumed to be \emph{strongly monotone} and Lipschitz-continuous and \emph{knowledge} of 
the strong monotonicity and Lipschitz moduli are required for computing the stepsizes. In this setting, method \eqref{algorithm.wang} 
improves upon method \eqref{algo:nedich:eq1}-\eqref{algo:nedich:eq2} when $X_0=\re^n$, $\constr$ is finite and  the projection onto each 
$X_i$ is easy.
	
Regularized iterative Tychonov and proximal point methods for monotone stochastic variational inequalities were introduced 
in Koshal et al. \cite{Uday}. 
In such methods, instead of solving a sequence of regularized variational inequality problems, the regularization 
parameter is updated in each iteration and a \emph{single projection step} associated with the regularized problem is taken. 
This is desirable since (differently from the deterministic case), termination criteria are generally hard to meet in the 
stochastic setting. The algorithm proposed allows for a Cartesian structure on the variational inequality, so as to encompass 
the \emph{distributed} solution of Cartesian SVIs. Namely, the feasible set $X\subset\re^n$ has the the form
$
X=X^1\times\cdots\times X^m,
$
where each Cartesian component $X^j\subset\re^{n_j}$ is a closed and convex set, $v=(v_1,\ldots,v_m)$ and the random 
operator has components $F=(F_1(v_1,\cdot),\ldots,F_m(v_m,\cdot))$ with $F_j(v_j,\cdot):\Xi\times\re^n\rightarrow\re^{n_j}$ 
for $j=1,\ldots,m$ and $\sum_{j=1}^m n_j=n$. The algorithm  in Koshal et al. \cite{Uday} 
is described as follows. Given the $k$-th iterate $x^k\in X$ with components $x^k_j\in X^j$, for $j=1,\ldots,m$, 
the next iterate is given by the distributed projection computations: for $j=1,\ldots,m$,
\begin{equation}\label{tichonov.exact}
x_j^{k+1}=\Pi_{X^j}[x_j^k-\alpha_{k,j}(F_j(v^k_j,x^k)+\epsilon_{k,j} x_j^k)],
\end{equation}
where $\{\alpha_{k,1},\ldots,\alpha_{k,m}\}$ are the 
stepsize sequences, $\{\epsilon_{k,1},\ldots,\epsilon_{k,m}\}$ are the regularization parameter sequences 
and $\{v_{1}^k,\ldots,v_{m}^k\}$ are the samples. This method is shown to converge 
under monotonicity and Lipschitz-continuity of $T$ and a partial coordination 
between the stepsize and regularization parameter sequences (see Assumption \ref{approx.step.tik}). 
The iterative proximal point follows a similar pattern but differently from the Tykhonov method, 
this method requires strict monotonicity, which in particular implies uniqueness of solutions. It should me mentioned that 
two important classes of problems which can be formulated as stochastic Cartesian variational inequalities are the 
stochastic \emph{Nash equilibria} and the stochastic \emph{multi-user optimization} problem; see Koshal et 
al. \cite{Uday} for a precise definition. In these problems, the $i$-th agent has only access to its constraint set $X^i$ 
and $F_i$ (which depends on other agents decisions) so that a distributed solution of the SVI is required. Moreover, 
it is convenient to allow agents to update \emph{independently} their stepsizes and regularization sequences, subjected 
just to a limited coordination.
	
\subsection{Proposed methods and contributions}\label{ss1.3} 

In many stochastic approximation methods, the stochastic error $\varsigma(x):=F(v,x)-T(x)$ is assumed to be bounded, 
demanding the use of small stepsizes with a slow performance. In this case, the use of easily computable approximate projections, 
instead of exact ones, can significantly improve the performance of the algorithm. Additionally, in many cases the constraint set $X$ 
is known, but it contains a very large number of constraints, or $X$ is not known a priori, but is rather learned along 
time through random samples of its constraints. An important feature of incremental constraint projection methods is that 
they process sample operators and sample constraints \emph{sequentially}. This incremental structure is well suited 
for a variety of applications involving large data sets, online optimization and distributed learning. 
For problems that require online learning, incremental projection methods of the type \eqref{algo:nedich:eq1}-\eqref{algo:nedich:eq2} 
or \eqref{algorithm.wang} are practically the only option to be used without the knowledge of all the constraints. 

In view of these considerations, we wish to devise methods which incorporate \emph{incremental constraint projections} 
with \emph{stochastic approximation} of the operator. There has been only one previous work on incremental projections 
for SVIs, namely  Wang and Bertsekas \cite{wang}. In this work  \emph{strong monotonicity} of the operator and \emph{knowledge} 
of the strong monotonicity and Lipschitz moduli were assumed. These are very demanding assumptions in practice and theory. 
\textsf{Our first objective} is to weaken such property to plain \emph{monotonicity} without requiring knowledge of the Lipschitz constant. 
\textsf{Our second objective} is to use \emph{incremental constraint projections} in \emph{distributed methods} for 
multi-agent optimization and equilibrium problems arising in networks. Such joint analysis seems to be new (to the best of our knowledge, 
all previous works in distributed methods for such problems use \emph{exact projections}). This objective is a non-trivial 
generalization of previous known distributed methods since, besides preserving the parallel computations of projections and 
the use of asynchronous agent's parameters of such methods, we wish to allow each user to project inexactly over its decision set 
in a random fashion and without additional coordination.

Assuming the structures \eqref{assump:constr:intro}-\eqref{assump:constr:intro:repres}, 
in the centralized case ($m=1$), we propose the following incremental constraint projection method:
\begin{eqnarray}
y^{k}=\Pi_{X_0}[x^k-\alpha_{k}\left(F(v^k,x^k)+\epsilon_{k} x^k\right)],\label{algo.tikhonov.approx:intro}\\
x^{k+1}=\Pi_{X_0}\left[y^k-\beta_{k}\frac{g^+_{\omega_{k}}(y^k)}{\Vert d^k\Vert^2}d^k\right],\label{algo.tikhonov.approx:eq2:intro}
\end{eqnarray}
where $\{\alpha_k,\beta_k\}$ are stepsize sequences, $\{\epsilon_k\}$ is the regularization parameter sequence, $\{v^k\}$ 
is the sample sequence, $\{\omega_k\}$ is the random control, and $d^{k}\in\partial g^+_{\omega_{k}}(y^k)\setminus-\{0\}$ 
if $g_{\omega_{k}}(y^k)>0$ and $d^k=d$ for any $d\in\re^{n}\setminus\{0\}$ otherwise. We remark that the projection 
onto $X_0$ in \eqref{algo.tikhonov.approx:intro} is dispensable if $\dom(g_i)=\re^n$ and $\{\partial g^+_i:i\in\constr\}$ 
is uniformly bounded on $\re^n$, a condition satisfied, e.g., if the soft constraints have easy computable projections, 
as commented below (see Remark \ref{rem:no:proj:X0} in Subsection \ref{s:pre:proj}). 
The above incremental algorithm advances in such a way that the ``operator step'' and the ``feasibility step'' 
are updated in separate stages. In the first stage, given the current iterate $x^k$, the method advances in 
the direction of a sample $-F(v^k,x^k)$ of the random operator, producing an auxiliary iterate $y^k$. In
this step, the hard constraint set $X_0$ is considered while the soft constraints $\{X_i:i\in\constr\}$ are 
``ignored''. In the second stage, a soft constraint $X_{\omega_k}$ is randomly chosen with $\omega_k\in\constr$, 
and the method advances in the direction opposite to a subgradient of $g^+_{\omega_k}$ at the point $y^k$, 
producing the next iterate $x^{k+1}$. Thus, the method exploits simultaneously the stochastic 
approximation of the random operator (in the first stage) and a randomization of the incremental 
selection of constraint projections (in the second stage). In Section \ref{s2}, this method is analyzed 
with no regularization, i.e., $\epsilon^k\equiv0$ and the monotone operator satisfies the \emph{weak sharpness} 
property (see Section \ref{s:weak:sharp}) while in Section \ref{s3}, we consider the same method with positive 
regularization parameters requiring just \emph{monotonicity} of the operator. 

We make some remarks to illustrate that the mentioned framework is very general. If, for $i\in\constr$, 
the Euclidean projection onto $X_i$ is easy, then we can always construct a function with ``easy" subgradients. 
Indeed, defining the function
$
g_i(x):=\dist(x,X_i),
$
for $x\in\re^n$, then $g_i$ is convex, nonnegative and finite valued over 
$\re^n$, and for any $x\notin X_i$, 
$$
\frac{x-\Pi_{X_i}(x)}{g_i(x)}=\frac{x-\Pi_{X_i}(x)}{\Vert x-\Pi_{X_i}(x)\Vert}\in\partial g_i(x),
$$ 
provides a subgradient which is easy to evaluate. Moreover, $\sup_{d\in\partial g_i(x)}\Vert d\Vert\le1$ for all $x\in\re^n$. 
In this case, using the above directions as  subgradients 
$d^k$ of $g^+_{\omega_k}$ at $y^k$, method \eqref{algo.tikhonov.approx:intro}-\eqref{algo.tikhonov.approx:eq2:intro} can be rewritten as 
\begin{eqnarray}
y^{k}&=&x^k-\alpha_k\left(F(v^k,x^k)+\epsilon_kx^k\right),\label{equation:algo:easy:proj:1}\\
x^{k+1}&=&\Pi_{X_0}\left[y^{k}-\beta_k\left(y^k-\Pi_{X_{\omega_k}}(y^k)\right)\right].\label{equation:algo:easy:proj:2}
\end{eqnarray}
If, additionally, $X_0=\re^n$ and $\beta_k\equiv1$ then the method takes the more basic form
$$
x^{k+1}=\Pi_{X_{\omega_k}}\left[x^k-\alpha_k F(v^k,x^k)\right].
$$

In Section \ref{s3}, we analyse a distributed variant. In this setting, the feasible set $X\subset\re^n$ has the form
$
X=X^1\times\cdots\times X^m,
$
where each Cartesian component $X^j\subset\re^{n_j}$ is a closed and convex set, 
$F(v,\cdot)=(F_1(v_1,\cdot),\ldots,F_m(v_m,\cdot))$ with $v=(v_1,\ldots,v_m)$, 
$F_j(v_j,\cdot):\Xi\times\re^n\rightarrow\re^{n_j}$ for $j=1,\ldots,m$ and $\sum_{j=1}^m n_j=n$. 
Moreover, we assume each Cartesian component has the constraint form
\begin{equation*}
X^j=X_0^j\cap\left(\cap_{i\in\constr_j}X_i^j\right),
\end{equation*}
where $\{X_0^j\}\cup\{X_i^j:i\in\constr_j\}$ is a collection of closed and convex subsets of $\re^{n_j}$. 
Also, for every $i\in\constr_j$, we assume $X_i^j$ is representable in $\re^{n_j}$ as 
\begin{equation*}
X_i^j=\{x\in\re^{n_j}:g_{i}(j|x)\le0\},
\end{equation*}
for some convex function $g_{i}(j|\cdot):\re^{n_j}\rightarrow\re\cup\{\infty\}$. 
We thus propose the following distributed method: for each $j=1,\ldots,m$, 
\begin{eqnarray}
y_j^{k}&=&\Pi_{X_0^j}\left[x_j^k-\alpha_{k,j}\left(F_j(v^k_j,x^k)+\epsilon_{k,j} x_j^k\right)\right],\label{algo.tikhonov.approx:cart:intro}\\
x^{k+1}_j&=&\Pi_{X_0^j}\left[y^k_j-\beta_{k,j}\frac{g^+_{\omega_{k,j}}(j|y^k_j)}{\Vert d^k_j\Vert^2}d^k_j\right],
\label{algo.tikhonov.approx:cart:eq2:intro}
\end{eqnarray}
where, for every agent $j=1,\ldots,m$, $\{\alpha_{k,j},\beta_{k,j}\}$ are stepsize sequences, $\{\epsilon_{k,j}\}$ is the 
regularization parameter sequence, $\{v^k_j\}$ is the sample sequence, $\{\omega_{k,j}\}$ is the random control 
and $d^{k}_j\in\partial g^+_{\omega_{k,j}}(j|y^k_j)\setminus\{0\}$ if $g_{\omega_{k,j}}(j|y^k_j)>0$, and $d^k_j=d$ 
for any $d\in\re^{n_j}\setminus\{0\}$ otherwise. Method \eqref{algo.tikhonov.approx:intro}-\eqref{algo.tikhonov.approx:eq2:intro} 
is the special case of \eqref{algo.tikhonov.approx:cart:intro}-\eqref{algo.tikhonov.approx:cart:eq2:intro} with $m=1$.

We mention the following contributions of methods \eqref{algo.tikhonov.approx:intro}-\eqref{algo.tikhonov.approx:eq2:intro} 
and \eqref{algo.tikhonov.approx:cart:intro}-\eqref{algo.tikhonov.approx:cart:eq2:intro}: 

\begin{itemize}
\item[(i)]\textbf{Incremental constraint projection methods for plain monotone SVIs:} In Wang and Bertsekas \cite{wang}, 
incremental constraint projection methods for SVIs were proposed assuming strong monotonicity with knowledge of the 
strong monotonicity and Lipschitz moduli. We propose a method with incremental constraint projections for SVIs requiring 
just \emph{monotonicity with no knowledge of the Lipschitz constant}, making our method much more general and applicable. 
Using explicit stepsizes, we establish almost sure asymptotic convergence, feasibility rate of $O(1/k)$ in terms of the mean 
squared distance to the feasible set and, in the case of a compact set, we provide a near optimal solvability convergence 
rate of $O\left(\frac{k^\delta\ln k}{\sqrt{k}}\right)$ in terms of the mean dual gap function of the SVI for arbitrary small $\delta>0$.
 
\item[(ii)]\textbf{Incremental constraint projections in distributed methods:} Distributed methods for SVIs have recently attained 
importance recently in the framework of optimization or equilibrium problems in networks. In this context, one important goal is 
to allow \emph{distributed computation} of projections, allow agents to update their parameters \emph{independently} and drop the 
strong or strict monotonicity property \emph{without indirect regularization} which is hard to cope with in the stochastic setting. 
The work in Koshal et al. \cite{Uday} addresses these issues but using \emph{exact projections}, and to the best of our knowledge, 
all previous works in distributed methods, even for convex optimization, seem to project exactly. Our main contribution in this context 
is to include \emph{incremental projections} in \emph{distributed methods} for SVI (and in particular for stochastic optimization). 
In this context, we allow agents to project randomly in simpler components of its own decision set without information of 
other agents' decision sets. Importantly, we preserve all properties in Koshal et al. \cite{Uday} just mentioned. The use of incremental 
projections allows easier computation of projections or flexibility when the constraints are learned via an online procedure. 
In order to achieve such contribution, we deal with a more refined convergence analysis and a new partial coordination assumption, 
not needed in the case of synchronous stepsizes or exact projections:

\begin{equation}\label{equation:add:coordination}
\sum_{k=0}^\infty\frac{(\alpha_{k,\max}-\alpha_{k,\min})^2}{\alpha_{k,\min}\epsilon_{k,\min}}<\infty,
\end{equation}
where $\alpha_{k,\max}=\max_{i=1,\ldots,m}\alpha_{k,i}$, $\alpha_{k,\min}=\min_{i=1,\ldots,m}\alpha_{k,i}$ 
and $\epsilon_{k,\min}=\min_{i=1,\ldots,m}\epsilon_{k,i}$. Using explicit asyncronous stepsizes and regularization sequences, 
we establish a.s. asymptotic convergence, feasibility rate of $O(1/k)$ in terms of the mean squared distance to the feasible set and, 
in the case of a compact feasible set, we provide a near optimal solvability convergence rate of $O\left(\frac{k^\delta\ln k}{\sqrt{k}}\right)$ 
in terms of the mean dual gap function of the SVI for arbitrary small $\delta>0$. 
The partial coordination \eqref{equation:add:coordination} appears in the rate statements as a decaying error related to the use of 
asynchronous stepsizes and asynchronous inexact random projections. To the best of our knowledge, \emph{even for the case 
of exact projections} no convergence rates have been reported for iterative distributed methods for SVIs.

\item[(iii)]\textbf{Weak sharpness property and incremental projections:} The weak sharpness property for VIs was proposed 
in \cite{marcotte}. It has been used as a sufficient condition for finite convergence of algorithms for optimization and VI problems 
in numerous works, e.g. \cite{marcotte, ferris}. To the best of our knowledge, the use of the weak sharpness property as a 
suitable property for \emph{incremental projection} methods, as analyzed in this work, has not been addressed before,  
even for VIs or optimization problems in the deterministic setting. We use an equivalent form of weak sharpness suitable for 
incremental projections. The proof of such equivalence seems to be new. Using explicit stepsizes without knowledge of the sharp-modulus, 
we prove a.s. asymptotic convergence, feasibility rate of $O(1/k)$ in terms of the mean squared distance to the feasible set and 
solvability rate of $O(1/\sqrt{k})$ (up to first order logarithmic terms) in terms of the mean distance to the solution set, for bounded or 
unbounded feasible sets. We also prove that after a finite number of iterations, any solution of a stochastic optimization problem 
with linear objective and the same feasible set as the SVI is a solution of the original SVI. We note that the weak sharpness property 
differs from strong monotonicity, allowing nonunique solutions. In that respect such analysis complements item (i) above.
\end{itemize}

The paper is organized as follows: Section \ref{s:pre} includes preliminary results  
such as tools from the  projection operator and probability, as well as required preliminaries on the weak sharpness property. 
Section \ref{s2} analyzes the method for weak sharp monotone operators. Subsection \ref{ss2.4} presents the 
correspondent complexity analysis. Section \ref{s3} deals with the regularized 
version for general monotone operators. Subsection \ref{subsection:rate:analyis:tyk} presents the correspondent 
complexity analysis. We list the assumptions in each section, along with the algorithm statements and their convergence analysis. 

\section{Preliminaries}\label{s:pre}

\subsection{Projection operator and notation}\label{s:pre:proj}
For $x,y\in\re^n$, we denote $\langle x,y\rangle$ the standard inner product and $\Vert x\Vert=\sqrt{\langle x,x\rangle}$ 
the correspondent Euclidean norm. We shall denote by $\dist(\cdot,C)$ the distance function to a general set $C$, 
namely, $\dist(x,C)=\inf\{\Vert x-y\Vert:y\in C\}$. For $X$ as in Definition \ref{SVI.def} 
we denote $\dist(x):=\dist(x,X)$. By $\cl C$ and $\mathcal{D}(C)$ we denote the closure and the diameter of the set $C$, respectively. 
For a closed and convex set $C\subset\mathbb{R}^n$, we denote by $\Pi_C$ the orthogonal projection onto $C$. 
For a function $g:\re^n\rightarrow\re^n$ we denote by $g^+$ its positive part, defined by $g^+(x)=\max\{0,g(x)\}$ for $x\in\re^n$. 
If $g$ is convex, we denote by $\partial g$ its subdifferential and $\dom(g)$ its domain. 

The following properties of the projection operator are well known; see e.g. Facchinei 
and Pang \cite{facchinei} and Auslender and Teboulle \cite{auslender:teboulle}.  
\begin{lemma}\label{proj}
Take a closed and convex set $C\subset\mathbb{R}^n$. Then
\begin{itemize}
\item[i)] For all $x\in\mathbb{R}^n, y\in C$,
$\langle x-\Pi_C(x),y-\Pi_C(x)\rangle\le0.$
\item[ii)]For all $x,y\in\mathbb{R}^n$,
$
\Vert \Pi_C(x)-\Pi_C(y)\Vert\le\Vert x-y\Vert.
$
\item[iii)] Let $z\in\re^n$, $x\in C$ with $y:=\Pi_C[x-z]$. Then for all $u\in C$,
$$
2\langle z,y-u\rangle\le\Vert x-u\Vert^2-\Vert y-u\Vert^2-\Vert y-x\Vert^2.
$$
\end{itemize}
\end{lemma}
The following lemma will be used in the analysis of the methods of Sections \ref{s2} and \ref{s3}. 
It is  proved in Nedi\'c \cite{nedich} and Polyak \cite{polyak.2}, but in a slightly different form, 
suitable for convex optimization problems. The changes required for the case of monotone variational inequalities are straightforward.

\begin{lemma}\label{lemma:feas:step}
Consider a closed and convex 
$X_0\subset\re^n$,  
and let $g:\re^n\to\re\cup\{\infty\}$ be a convex function with $\dom(g)\subset X_0$. Suppose that 
there exists $C_g>0$ such that 
$\Vert z\Vert\le C_g$ 
for all $x\in X_0$ and all $z\in\partial g^+(x)$. 
Take $x_1\in X_0$, $u\in\re^n$, $\alpha>0$, $\beta\in(0,2)$ and define $y,x_2\in X_0$ as
\begin{eqnarray*}
y&=&\Pi_{X_0}[x_1-\alpha u],\\
x_2&=&\Pi_{X_0}\left[y-\beta\frac{g^+(y)}{\Vert d\Vert^2}d\right],
\end{eqnarray*}  
where $d\in\re^n-\{0\}$ is such that $d\in\partial g^+(y)-\{0\}$ if $g^+(y)>0$. Then for any $x_0\in X_0$ such that $g^+(x_0)=0$, 
and any $\tau>0$, it holds that
$$
\Vert x_2-x_0\Vert^2\le\Vert x_1-x_0\Vert^2-2\alpha\langle x_1-x_0,u \rangle
+\left[1+\tau\beta(2-\beta)\right]\alpha^2\Vert u\Vert^2-
\frac{\beta(2-\beta)}{C_g^2}\left(1-\frac{1}{\tau}\right)\left(g^+(x_1)\right)^2.
$$  
\end{lemma}

\begin{rem}\label{rem:no:proj:X0}
We remark that if $\dom(g)=\re^n$ and the subgradients of $g^+$ are uniformly bounded over $\re^n$, 
then the result of Lemma \ref{lemma:feas:step} holds with $y\in\re^n$ given as
$
y=x_1-\alpha u,
$
instead of $y=\Pi_{X_0}\left[x_1-\alpha u\right]$.
\end{rem}

The abbreviation ``a.s.'' means ``almost surely'' and the abbreviation ``i.i.d.'' means ``independent and 
identically distributed''. Given sequences $\{x^k\}$ and $\{y^k\}$, the notation $x^k=O(y^k)$ or $x^k\lesssim y^k$
means that there  exists $C>0$, such that $\Vert x^k\Vert\le C\Vert y^k\Vert$ for all $k$. 
The notation $x^k\sim y^k$ means $x^k\lesssim y^k$ and $y^k\lesssim x^k$. Given a $\sigma$-algebra $\alg$ and a random variable $\xi$, 
we denote by $\esp[\xi]$ and $\esp[\xi|\alg]$ 
the expectation and conditional expectation, respectively. Also, we write $\xi\in\alg$ for ``$\xi$ is $\alg$-measurable''. 
$\sigma(\xi_1,\ldots,\xi_n)$ indicates the $\sigma$-algebra generated by the random variables $\xi_1,\ldots,\xi_n$. 
$\mathbb{N}_0$ denotes the set of natural numbers including zero. For $m\in\mathbb{N}$, we use the notation 
$[m]:=\{1,\ldots,m\}$. For $r\in\mathbb{R}$, $\lceil r\rceil$ denotes the smallest integer greater than $r$. 
We denote by $\re^m_{>0}$ the interior of the 
nonnegative orthant $\re^m_+$. 

\subsection{Probabilistic tools}
As in other stochastic approximation methods, a fundamental tool to be used is the following Convergence Theorem of Robbins and Siegmund \cite{rob},
which can be seen as the stochastic version of the properties of quasi-Fej\'er convergent sequences.

\begin{theorem}\label{rob}
Let $\{y_k\},\{u_k\}, \{a_k\}, \{b_k\}$ be sequences of non negative random variables, 
adapted to the filtration $\{\alg_k\}$, such that a.s. $\sum a_k<\infty$, $\sum b_k<\infty$ and for all $k\in\mathbb{N}$,
$
\esp\big[y_{k+1}\big| \alg_k\big]\le(1+a_k)y_k-u_k+b_k.
$
Then a.s. $\{y_k\}$ converges and $\sum u_k<\infty$.
\end{theorem}

We will also use the following result, whose proof can be found in 
Lemma 10 of Polyak \cite{polyak}.
\begin{theorem}
\label{rob2}
Let $\{y_k\}, \{a_k\}, \{b_k\}$ be sequences of nonnegative random variables, adapted to the filtration $\{\alg_k\}$, 
such that a.s. $a_k\in[0,1]$, $\sum a_k=\infty$, $\sum b_k<\infty$, $\lim_{k\rightarrow\infty}\frac{b_k}{a_k}=0$ and for all $k\in\mathbb{N}$,
$
\esp\big[y_{k+1}\big| \alg_k\big]\le(1-a_k)y_k+b_k.
$
Then a.s. $\{y_k\}$ converges to zero.
\end{theorem}

\subsection{Weak sharpness}\label{s:weak:sharp}

We briefly discuss the \emph{weak sharpness} property of variational inequalities. 
For $X\subset\re^n$ and $x\in X$, $\polar_X(x)$ denotes
the normal cone of $X$ at $x$, given by 
\begin{equation*}\label{polar.cone}
\polar_X(x)=\{v\in\re^n:\langle v,y-x\rangle\le0,\forall y\in X\},
\end{equation*}
The tangent cone of $X$ at $x\in X$ is defined as 
\begin{equation}\label{tangent:cone:def}
\tang_{X}(x)=\{d\in\re^n:\exists t_k>0,\exists d^k\in\re^n,\forall k\in\mathbb{N}, x+t_kd^k\in X, d^k\rightarrow d\}.
\end{equation}
For a closed and convex set $X$, the tangent cone at a point $x\in X$ 
has the following alternative representations (see Rockafellar and Wets \cite{terry&wets}, Proposition 6.9 and Corollary 6.30):
\begin{equation}\label{tangent:cone:charac}
\tang_{X}(x)=\cl\{\alpha(y-x):\alpha>0,y\in X\}=[\polar_X(x)]^\circ,
\end{equation}
where for a given set $Y\subset\re^n$, the polar set $Y^\circ$ is defined as 
$
Y^\circ=\{v\in\re^n:\langle v,y\rangle\le0,\forall y\in Y\}.
$

In Burke and Ferris \cite{ferris}, the notion of \emph{weak sharp} minima 
for the problem $\min_{x\in X}f(x)$ with solution set $X^*$ 
was introduced: 
there exists $\rho>0$ such that 
\begin{equation}\label{weak.sharp.minima}
f(x)-f^*\ge\rho\dist(x,X^*),
\end{equation}
for all $x\in X$, where $f^*$ is the minimum value of $f$ at $X$. Relation \eqref{weak.sharp.minima} 
means that $f-f^*$ gives an error bound on the solution set $X^*$. 
In Burke and Ferris \cite{ferris}, it is proved that if $f$ is a closed, proper, and differentiable convex function and if 
the sets $X$ and $X^*$ are nonempty, closed, and convex, then \eqref{weak.sharp.minima} is equivalent to the following geometric condition: 
for all $x^*\in X^*$, 
\begin{equation}\label{weak.sharp.geom.min}
-\nabla f(x^*)\in\interior\Bigg(\bigcap_{x\in X^*}[\tang_{X}(x)\cap\polar_{X^*}(x)]^\circ\Bigg).
\end{equation}

In optimization problems, the objective function can be used for determining regularity of solutions. In variational inequalities  
one can use for that purpose the above geometric definition or exploit the use of gap functions associated to the VI. 
The dual gap function $\mathsf{G}:\re^n\to\re\cup\{\infty\}$ is defined as
\begin{equation}\label{def:gap:function}
\mathsf{G}(x):=\sup_{y\in X}\langle T(y),x-y\rangle.
\end{equation}
In the sequel, we denote by $B(0,1)$ the unit ball in $\re^n$ and by $X^*$ the solution set of VI$(T,X)$. In order to define a 
meaningful notion of weak sharpness for VIs, 
the following statements
were considered 
in Marcotte and Zhu \cite{marcotte}: 
	
\begin{itemize}
\item[(i)] There exists $\rho>0$, such that for all $x^*\in X^*$,
\begin{equation}\label{weak.sharp.geom2}
-T(x^*)+\rho B(0,1)\in\bigcap_{x\in X^*}[\tang_{X}(x)\cap\polar_{X^*}(x)]^\circ.
\end{equation}

\item[(ii)] There exists $\rho>0$, such that for all $x^*\in X^*$,
\begin{equation}\label{weak.sharp.aux}
\langle T(x^*),z\rangle\ge\rho\Vert z\Vert, \forall z\in \tang_{X}(x^*)\cap\polar_{X^*}(x^*).
\end{equation}

\item[(iii)] For all $x^*\in X^*$,
\begin{equation}\label{weak.sharp.geom}
-T(x^*)\in\interior\Bigg(\bigcap_{x\in X^*}[\tang_{X}(x)\cap\polar_{X^*}(x)]^\circ\Bigg).
\end{equation}

\item[(iv)] There exist $\rho>0$ such that for all $x\in X$,
\begin{equation}\label{weak.sharp.gap}
\mathsf{G}(x)\ge\rho\dist(x,X^*).
\end{equation}
\end{itemize}

Statement (iii) is the definition of a weak sharp VI$(T,X)$  
given in Marcotte and Zhu \cite{marcotte}. 
In Theorem 4.1 of 
Marcotte and Zhu \cite{marcotte}, 
it was proved that (i)-(ii) are equivalent, 
and that (i)-(iv) are equivalent 
when $X$ is compact  and
$T$ is paramonotone (also known as monotone$^+$)
i.e., $T$ is motonone and 
$
\langle T(x)-T(y),x-y\rangle=0\Rightarrow T(x)=T(y),
$
for all $x,y\in\re^n$
(see Iusem \cite{iusem1} for other properties of paramonotone operators).
 
Relation \eqref{weak.sharp.gap} means that the gap function  $G$ provides an error bound on the solution set $X^*$. Paramonotonicity 
implies that $T$ is constant on the solution set $X^*$. Important classes of paramonotone operators are, for example, co-coercive, 
symmetric monotone and strictly monotone composite operators (see Facchinei and Pang \cite{facchinei}, Chapter 2).

Recently, 
the following assumption was introduced
in Yousefian et al. \cite{uday2}: 
there  exists $\rho>0$ such that for all $x^*\in X^*$ and  all $x\in X$,
\begin{equation}\label{weakly.sharp1}
\langle T(x^*),x-x^*\rangle\ge\rho\dist(x,X^*).
\end{equation}
Clearly, \eqref{weakly.sharp1} implies \eqref{weak.sharp.gap}. 
We show next that  \eqref{weakly.sharp1} implies \eqref{weak.sharp.aux} and the converse statement holds when $T$ is constant on $X^*$.
Thus, when $T$ is constant on $X^*$, \eqref{weak.sharp.geom2}, \eqref{weak.sharp.aux} and \eqref{weakly.sharp1} are equivalent, 
and when $T$ is paramonotone and $X$ is compact, conditions \eqref{weak.sharp.geom2}-\eqref{weakly.sharp1} are all equivalent. 
Hence, the following proposition, which appears to be new and is proved in the Appendix, gives a precise relation between 
property \eqref{weakly.sharp1} with the previous notions of weak sharpness \eqref{weak.sharp.geom2}-\eqref{weak.sharp.gap} 
presented in Marcotte and Zhu \cite{marcotte}. Property \eqref{weakly.sharp1} is well suited for the incremental 
constraint projection-type methods considered here.

\begin{proposition}\label{weak.sharp.equivalence}
Let $T:\re^n\rightarrow\re^n$ be a continuous monotone operator and $X\subset\re^n$ a closed and convex set. The following holds:
\begin{itemize}
\item[i)] Condition \eqref{weakly.sharp1} implies \eqref{weak.sharp.aux}.
\item[ii)] If $T$ is constant on $X^*$, then \eqref{weak.sharp.aux} implies \eqref{weakly.sharp1}.
\end{itemize}
\end{proposition}

Finally, we will use the following result in Theorem 4.2. of Marcotte and Zhu \cite{marcotte}:
\begin{theorem}\label{marcotte.thm4.2}
If $T$ is continuous and there exists $z\in\re^n$ such that
$-z\in\interior\left(\bigcap_{x\in X^*}[\tang_{X}(x)\cap\polar_{X^*}(x)]^\circ\right)$,
then 
$\argmin_{x\in X}\langle z,x\rangle\subset X^*$.
\end{theorem}

As a consequence of Theorem \ref{marcotte.thm4.2}
under weak sharpness and uniform continuity of $T$, any algorithm which generates a sequence 
$\{x^k\}$ such that $\dist(x^k,X^*)\rightarrow0$ has the property that after a finite number of iterations $M$, 
any solution of the auxiliary program
$
\min_{x\in X}\langle T(x^M),x\rangle,
$ 
with a linear objective, is a solution of the original variational inequality (see Theorem 5.1 in Marcotte and Zhu \cite{marcotte}). 
When $X$ is a polyhedron,  this result can be interpreted as a finite convergence property of algorithms for VI with 
the weak sharpness property, since a linear program is finitely solvable. 
Other algorithmic implications of weak sharpness are developed in Marcotte and Zhu \cite{marcotte}. 

\section{An incremental projection method under weak sharpness}\label{s2}

In the following section we assume that the feasible set has the form
\begin{equation}\label{assump:constraint}
X=X_0\cap\left(\cap_{i\in\constr}X_i\right),
\end{equation} 
where $\{X_0\}\cup\{X_i:i\in\constr\}$ is a collection of closed and convex subsets of $\re^n$. 
We assume that the evaluation of the projection onto $X_0$ is computationally easy and that for all $i\in\constr$, 
$X_i$ is representable as 
\begin{equation}\label{assump:represent}
X_i=\{x\in\re^n:g_i(x)\le0\},
\end{equation}
for some convex function $g_i$ with $\dom (g_i)\subset X_0$. 
Also we assume that, for every $i\in\constr$, subgradients of $g^+_i(x)$ at points $x\in X_0-X_i$ are easily computable 
and that $\{\partial g_i^+:i\in\constr\}$ is uniformly bounded over $X_0$, that is, there exists $C_g>0$ such that 
\begin{equation}\label{assump:subgrad:bound}
\Vert d\Vert\le C_g\,\,\, 
\forall x\in X_0, \forall i\in\constr,\forall d\in\partial g_i^+(x). 
\end{equation}

\subsection{Statement of the algorithm}\label{ss2.1}
Next we formally state the algorithm.

\begin{algo}[Incremental constraint projection method]
\quad
\begin{enumerate}
\item{\bf Initialization:} Choose the initial iterate $x^0\in\mathbb{R}^n$, the stepsizes  
$\{\alpha_k\}$ and $\{\beta_k\}$, the random controls $\{\omega_k\}$ and the operator samples $\{v^k\}$.
\item{\bf Iterative step:} Given $x^k$, define:
\begin{eqnarray}
y^{k}&=&\Pi_{X_0}[x^k-\alpha_kF(v^k,x^k)],\label{algo.extra.approx.eq1}\\
x^{k+1}&=&\Pi_{X_0}\left[y^{k}-\beta_k\frac{g^+_{\omega_k}(y^{k})}{\Vert d^{k}\Vert^2}d^{k}\right],
\label{algo.extra.approx.eq2}
\end{eqnarray}
where $d^{k}\in\partial g^+_{\omega_k}(y^{k})-\{0\}$ if $g^+_{\omega_k}(y^k)>0$; $d^k=d\in\re^{n}-\{0\}$ 
if $g^+_{\omega_k}(y^k)=0$.
\end{enumerate}
\end{algo}

\subsection{Discussion of the assumptions}\label{ss2.2}

In the sequel we consider the natural filtration
$$
\alg_k=\sigma(\omega_0,\ldots,\omega_{k-1},v^0,\ldots,v^{k-1}).
$$ 
Next we present the assumptions necessary for our convergence analysis.

\begin{assump}[Consistency]\label{existence}
The solution set $X^*$ of VI$(T,X)$ is nonempty.
\end{assump}

\begin{assump}[Monotonicity]\label{monotonicity}
The mean operator $T$ in \eqref{expected} satisfies: for all $y,x\in\re^n$,
$$
\langle T(y)-T(x),y-x\rangle\ge 0.
$$
\end{assump}

\begin{assump}[Lipschitz-continuity or boundedness]\label{lipschitz}
We suppose $T:\re^n\rightarrow\re^n$ is continuous and, at least, one of the following assumptions hold:
\begin{itemize}
\item[(i)] There exists measurable $L(v):\Omega\rightarrow\re_+$ with finite second moment, such that a.s. for all $y,x\in \re^n$,
$$
\Vert F(v,y)-F(v,x)\Vert\le L(v)\Vert y-x\Vert.
$$
We denote $L:=\sqrt{\esp[L(v)^2]}$.
\item[(ii)] There exists $C_F>0$ such that
$$
\sup_{x\in X_0}\esp\left[\Vert F(v,x)\Vert^2\right]\le 2C_F^2.
$$
\end{itemize}
\end{assump}

Item (i) implies in particular that $T$ is $L$-Lipschitz continuous. Both items (i) or (ii) are standard in stochastic optimization. Let 
\begin{equation}\label{def:sigma}
\sigma(x)^2:=\esp\big[\Vert F(v,x)-T(x)\Vert^2\big]
\end{equation}
denote the variance of $F(v,x)$ for $x\in\re^n$. Both item (ii) and \eqref{noise.bound2} imply that the variance 
function $\sigma(\cdot)^2$ is bounded above \emph{uniformly} over $X$. Item (i) is a weaker assumption since it only requires the 
map $\sigma(\cdot)^2$ to be \emph{finite} at every point in $X$ (allowing $X$ to be unbounded). Except for Wang and Bertsekas \cite{wang} 
in the strongly monotone case, conditions in item (ii) or in \eqref{noise.bound2} were  requested in all the previous 
literature on SA methods 
for SVI or stochastic optimization. Under Assumption \ref{lipschitz}(i), we do not require \eqref{noise.bound2}.

\begin{assump}[IID sampling]\label{unbiased}
The sequence $\{v^k\}$ is an independent identically distributed sample sequence of $v$.
\end{assump}

The above assumption implies in particular that a.s. for all $x\in\mathbb{R}^n$ and all $k\in\mathbb{N}$,
$
\esp\big[F(v^k,x)\big|\alg_k\big]=T(x).
$
We now state the assumptions concerning the incremental projections. 
\begin{assump}[Constraint sampling and regularity]\label{assump:constraint:reg}
There exists $c>0$ such that a.s. for all $x\in X_0$ and all $k\in\mathbb{N}_0$,
$$
\dist(x,X)^2\le c\cdot\esp\left\{\left[g_{\omega_k}^+(x)\right]^2\Big|\alg_k\right\}.
$$
\end{assump}

Assumption \ref{assump:constraint:reg} is very general and it was assumed in Nedi\'c \cite{nedich}. For completeness we present next a 
lemma showing Assumption \ref{assump:constraint:reg} holds in the relevant case in which the feasible set $X$ satisfies a 
standard metric regularity property, the number $|\mathcal{I}|$ of constraints is finite (and possibly very large) and an i.i.d. 
uniform sampling of the constraints is chosen. 

\begin{lemma}[Sufficient condition for Assumption \ref{assump:constraint:reg}]\label{lemma:sufficience:reg}
Suppose $\{v^k\}$ and $\{\omega_k\}$ are independent sequences, $|\mathcal{I}|<\infty$ and the following hold:
\begin{itemize}
\item[(i)] The sequence $\{\omega_k\}$ is an i.i.d. sample of a random variable $\omega$ 
taking values on $\mathcal{I}$ such that for some $\lambda>0$,
$$
\prob\left(\omega=i\right)\ge\frac{\lambda}{|\constr|},\quad\forall i\in\constr,
$$
\item[(ii)] The set $X$ is \emph{metric regular}: there is $\eta>0$ such that for all $x\in X_0$,
$$
\dist(x,X)^2\le \eta\max_{i\in\constr}[g_i^+(x)]^2.
$$
\end{itemize}
Then Assumption \ref{assump:constraint:reg} holds with $c=\frac{\eta|\constr|}{\lambda}$.
\end{lemma}
\begin{proof}
Since $\{v^k\}$ and $\{\omega_k\}$ are independent and the $\{\omega_k\}$'s are i.i.d., 
we have that for all $k\in\mathbb{N}_0$, $\omega_k$ is independent of $\alg_k$. Hence for all $k\in\mathbb{N}_0$ and $x\in X$,
\begin{eqnarray*}
\esp\left\{\left[g_{\omega_k}^+(x)\right]^2\big|\alg_k\right\}&=&\esp\left\{\left[g_{\omega_k}^+(x)\right]^2\right\}\\
&=&\esp\left\{\left[g_{\omega}^+(x)\right)^2\right\}\\
&=&\sum_{i\in\constr}\left[g_{i}^+(x)\right]^2\prob(\omega=i)\\
&\ge &\frac{\lambda}{|\constr|}\sum_{i\in\constr}\left[g_{i}^+(x)\right]^2\\
&\ge &\frac{\lambda}{|\constr|}\max_{i\in\constr}\left[g_{i}^+(x)\right]^2\\
&\ge &\frac{\lambda}{\eta|\constr|}\dist(x,X)^2,
\end{eqnarray*}
using  the fact that $\omega_k$ has the same distribution as $\omega$ in the second equality, 
Lemma \ref{lemma:sufficience:reg}(i) in the first inequality and Lemma \ref{lemma:sufficience:reg}(ii)
in the last inequality. 
\end{proof} 
 
Item (i) above is satisfied when $\omega$ is uniform over $\constr$, i.e., $\prob(\omega=i)=1/|\constr|$ for all $i\in\constr$. 
As an example, item (ii) in Lemma \ref{lemma:sufficience:reg} is satisfied for any compact convex set under a Slater condition, 
as proved by Robinson (see Pang \cite{pang}). 
A particular case of item (ii) occurs when for some $\eta>0$ and all $x\in\re^n$,
\begin{equation}\label{assump:linear:regularity}
\dist(x,X)^2\le\eta\max_{i\in\constr}\dist(x,X_i)^2.
\end{equation}
In this case $g_i:=\dist(\cdot,X_i)$ for $i\in\constr$ and the method \eqref{algo.extra.approx.eq1}-\eqref{algo.extra.approx.eq2} 
may be rewritten as \eqref{equation:algo:easy:proj:1}-\eqref{equation:algo:easy:proj:2} assuming easy projections onto the 
soft constraints. Condition \eqref{assump:linear:regularity} is called \emph{linear regularity}; see Bauschke and Borwein \cite{bau.bor}, 
Deutsch and Hundal \cite{deu.hun}. As proved by Hoffman, \eqref{assump:linear:regularity} is satisfied for any polyhedron (see Pang \cite{pang}).
 
\begin{assump}[Small stepsizes]\label{approx.step}
For all $k\in\mathbb{N}$, 
$\alpha_k>0$, $\beta_k\in(0,2)$, 
and
$$
\sum_{k=0}^\infty\alpha_k=\infty,\quad
\sum_{k=0}^\infty\alpha_k^2<\infty,\quad
\sum_{k=0}^\infty\frac{\alpha_k^2}{\beta_k(2-\beta_k)}<\infty.
$$
\end{assump}

We remark here that the use of small stepsizes is forced by 
two factors: the use of approximate projections instead of exact ones, 
and the stochastic approximation. Indeed, even with \emph{exact projections}, 
the method \eqref{algo.extra.approx.eq1}-\eqref{algo.extra.approx.eq2} still requires small stepsizes
in order to guarantee asymptotic convergence. 

Finally we state the weak-sharpness property assumed only in this section.
\begin{assump}[weak sharpness]\label{weakly.sharp.hyphotesis}
There exists $\rho>0$, such that for all $x^*\in X^*$ and all $x\in X$,
\begin{equation}\label{weakly.sharp.assumption}
\langle T(x^*),x-x^*\rangle\ge\rho\dist(x,X^*),
\end{equation}
\end{assump}

\subsection{Convergence analysis}

We need the following lemma whose proof is immediate.
\begin{lemma}\label{l1}
Suppose that Assumptions \ref{lipschitz}(i)-\ref{unbiased} hold. Define the function $B:\re^n\rightarrow[0,\infty)$ as
\begin{eqnarray}\label{def:B}
B(x):=\sqrt{\esp\big[\Vert F(v,x)\Vert^2\big]},
\end{eqnarray}
for any $x\in\re^n$. Then, almost surely, for all $x,y\in\mathbb{R}^n$, $k\in\mathbb{N}$,
$$
\Vert T(x)\Vert^2\le B(x)^2\le 2L^2\Vert x-y\Vert^2+2B(y)^2.
$$
\end{lemma}
 
We now prove an iterative relation to be used in the convergence analysis. 
We mention that \eqref{lemma:recursion1} is sufficient for the convergence analysis and includes the case of unbounded $X$ and $T$. 
If the operator is bounded or $X_0$ is compact, then \eqref{lemma:recursion2}  allows an improvement  of 
the convergence rate given in Section \ref{ss2.4}. In the following we define for all $x\in\re^n$, $k\in\mathbb{N}$ and $\tau>1$,
\begin{equation}\label{def:Bk:Ak:C}
\mathsf{B}_{k}:=\beta_k(2-\beta_k),\quad
\mathsf{A}_{k,\tau}:=\mathsf{B}_k(\tau-1)/(c C_g^2\tau),\quad
\mathsf{C}(x):=\rho+B(x).
\end{equation}

\begin{lemma}[Recursive relations]\label{lema:recursion}
Suppose that Assumptions \ref{existence}-\ref{weakly.sharp.hyphotesis} hold.

If Assumption \ref{lipschitz}(i) holds, then for all $x^*\in X^*$, $\tau>1$ and $k\in\mathbb{N}$,
\begin{eqnarray}
\esp\big[\Vert x^{k+1}-x^*\Vert^2\big|\alg_k\big]&\le & 
\left[1+2\left(1+\mathsf{B}_{k}\tau\right)L^2\alpha_k^2\right]\Vert x^{k}-x^*\Vert^2-2\rho\alpha_k\dist(x^k,X^*)\nonumber\\
&+&\left[\frac{\mathsf{C}(x^*)^2}{\mathsf{A}_{k,\tau}}+2\left(1+\mathsf{B}_{k}\tau\right)B(x^*)^2\right]\alpha_k^2,\label{lemma:recursion1}
\end{eqnarray}
and 
\begin{eqnarray}
\esp\big[\Vert x^{k+1}-x^*\Vert^2\big|\alg_k\big] &\le & 
\left[1+2\left(1+\mathsf{B}_{k}\tau\right)L^2\alpha_k^2\right]\Vert x^{k}-x^*\Vert^2-\frac{\mathsf{A}_{k,\tau}}{2}\dist(x^k,X)^2\nonumber\\
&+&\left[\frac{2\mathsf{C}(x^*)^2}{\mathsf{A}_{k,\tau}}+2\left(1+\mathsf{B}_{k}\tau\right)B(x^*)^2\right]\alpha_k^2.\label{lemma:recursion1:feas}
\end{eqnarray}

If Assumption \ref{lipschitz}(ii) holds, then for all $x^*\in X^*$, $\tau>1$ and $k\in\mathbb{N}$,
\begin{equation}\label{lemma:recursion2}
\esp\big[\dist(x^{k+1},X^*)^2\big|\alg_k\big]\le 
\dist(x^k,X^*)^2-2\rho\alpha_k\dist(x^k,X^*)
+\left[\frac{\left(\rho+\sqrt{2C_F}\right)^2}{\mathsf{A}_{k,\tau}}+2\left(1+\mathsf{B}_{k}\tau\right)C_F^2\right]\alpha_k^2,
\end{equation}
and 
\begin{equation}\label{lemma:recursion2:feas}
\esp\big[\dist(x^{k+1},X^*)^2\big|\alg_k\big]\le 
\dist(x^k,X^*)^2-\frac{\mathsf{A}_{k,\tau}}{2}\dist(x^k,X)^2
+\left[\frac{2\left(\rho+\sqrt{2C_F}\right)^2}{\mathsf{A}_{k,\tau}}+2\left(1+\mathsf{B}_{k}\tau\right)C_F^2\right]\alpha_k^2.
\end{equation}
\end{lemma}
\begin{proof}
Take $x^*\in X^*$, $\tau>1$ and $k\in\mathbb{N}$. We claim that 
$$	
\Vert x^{k+1}-x^*\Vert^2\le\Vert x^k-x^*\Vert^2-2\alpha_k\langle x^k-x^*,F(v^k,x^k)\rangle+
$$
\begin{equation}\label{lemma:recursion:eq1}
\left[1+\tau\beta_k(2-\beta_k)\right]\alpha_k^2\Vert F(v^k,x^k)\Vert^2-
\frac{\beta_k(2-\beta_k)}{C_g^2}\left(1-\frac{1}{\tau}\right)\left(g^+_{\omega_k}(x^k)\right)^2.
\end{equation}	
Indeed, by the  definition of the method \eqref{algo.extra.approx.eq1}-\eqref{algo.extra.approx.eq2}, 
we can invoke Lemma \ref{lemma:feas:step} with $g:=g_{\omega_k}$, 
$x_1:=x^k$, $x_2:=x^{k+1}$, $y:=y^k$, $x_0:=x^*$, $\alpha:=\alpha_k$, $u:=F(v^k,x^k)$, $\beta:=\beta_k$ and $d:=d^k$, 
obtaining \eqref{lemma:recursion:eq1}. 

We now take the conditional expectation with respect to $\alg_k$ in \eqref{lemma:recursion:eq1} obtaining,
\begin{eqnarray}	
\esp\left[\Vert x^{k+1}-x^*\Vert^2\big|\alg_k\right]&\le &\Vert x^k-x^*\Vert^2-2\alpha_k\langle x^k-x^*,T(x^k)\rangle+\left[1+\tau\beta_k(2-\beta_k)\right]\alpha_k^2\esp\left[\Vert F(v^k,x^k)\Vert^2\big|\alg_k\right]\nonumber\\
&&-\frac{\beta_k(2-\beta_k)}{C_g^2}\left(1-\frac{1}{\tau}\right)\esp\left[\left(g^+_{\omega_k}(x^k)\right)^2\big|\alg_k\right]\nonumber\\
&\le &\Vert x^k-x^*\Vert^2+2\alpha_k\langle x^*-x^k,T(x^k)\rangle+\left[1+\tau\beta_k(2-\beta_k)\right]\alpha_k^2\esp\left[\Vert F(v^k,x^k)\Vert^2\big|\alg_k\right]\nonumber\\
&&-\frac{\beta_k(2-\beta_k)}{c\cdot C_g^2}\left(1-\frac{1}{\tau}\right)\dist(x^k,X)^2,\label{lemma:recursion:eq2}
\end{eqnarray}
using $x^k\in\alg_k$  and Assumption \ref{unbiased} in the first inequality, 
and Assumption \ref{assump:constraint:reg} in the second inequality.

Next, we will bound the second term in the right hand side of \eqref{lemma:recursion:eq2}. We write
\begin{equation}\label{eq8}
\langle T(x^k),x^*-x^k\rangle=\langle T(x^k)-T(x^*),x^*-x^k\rangle+
\langle T(x^*),x^*-\Pi_X(x^k)\rangle+
\langle T(x^*),\Pi_X(x^k)-x^k\rangle.
\end{equation}
By monotonicity of $T$ (Assumption \ref{monotonicity}), the first term in the right hand side of \eqref{eq8} satisfies
\begin{equation}\label{term1}
\langle T(x^k)-T(x^*),x^*-x^k\rangle\le0.
\end{equation}
Regarding the second term in the right hand side of \eqref{eq8},  
the weak sharpness property (Assumption \ref{weakly.sharp.hyphotesis}) and the fact that $x\in X^*$ imply
\begin{equation}\label{term2}
\langle T(x^*),x^*-\Pi_X(x^k)\rangle
\le-\rho\dist\left(\Pi_X(x^k),X^*\right).
\end{equation}
We now observe that 
$
\left|\dist\left(\Pi_X(x^k),X^*\right)-\dist(x^k,X^*)\right|\le\Vert \Pi_X(x^k)-x^k\Vert=\dist(x^k,X),
$ 
so that
\begin{equation}\label{assump:recursion:eq3}
\dist\left(\Pi_X(x^k),X^*\right)\ge\dist(x^k,X^*)-\dist(x^k,X).
\end{equation}
From \eqref{term2}-\eqref{assump:recursion:eq3}, we get 
\begin{equation}\label{term2.2}
\langle T(x^*),x^*-\Pi_X(x^k)\rangle
\le -\rho\dist(x^k,X^*)+\rho\dist(x^k,X).
\end{equation}
Concerning the third term in the right hand side of \eqref{eq8}, we have
\begin{equation}\label{term3}
\langle T(x^*),\Pi_X(x^k)-x^k\rangle
\le\Vert T(x^*)\Vert\Vert\Pi_X(x^k)-x^k\Vert
\le B(x^*)\dist(x^k,X),
\end{equation}
using Cauchy-Schwarz inequality in the first inequality, and the definition of $B(x^*)$ in Lemma \ref{l1} in the second inequality. 
Combining \eqref{term1}, \eqref{term2.2} and \eqref{term3} with \eqref{eq8}, we finally get
\begin{equation}\label{eq9}
\langle T(x^k),x^*-x^k\rangle\le-\rho\dist(x^k,X^*)+\left(\rho+B(x^*)\right)\dist(x^k,X).
\end{equation}

We use \eqref{eq9} in \eqref{lemma:recursion:eq2} and get
\begin{eqnarray}
\esp\big[\Vert x^{k+1}-x^*\Vert^2\big|\alg_k\big]&\le &\Vert x^{k}-x^*\Vert^2-2\rho\alpha_k\dist(x^k,X^*)
+\left[1+\tau\beta_k(2-\beta_k)\right]\alpha_k^2\esp\left[\Vert F(v^k,x^k)\Vert^2\big|\alg_k\right]\nonumber\\
&&-\frac{\beta_k(2-\beta_k)}{c\cdot C_g^2}\left(1-\frac{1}{\tau}\right)\dist(x^k,X)^2
+2(\rho+B(x^*))\alpha_k\dist(x^k,X).\label{eq10}
\end{eqnarray}

From Lemma \ref{l1} and the fact that $x^k\in\alg_k$, we obtain
\begin{equation}\label{lemma:recursion:eq4}
\esp\big[\Vert F(v^k,x^k)\Vert^2\big|\alg_k\big]\le 2L^2\Vert x^k-x^*\Vert^2+2B(x^*)^2.
\end{equation}

Now we rearrange the last two terms in the right hand side of \eqref{eq10}, using the fact that
$2ab\le\lambda a^2+\frac{b^2}{\lambda}$ for any $\lambda>0$. With $a:=\dist(x^k,X)$, $b:=\mathsf{C}(x^*)\alpha_k$ 
and $\lambda:=\mathsf{A}_{k,\tau}$ we get
\begin{equation}\label{eee1}
-\mathsf{A}_{k,\tau}\dist(x^k,X)^2+2\mathsf{C}(x^*)\alpha_k\dist(x^k,X)
\le\frac{\mathsf{C}(x^*)^2\alpha_k^2}{\mathsf{A}_{k,\tau}}.
\end{equation}
Putting together relations \eqref{eq10}-\eqref{eee1} and rearranging terms, 
we finally get \eqref{lemma:recursion1}, as requested. 

Alternatively, we can replace \eqref{eee1} by the bound
\begin{equation}\label{eee1:feas}
-\mathsf{A}_{k,\tau}\dist(x^k,X)^2+2\mathsf{C}(x^*)\alpha_k\dist(x^k,X)
\le-\frac{\mathsf{A}_{k,\tau}}{2}\dist(x^k,X)^2+\frac{2\mathsf{C}(x^*)^2\alpha_k^2}{\mathsf{A}_{k,\tau}},
\end{equation}
using the fact that $2ab\le\lambda a^2+\frac{b^2}{\lambda}$ with $a:=\dist(x^k,X)$, $b:=\mathsf{C}(x^*)\alpha_k$ 
and $\lambda:=\mathsf{A}_{k,\tau}/2$. Putting together relations \eqref{eq10}-\eqref{lemma:recursion:eq4} 
and \eqref{eee1:feas} and rearranging terms, we get \eqref{lemma:recursion1:feas}, as requested.

Suppose now that Assumption \ref{lipschitz}(ii) holds. In this case, the inequalities in \eqref{term3} can be replaced by
\begin{equation}\label{term3:unif}
\langle T(x^*),\Pi_X(x^k)-x^k\rangle
\le\Vert T(x^*)\Vert\Vert\Pi_X(x^k)-x^k\Vert
\le \sqrt{2C_F}\dist(x^k,X),
\end{equation}
using Assumption \ref{lipschitz}(ii) and  
the fact that $\Vert T(x^*)\Vert^2\le\esp\big[\Vert F(v^k,x^*)\Vert^2\big|\alg_k\big]\le2C_F^2$, which follows from Jensen's inequality,
in the last inequality. 
Hence, combining \eqref{term1}, \eqref{term2.2} and \eqref{term3:unif} we get, instead of \eqref{eq9}, 
\begin{equation}\label{eq9:unif}
\langle T(x^k),x^*-x^k\rangle\le-\rho\dist(x^k,X^*)+\left(\rho+\sqrt{2C_F}\right)\dist(x^k,X).
\end{equation}
Using Assumption \ref{lipschitz}(ii) and \eqref{eq9:unif} 
in \eqref{lemma:recursion:eq2} we get
\begin{eqnarray}
\esp\big[\Vert x^{k+1}-x^*\Vert^2\big|\alg_k\big]&\le &
\Vert x^{k}-x^*\Vert^2-2\rho\alpha_k\dist(x^k,X^*)
+2C_F^2\left[1+\mathsf{B}_{k}\tau\right]\alpha_k^2\nonumber\\
&&-\mathsf{A}_{k,\tau}\dist(x^k,X)^2
+2\left(\rho+\sqrt{2C_F}\right)\alpha_k\dist(x^k,X).\label{eq10:unif}
\end{eqnarray}
In view of Assumption \ref{existence}, we define $\bar x^k:=\Pi_{X^*}(x^k)$. 
Note that $\bar x^k\in\alg_k$ because $\Pi_{X^*}$ is continuous and $x^k\in\alg_k$. 
From \eqref{eq10:unif} we get 
\begin{eqnarray}
\esp\big[\dist(x^{k+1},X^*)^2\big|\alg_k\big]&\le &
\esp\big[\Vert x^{k+1}-\bar x^k\Vert^2\big|\alg_k\big]\le 
\dist(x^{k},X^*)^2-2\rho\alpha_k\dist(x^k,X^*)
+2C_F^2\left[1+\mathsf{B}_{k}\tau\right]\alpha_k^2\nonumber\\
&&-\mathsf{A}_{k,\tau}\dist(x^k,X)^2
+2\left(\rho+\sqrt{2C_F}\right)\alpha_k\dist(x^k,X),\label{eq10:unif:dist}
\end{eqnarray}
using $x^k,\bar x^k\in\alg_k$, $\Vert x^k-\bar x^k\Vert=\dist(x^k,X^*)$ and \eqref{eq10:unif} in 
the second inequality. We rearrange now the last two terms in the right hand side of \eqref{eq10:unif:dist} 
(in a way similar to \eqref{eee1} or \eqref{eee1:feas}), and obtain \eqref{lemma:recursion2} or \eqref{lemma:recursion2:feas}.
\end{proof}
	
\begin{theorem}[Asymptotic convergence]\label{convergence}
Under Assumptions \ref{existence}-\ref{weakly.sharp.hyphotesis}, method \eqref{algo.extra.approx.eq1}-\eqref{algo.extra.approx.eq2} 
generates a sequence $\{x^k\}$ which a.s. is bounded and 
$
\lim_{k\rightarrow\infty}\dist(x^k,X^*)=0.
$
In particular, a.s. all cluster points of $\{x^k\}$ belong to $X^*$.
\end{theorem}
\begin{proof}
We begin by imposing Assumption \ref{lipschitz}(i). Choose some $x^*\in X^*$ (Assumption \ref{existence}) and $\tau>1$. 
By Assumption \ref{approx.step} and the definitions 
given in Lemma \ref{lema:recursion}, we have $\sum_k\alpha_k^2<\infty$, $\sum_k\alpha_k^2\mathsf{A}_{k,\tau}^{-1}<\infty$ 
and $0<\mathsf{B}_{k}\tau\le\tau$, since $\beta_k(2-\beta_k)\in(0,1]$, 
for $\beta_k\in(0,2)$ for all $k$. Hence, we can invoke 
\eqref{lemma:recursion1} in 
Theorem \ref{rob} in order to 
to conclude that, a.s., $\{\Vert x^k-x^*\Vert\}$ converges and, in particular, $\{x^k\}$ is bounded. 

In view of Assumption \ref{existence}, we can define $\bar x^k:=\Pi_{X^*}(x^k)$. We have $\bar x^k\in\alg_k$ 
because $x^k\in\alg_k$ and $\Pi_{X^*}$ is continuous. Since \eqref{lemma:recursion1} in  Lemma \ref{lema:recursion} 
holds for any $x^*\in X^*$ and $\dist(x^k,X^*)=\Vert x^k-\bar x^k\Vert$, we conclude that for all $k\in\mathbb{N}$,
\begin{eqnarray}
\esp\big[\dist(x^{k+1},X^*)^2\big|\alg_k\big] &\le &\esp\big[\Vert x^{k+1}-\bar x^k\Vert^2\big|\alg_k\big]\nonumber\\ 
&\le &\left[1+2\left(1+\mathsf{B}_{k}\tau\right)L^2\alpha_k^2\right]\Vert x^{k}-\bar x^k\Vert^2-2\rho\alpha_k\dist(x^k,X^*)\nonumber\\
&+&\left[\frac{\mathsf{C}(\bar x^k)}{\mathsf{A}_{k,\tau}}+2\left(1+\mathsf{B}_{k}\tau\right)B(\bar x^k)^2\right]\alpha_k^2\nonumber\\
&=&\left[1+2\left(1+\mathsf{B}_{k}\tau\right)L^2\alpha_k^2\right]\dist(x^k,X^*)^2-2\rho\alpha_k\dist(x^k,X^*)\nonumber\\
&+&\left[\frac{\mathsf{C}(\bar x^k)}{\mathsf{A}_{k,\tau}}+2\left(1+\mathsf{B}_{k}\tau\right)B(\bar x^k)^2\right]\alpha_k^2,\label{eq12}
\end{eqnarray}
using relation \eqref{lemma:recursion1} and $\bar x^k\in\alg_k$ in the second inequality.  

We observe that	the function $B:X^*\rightarrow\re_+$ defined in Lemma \ref{l1} is locally bounded because 
$T$ is continuous. Using this fact, the continuity of $\Pi_{X^*}$, 
the a.s.-boundedness of $\{x^k\}$ and $\bar x^k=\Pi_{X^*}(x^k)$, 
we conclude that $\{B(\bar x^k)\}$ and $\{\mathsf{C}(\bar x^k)\}$ are a.s.-bounded. From the a.s.-boundedness 
of $\{B(\bar x^k)\}$ and $\{\mathsf{C}(\bar x^k)\}$ and the conditions $\sum_k\alpha_k^2<\infty$, 
$\sum_k\alpha_k^2\mathsf{A}_{k,\tau}^{-1}<\infty$ and $0<\mathsf{B}_{k}\tau\le\tau$ for all $k$, which hold 
by Assumption \ref{approx.step}, we conclude from Theorem \ref{rob} and \eqref{eq12} that a.s. $\{\dist^2(x^k,X^*)\}$ converges, and 
$$
\sum_{k=0}^\infty2\rho\alpha_k\dist(x^k,X^*)<\infty.
$$
By  Assumption \ref{approx.step}, we also have that $\sum_k\alpha_k=\infty$, so that the above relation implies a.s.
$
\liminf_{k\rightarrow\infty}\dist(x^k,X^*)=0.
$
In particular, the sequence $\{\dist(x^k,X^*)\}$ has a subsequence that converges to zero almost surely. 
Since $\{\dist(x^k,X^*)\}$ a.s. converges, we conclude that the whole sequence a.s. converges to $0$. 
The proof under Assumption \ref{lipschitz}(ii) is similar, using \eqref{lemma:recursion2}.
\end{proof}

\subsection{Convergence rate analysis}\label{ss2.4}

In this subsection we present convergence rate results for the method \eqref{algo.extra.approx.eq1}-\eqref{algo.extra.approx.eq2} 
under the weak sharpness property \eqref{weakly.sharp.assumption}. The solvability metric will be $\dist(\cdot,X^*)$ while the 
feasibility metric will be $\dist(\cdot,X)^2$. We define, for $\ell\le k$, 
\begin{equation}\label{def:S:hatx:weaksharp}	
\mathsf{S}_\ell^k:=\sum_{i=\ell}^k\alpha_i,\quad
\widehat x^k:=\frac{\sum_{i=0}^k\alpha_i x^i}{\mathsf{S}_0^k},\quad
\widehat x^k_\ell:=\frac{\sum_{i=\ell}^k\alpha_i x^i}{\mathsf{S}_\ell^k},
\end{equation}
\begin{equation}\label{def:Z:tildex:weaksharp}
\mathsf{Z}_{\ell}^k:=\sum_{i=\ell}^k\mathsf{B}_i,\quad
\widetilde x^k:=\frac{\sum_{i=0}^k\mathsf{B}_i x^i}{\mathsf{Z}_0^k},\quad
\widetilde x^k_\ell:=\frac{\sum_{i=\ell}^k\mathsf{B}_i x^i}{\mathsf{Z}_\ell^k},
\end{equation}
where $\widehat x^k$ is the ergodic average of the iterates and $\widehat x^k_\ell$ is the window-based ergodic average of the iterates 
when the stepsizes $\{\alpha_k\}$ are used to compute the weights. The solvability metric will be given in terms of 
$\widehat x^k$ or $\widehat x^k_\ell$. The definitions of $\widetilde x^k$ and $\widetilde x^k_\ell$ are analogous, 
but using $\mathsf{B}_k=\beta_k(2-\beta_k)$ for computing the weights. The feasibility metric will be given in terms of such ergodic averages. 

In order to obtain convergence rates for the case of an unbounded feasible set $X$ or unbounded constraint sets
$\{X_0\}\cup\{X_i:i\in\mathcal{I}\}$, we shall need the following proposition, which ensures that the sequence is bounded in $L^2$. 
A typical situation is the case in which $X$ is a polyhedron, i.e. $X_0=\re^n$ and the selected constraints 
$\{X_i\}_{i\in\mathcal{I}}$ are halfspaces, which have easily computable projections but are unbounded sets. 
If the uniform bound of Assumption \ref{lipschitz}(ii) holds, then sharper bounds are given 
in \eqref{prop:mean:bound2}. We shall define for $\tau>1$,
\begin{equation}\label{def:G:H}
\mathsf{G}_\tau:=c C_g^2\tau(\tau-1)^{-1},\quad
\mathsf{H}_\tau:=2\left(1+\tau\right),
\end{equation}
and for $\ell\le k$,	
\begin{equation}\label{def:ak:bk}
\mathsf{a}_\ell^k:=\sum_{i=\ell}^k\alpha_i^2,\quad 
\mathsf{b}_\ell^k:=\sum_{i=\ell}^k\frac{\alpha_i^2}{\beta_i(2-\beta_i)}.
\end{equation}

\begin{proposition}[Boundedness in $L^2$]\label{prop:boundedness:L1}
Suppose that Assumptions \ref{existence}-\ref{weakly.sharp.hyphotesis} hold.

Under Assumption \ref{lipschitz}(i), choose $\tau>1$, $k_0\in\mathbb{N}$ and $0<\gamma<\frac{1}{2(1+\tau)L^2}$ such that
\begin{equation}\label{prop:mean:bound}
\sum_{k\ge k_0}\frac{\alpha^2_k}{\beta_k(2-\beta_k)}<\gamma.
\end{equation}
Then for all $x^*\in X^*$,
\begin{equation}\label{prop:mean:bound1}
\sup_{k\ge k_0}\esp\left[\Vert x^{k}-x^*\Vert^2\right]\le\frac{\esp\left[\Vert x^{k_0}-x^*\Vert^2\right]+
\left[\mathsf{G}_\tau\mathsf{C}(x^*)^2+\mathsf{H}_\tau B(x^*)^2\right]\gamma}{1-\mathsf{H}_\tau L^2\gamma}.
\end{equation}
	
If Assumption \ref{lipschitz}(ii) holds, then for all $k\in\mathbb{N}$,
\begin{equation}\label{prop:mean:bound2}
\sup_{0\le i\le k}\esp\left[\dist(x^i,X^*)^2\right]\le
\dist(x^0,X^*)^2+\mathsf{G}_\tau\left(\rho+\sqrt{2C_F}\right)^2
\cdot\mathsf{b}_0^{k-1}+\mathsf{H}_\tau C_F^2\cdot\mathsf{a}_0^{k-1}.
\end{equation}
\end{proposition}
\begin{proof}
We first prove \eqref{prop:mean:bound1} under Assumption \ref{lipschitz}(i). 
Recall the definitions of $\mathsf{A}_{k,\tau}$ and $\mathsf{B}_{k}\tau$ in \eqref{def:Bk:Ak:C}.
By Assumption \ref{approx.step}, we can choose $k_0\in\mathbb{N}$ and $\gamma>0$ such that 
\eqref{prop:mean:bound} holds. Observe that 
$\beta_k(2-\beta_k)\in(0,1]$, 
because $\beta_k\in(0,2)$, 
so that 
$$
\sum_{k\ge k_0}\alpha^2_k\le\sum_{k\ge k_0}\frac{\alpha^2_k}{\beta_k(2-\beta_k)}<\gamma.
$$

Fix $x\in X^*$ and $\tau>1$. Define 
$$
z_k:=\esp[\Vert x^k-x^*\Vert^2], \quad
D_k^2:=\frac{\mathsf{C}(x^*)^2}{\mathsf{A}_{k,\tau}}+2\left(1+\mathsf{B}_{k}\tau\right)B(x^*)^2,\quad
D^2:=\frac{\mathsf{C}(x^*)^2c C_g^2\tau}{\tau-1}+2\left(1+\tau\right)B(x^*)^2.
$$
For any $k>k_0$, we take the total expectation and sum \eqref{lemma:recursion1} from $k_0$ to $k-1$, obtaining
\begin{equation}\label{prop:telescope}
z_k\le z_{k_0}+ 
\sum_{i=k_0}^{k-1}\left[2\left(1+\mathsf{B}_{i}\tau\right)L^2\alpha_i^2z_i+D_i^2\alpha_i^2\right].
\end{equation}
Given an arbitrary $a>z_{k_0}^{1/2}$, define
\begin{equation}\label{prop:tau}
\Gamma_a:=\inf\{k\ge k_0:z_k>a^2\}.
\end{equation}
Suppose first that $\Gamma_a<\infty$ for all $a>z_{k_0}^{1/2}$. Then by \eqref{prop:mean:bound}, 
\eqref{prop:telescope} and \eqref{prop:tau} we get
$$
a^2<z_{\Gamma_a}\le z_{k_0}+ 
\sum_{i=k_0}^{\Gamma_a-1}\left[2\left(1+\mathsf{B}_{i}\tau\right)L^2\alpha_i^2a^2+D_i^2\alpha_i^2\right]
\le z_{k_0}+2(1+\tau)L^2\gamma a^2+D^2\gamma,
$$
using the fact that $\beta_i(2-\beta_i)\in(0,1]$ in the definition of $\mathsf{B}_{i}\tau$, 
and the definitions of $\mathsf{A}_{i,\tau}$, $D_i^2$ and $D^2$. Hence 
$$
a^2\le\frac{z_{k_0}+D^2\gamma}{1-2(1+\tau)L^2\gamma},
$$
using the fact that $0<\gamma<[2(1+\tau)L^2]^{-1}$. Since $a>z_{k_0}^{1/2}$ is arbitrary, it follows that 
\begin{equation}\label{prop:bound:case1}
\sup_{k\ge k_0}z_k\le\frac{z_{k_0}+D^2\gamma}{1-2(1+\tau)L^2\gamma},
\end{equation}
using again the fact that  $0<\gamma<[2(1+\tau)L^2]^{-1}$. In view of \eqref{prop:tau}-\eqref{prop:bound:case1},
we have a contradiction with the 
assumption that $\Gamma_a<\infty$ for any $a>z_{k_0}^{1/2}$. Hence, there exists some $\bar a>z_{k_0}^{1/2}$ such that 
$\Gamma_{\bar a}=\infty$, so that the set in the right hand side of \eqref{prop:tau} is empty. 
In this case  we have $\sup_{k\ge k_0}z_k\le \bar a^2<\infty$. If $\sup_{k\ge k_0}z_k=z_{k_0}$, 
then \eqref{prop:mean:bound1} holds trivially, since $1-\mathsf{H}_\tau L^2\gamma\in(0,1)$. 
Otherwise, $\hat a:=(\sup_{k\ge k_0}z_k)^{1/2}>z_{k_0}^{1/2}$. From \eqref{prop:mean:bound}, \eqref{prop:telescope},  
$\beta_i\in(0,2)$ and the definitions of $\mathsf{A}_{i,\tau}$, $\mathsf{B}_{i}\tau$, 
$D_i^2$ and $D$, we have for all $k\ge k_0$,
$$
z_k\le z_{k_0}+ 
\sum_{i=k_0}^{k-1}\left[2\left(1+\mathsf{B}_{i}\tau\right)L^2\alpha_i^2\hat a^2+D_i^2\alpha_i^2\right]
\le z_{k_0}+2(1+\tau)L^2\gamma\hat a^2+D^2\gamma,
$$
implying that
$
\hat a^2=\sup_{k\ge k_0}z_k\le z_{k_0}+2(1+\tau)L^2\gamma\hat a^2+D^2\gamma,
$
so that
\begin{equation}\label{prop:bound:case2}
\sup_{k\ge k_0}z_k=\hat a^2\le\frac{z_{k_0}+D^2\gamma}{1-2(1+\tau)L^2\gamma},
\end{equation}
using again $0<\gamma<[2(1+\tau)L^2]^{-1}$. From \eqref{prop:bound:case2} and the
definitions of $\mathsf{G}_\tau$, $\mathsf{H}_\tau$ and $D$, we conclude that \eqref{prop:mean:bound1} holds.

We now prove \eqref{prop:mean:bound2} under Assumption \ref{lipschitz}(ii). 
As before, we define
$$
\widehat D_k^2:=\frac{\left(\rho+\sqrt{2C_F}\right)^2}{\mathsf{A}_{k,\tau}}+2\left(1+\mathsf{B}_{k}\tau\right)C_F^2,
$$
Taking total expectation in \eqref{lemma:recursion2} and summing from $0$ to $k-1$, we get 
\begin{equation}\label{e101}
\esp\left[\dist(x^k,X^*)^2\right]\le\dist(x^0,X^*)^2+\sum_{i=0}^{k-1}\widehat D_i^2\alpha_i^2
\le\dist(x^0,X^*)^2+\left(\rho+\sqrt{2C_F}\right)^2\frac{c C_g^2\tau}{\tau-1}\mathsf{b}_0^{k-1}+
2(1+\tau)C_F^2\mathsf{a}_0^{k-1},
\end{equation}
for all $k\ge0$,
using the fact that $\beta_i\in(0,2)$ and the definitions of $\mathsf{A}_{i,\tau}$, $\mathsf{B}_{i}\tau$, $\widehat D_i^2$, 
$\mathsf{a}_0^{k-1}$ and $\mathsf{b}_0^{k-1}$. We conclude from \eqref{e101}, the definitions of $\mathsf{G}_\tau$ and $\mathsf{H}_\tau$, 
and the monotonicity of the sequences $\{\mathsf{a}_0^k,\mathsf{b}_0^k\}$
that \eqref{prop:mean:bound2} holds.
\end{proof}

Next we will give convergence rate results for the original sequence $\{x^k\}$ and for the ergodic average sequences. 
We consider separately the cases of unbounded operators (Assumption \ref{lipschitz}(i))
and the case of bounded ones (Assumption \ref{lipschitz}(ii)), because in the later case sharper rates are possible. 
In the remainder of this subsection, we refer the reader to definitions \eqref{def:B}, \eqref{def:Bk:Ak:C}, 
\eqref{def:S:hatx:weaksharp}-\eqref{def:Z:tildex:weaksharp} and \eqref{def:G:H}-\eqref{def:ak:bk}.

\begin{theorem}[Solvability and feasibility rates of convergence: unbounded case]\label{convergence.rate}
\quad

Suppose that Assumptions \ref{existence}-\ref{weakly.sharp.hyphotesis} and Assumption \ref{lipschitz}(i) hold.
Choose $\tau>1$, $k_0\in\mathbb{N}$ and $\phi\in(0,1)$ such that
\begin{equation}\label{thm:mean:bound}
\sum_{k\ge k_0}\frac{\alpha^2_k}{\beta_k(2-\beta_k)}\le\frac{\phi}{2(1+\tau)L^2}.
\end{equation}
Define for $x^*\in X^*$,	
\begin{equation}\label{thm:E}
\mathsf{E}_k(x^*,k_0,f,g):=f\cdot\left\{\Vert x^0-x^*\Vert^2+
\left[\mathsf{I}(x^*,k_0)L^2+B(x^*)^2\right]\mathsf{H}_\tau\mathsf{a}_0^k
+g\mathsf{G}_\tau\mathsf{C}(x^*)^2\mathsf{b}_0^k\right\},
\end{equation}
\begin{equation}\label{thm:I}
\mathsf{I}(x^*,k_0):=\frac{\max_{0\le k\le k_0}\esp\left[\Vert x^{k}-x^*\Vert^2\right]+
\left[\mathsf{G}_\tau\mathsf{H}_\tau^{-1}\mathsf{C}(x^*)^2L^{-2}+ B(x^*)^2L^{-2}\right]\phi}{1-\phi}.
\end{equation}
Then $\dist(x^k,X^*)$ a.s.-converges to $0$ and the following holds:
\begin{itemize}
\item[a)] For any $\epsilon>0$, there exists $M:=M_\epsilon\in\mathbb{N}$, such that for all $x^*\in X^*$,
$
\esp\left[\dist(x^M,X^*)\right]<\epsilon\le\mathsf{E}_\infty(x^*,k_0,1/2\rho,1)/\mathsf{S}_0^{M-1}.
$
\item[b)] For all $k\in\mathbb{N}$ and all $x^*\in X^*$, 
$
\esp\left[\dist(\widehat x^k,X^*)\right]\le\mathsf{E}_k(x^*,k_0,1/2\rho,1)/\mathsf{S}_0^k.
$
\item[c)] For any $\epsilon>0$, there exists $N:=N_\epsilon\in\mathbb{N}$, such that for all $x^*\in X^*$,
$
\esp\left[\dist(x^N,X)^2\right]<\epsilon\le\mathsf{E}_\infty(x^*,k_0,2\mathsf{G}_\tau,2)/\mathsf{Z}_0^{N-1}.
$
\item[d)] For all $k\in\mathbb{N}$ and all $x^*\in X^*$, 
$
\esp\left[\dist(\widetilde x^k,X)^2\right]\le\mathsf{E}_k(x^*,k_0,2\mathsf{G}_\tau,2)/\mathsf{Z}_0^k.
$
\end{itemize}
\end{theorem}
\begin{proof}
Fix $\tau>1$, $k_0\in\mathbb{N}$ and $\phi\in(0,1)$ as in \eqref{thm:mean:bound}. This is possible because 
$\sum_{i\ge k}\alpha_i^2\beta_i^{-1}(2-\beta_i)^{-1}$ converges to $0$ as $k\rightarrow\infty$ by Assumption \ref{approx.step}. 
We now invoke Lemma \ref{lema:recursion}. We take the total expectation in \eqref{lemma:recursion1} and sum 
from $\ell$ to $k$, obtaining, for every $x^*\in X^*$,
\begin{eqnarray}
&&2\rho\sum_{i=\ell}^{k}\alpha_i\esp\left[\dist(x^i,X^*)\right]\nonumber\\
&\le &\esp\left[\Vert x^\ell-x^*\Vert^2\right]+
\sum_{i=\ell}^{k}2\left(1+\mathsf{B}_{i}\tau\right)L^2\alpha_i^2\esp\left[\Vert x^{i}-x^*\Vert^2\right]
+\sum_{i=\ell}^{k}\left[\frac{\mathsf{C}(x^*)^2}{\mathsf{A}_{i,\tau}}
+2\left(1+\mathsf{B}_{i}\tau\right)B(x^*)^2\right]\alpha_i^2\nonumber\\
&\le &\esp\left[\Vert x^\ell-x^*\Vert^2\right]+
\sup_{\ell\le i\le k}\esp\left[\Vert x^{i}-x^*\Vert^2\right]\cdot\sum_{i=\ell}^{k}2\left(1+\mathsf{B}_{i}\tau\right)L^2\alpha_i^2
+\sum_{i=\ell}^{k}\left[\frac{\mathsf{C}(x^*)^2}{\mathsf{A}_{i,\tau}}+2\left(1+\mathsf{B}_{i}\tau\right)B(x^*)^2\right]\alpha_i^2\nonumber\\
&\le &\esp\left[\Vert x^\ell-x^*\Vert^2\right]+
\sup_{i\ge0}\esp\left[\Vert x^{i}-x^*\Vert^2\right]\cdot\mathsf{H}_\tau L^2\mathsf{a}_\ell^k
+\mathsf{G}_\tau\mathsf{C}(x^*)^2\mathsf{b}_\ell^k+\mathsf{H}_\tau B(x^*)^2\mathsf{a}_\ell^k,\label{thm:rate:conv:eq1}
\end{eqnarray}
using $\beta_i(2-\beta_i)\in(0,1]$ and the definitions of $\mathsf{A}_{i,\tau}$, $\mathsf{B}_{i}\tau$, 
$\mathsf{G}_{\tau}$, $\mathsf{H}_{\tau}$, $\mathsf{a}_\ell^k$ and $\mathsf{b}_\ell^k$ in the last inequality. 

We now invoke Proposition \ref{prop:boundedness:L1}. Setting $\gamma:=\frac{\phi}{2(1+\tau)L^2}$, \eqref{prop:mean:bound} 
can be rewritten as \eqref{thm:mean:bound}. From \eqref{prop:mean:bound1} and $1-\mathsf{H}_\tau L^2\in(0,1)$, we get, for all $x^*\in X^*$, 
\begin{equation}\label{thm:rate:conv:eq2}	
\sup_{i\ge0}\esp\left[\Vert x^{i}-x^*\Vert^2\right]\le
\frac{\max_{0\le i\le k_0}\esp\left[\Vert x^{i}-x^*\Vert^2\right]+
\left[\mathsf{G}_\tau\mathsf{C}(x^*)^2+\mathsf{H}_\tau B(x^*)^2\right]\gamma}{1-\mathsf{H}_\tau L^2\gamma}=\mathsf{I}(x^*,k_0),
\end{equation}
using the definitions of $\mathsf{H}_\tau=2(1+\tau)$, $\gamma$ and $\mathsf{I}(x^*,k_0)$.

We prove now item (a). For every $\epsilon>0$, define
\begin{equation}\label{thm:rate:conv:M:def}
M=M_\epsilon:=\inf\{k\in\mathbb{N}:\esp\left[\dist(x^k,X^*)\right]<\epsilon\}.
\end{equation}
From the definition of $M$ we have, for every $k<M$,
\begin{equation}\label{thm:rate:conv:eq3}
2\rho\epsilon\sum_{i=0}^{k}\alpha_i\le2\rho\sum_{i=0}^{k}\alpha_i\esp\left[\dist(x^i,X^*)\right].	
\end{equation}
We claim that $M$ is finite. Indeed, if $M=\infty$, then \eqref{thm:rate:conv:eq1}, \eqref{thm:rate:conv:eq2} 
and \eqref{thm:rate:conv:eq3} hold for $\ell:=0$ and all $k\in\mathbb{N}$. 
Hence, letting $k\rightarrow\infty$ and using that $\mathsf{a}_0^\infty<\infty$ and $\mathsf{b}_0^\infty<\infty$, which
hold by Assumption \ref{approx.step}, we obtain $\sum_k\alpha_k<\infty$, which contradicts Assumption \ref{approx.step}. 
Hence, the set in the right hand side of \eqref{thm:rate:conv:M:def} is nonempty, which implies 
$\esp[\dist(x^M,X^*)]<\epsilon$. Setting $\ell:=0$ and $k:=M-1$ in \eqref{thm:rate:conv:eq1}, \eqref{thm:rate:conv:eq2} 
and \eqref{thm:rate:conv:eq3}, we get for all $x^*\in X^*$,	
$$
\sum_{i=0}^{M-1}\alpha_i\le\frac{\mathsf{E}_{M-1}(x^*,k_0,1/2\rho,1)}{\epsilon}\le
\frac{\mathsf{E}_{\infty}(x^*,k_0,1/2\rho,1)}{\epsilon},
$$
using the definition of $\mathsf{E}_k(x^*,k_0,1/2\rho,1)$. We thus obtain item (a).

We now prove item (b). In view of the convexity of the function $x\longmapsto\dist(x,X^*)$, and the linearity and monotonicity 
of the expected value, we have 
\begin{equation}\label{thm:rate:conv:eq5}
\esp\left[\dist(\widehat x_\ell^k,X^*)\right]=
\esp\left[\dist\left(\frac{\sum_{i=\ell}^k\alpha_ix^i}{\sum_{i=\ell}^k\alpha_i},X^*\right)\right]\le
\frac{\sum_{i=\ell}^k\alpha_i\esp\left[\dist(x^i,X^*)\right]}{\sum_{i=\ell}^k\alpha_i}.
\end{equation}
Set $\ell:=0$, divide \eqref{thm:rate:conv:eq1} by 
$2\rho\sum_{i=0}^k\alpha_i=2\rho\mathsf{S}_0^k$ and use \eqref{thm:rate:conv:eq5}, the definition of $\mathsf{E}_k(x^*,k_0,1/2\rho,1)$ 
together with \eqref{thm:rate:conv:eq2}, in order to bound $\sup_{i\ge0}\esp[\Vert x^i-x^*\Vert^2]$, and obtain item (b) as a consequence.

The proofs of items (c) and (d) follow the proofline above, using \eqref{lemma:recursion1:feas} instead of \eqref{lemma:recursion1}.
\end{proof}

\begin{corollary}[Solvability and feasibility rates with robust stepsizes: unbounded case]
\label{cor:rate:conv}
Assume that the hypotheses of Theorem \ref{convergence.rate} hold. 
Given $\theta>0$ and $\lambda>0$, define $\{\alpha_k\}$ as: $\alpha_0=\alpha_1=\theta$ and for $k\ge2$,
\begin{equation}\label{cor:rate:conv:step}
\alpha_k:=\frac{\theta}{\sqrt{k\left(\ln k\right)^{1+\lambda}}},
\end{equation} 
and choose $\beta_k\equiv\beta\in(0,2)$, $\tau>1$ and $\phi\in(0,1)$. Take $k_0\ge2$ as the minimum natural number such that
\begin{equation}\label{cor:rate:conv:k0}
k_0\ge\exp\left[\left(\frac{2(1+\tau)L^2\theta^2}{\lambda\beta(2-\beta)\phi}\right)^{1/\lambda}\right]+1.
\end{equation}
Define
$$
\mathsf{J}_\beta(x^*,k_0,g):=\left[\mathsf{I}(x^*,k_0)L^2+B(x^*)^2\right]\mathsf{H}_\tau+
g\mathsf{G}_\tau\mathsf{C}(x^*)^2\beta^{-1}(2-\beta)^{-1},\forall x\in X^*,
$$
$$
\mathsf{Q}_{\beta,\lambda}(x^0,k_0,g):=\inf_{x^*\in X^*}\left\{\Vert x^0-x^*\Vert^2+
\mathsf{J}_\beta(x^*,k_0,g)\left[2+\frac{1}{2(\ln 2)^{1+\lambda}}+\frac{1}{\lambda(\ln 2)^\lambda}\right]\right\}.
$$

Then $\dist(x^k,X^*)$ a.s.-converges to $0$ and the following holds:
\begin{itemize}
\item[a)] For every $\epsilon>0$, there exists $M=M_\epsilon\ge2$ such that
$$
\esp\left[\dist(x^M,X^*)\right]<\epsilon\le 
\frac{\max\{\theta,\theta^{-1}\}}{2\rho}\cdot\frac{\left[\ln (M-1)\right]^{\frac{1+\lambda}{2}}}{\sqrt{M}}\mathsf{Q}_{\beta,\lambda}(x^0,k_0,1).
$$	
\item[b)] For all $k\ge2$,
$$
\esp\left[\dist(\widehat x^k,X^*)\right]\le
\frac{\max\{\theta,\theta^{-1}\}}{2\rho}\cdot\frac{(\ln k)^{\frac{1+\lambda}{2}}}{\sqrt{k+1}}\mathsf{Q}_{\beta,\lambda}(x^0,k_0,1).
$$
\item[c)] For every $\epsilon>0$, there exists $N=N_\epsilon\in\mathbb{N}$ such that
$$
\esp\left[\dist(x^N,X)^2\right]<\epsilon\le
2\mathsf{G}_\tau\frac{\max\{1,\theta^2\}}{\beta(2-\beta)}\cdot\frac{\mathsf{Q}_{\beta,\lambda}(x^0,k_0,2)}{N}.
$$	
\item[d)] For all $k\in\mathbb{N}_0$,
$$
\esp\left[\dist(\widetilde x^k,X)^2\right]\le
2\mathsf{G}_\tau\frac{\max\{1,\theta^2\}}{\beta(2-\beta)}\cdot\frac{\mathsf{Q}_{\beta,\lambda}(x^0,k_0,2)}{k+1}.
$$
\end{itemize}
\end{corollary}
\begin{proof}
We estimate $k_0$ in \eqref{thm:I}. Since
$$
\sum_{k\ge k_0}\alpha_k^2 <
\theta^2\int_{k_0-1}^{\infty}t^{-1}(\ln t)^{-(1+\lambda)}\dist t\nonumber\\
=\frac{\theta^2}{\lambda\left[\ln(k_0-1)\right]^\lambda},
$$
we conclude from \eqref{thm:mean:bound} that it is enough to choose the minimum $k_0\ge2$ such that
$$
\frac{\theta^2}{\lambda\left[\ln(k_0-1)\right]^\lambda}\le\frac{\beta(2-\beta)\phi}{2(1+\tau)L^2},
$$
that is to say, the minimum $k_0\ge2$ such that \eqref{cor:rate:conv:k0} holds.

Let $k\ge2$. We first estimate the sum of the stepsize sequence. For any $0\le\ell\le k$ we have
\begin{equation}\label{cor:rate:conv:eq1}	
\mathsf{S}_\ell^k=\sum_{i=\ell}^{k}\alpha_i 
\ge\frac{\theta(k-\ell+1)}{\sqrt{k(\ln k)^{1+\lambda}}},
\end{equation}
using the fact that the minimum stepsize between $\ell$ and $k\ge2$ is 
$\theta k^{-\frac{1}{2}}(\ln k)^{\frac{1+\lambda}{2}}$. The sum of the squares of the stepsizes sequence can be estimated as
\begin{eqnarray}
\mathsf{a}_0^k&\le &\mathsf{a}_0^\infty=\sum_{i=0}^{\infty}\alpha_i^2 =
2\theta^2+\frac{\theta^2}{2(\ln 2)^{1+\lambda}}+\sum_{i=3}^\infty\frac{\theta^2}{i(\ln i)^{1+\lambda}}\nonumber\\
&\le &2\theta^2+\frac{\theta^2}{2(\ln 2)^{1+\lambda}}+\theta^2\int_{2}^{\infty}t^{-1}(\ln t)^{-(1+\lambda)}\dist t
=\theta^2\left[2+\frac{1}{2(\ln 2)^{1+\lambda}}+\frac{1}{\lambda(\ln 2)^\lambda}\right].\label{cor:rate:conv:eq1:2}
\end{eqnarray}

We assume without loss on generality that 
we have $M\ge2$ 
in \eqref{thm:rate:conv:M:def}. 
Item (a) follows from \eqref{cor:rate:conv:eq1} with $k:=M-1$ and $\ell:=0$, \eqref{cor:rate:conv:eq1:2}, 
Theorem \ref{convergence.rate}(a) and the definitions of $\mathsf{J}_\beta(x^*,k_0,1)$, 
$\mathsf{E}_\infty(x^*,k_0,1/2\rho,1)$ and $\mathsf{a}_0^\infty=\beta(2-\beta)\mathsf{b}_0^\infty$.

Similarly, item (b) follows from \eqref{cor:rate:conv:eq1}-\eqref{cor:rate:conv:eq1:2} with $\ell:=0$, Theorem \ref{convergence.rate}(b) and 
the definitions of $\mathsf{J}_\beta(x^*,k_0,1)$ and $\mathsf{E}_k(x^*,k_0,1/2\rho,1)$ and the facts 
that $\mathsf{b}^k_0\le\mathsf{b}_0^\infty$ and	 $\mathsf{a}_0^\infty=\beta(2-\beta)\mathsf{b}_0^\infty$.

The proof of items (c) and (d) follows a similar proofline, using Theorem \ref{convergence.rate}(c)-(d) 
and the fact that $\mathsf{Z}_0^k=\beta(2-\beta)(k+1)$.
\end{proof}

Next we give convergence rates for the bounded case. For simplicity we just state the rates for the ergodic averages, but we 
note that similar rates can be derived for $x^k$ as in Theorem \ref{convergence.rate} and Corollary \ref{cor:rate:conv}.

\begin{theorem}[Solvability and feasibility rates: bounded case]\label{thm:rate:conv:bounded}
Suppose that Assumptions \ref{existence}-\ref{weakly.sharp.hyphotesis} and Assumption \ref{lipschitz}(ii) hold. 
Choose $\tau>1$. Define for $\ell\le k$ in $\mathbb{N}_0\cup\{\infty\}$,
$$
\mathsf{E}_\ell^k[R,f,g]:=f\left\{R^2+g
\mathsf{G}_\tau\Big(\rho+\sqrt{2C_F}\Big)^2\mathsf{b}_\ell^k+\mathsf{H}_\tau C_F^2\mathsf{a}_\ell^k\right\}.
$$
Then, $\dist(x^k,X^*)$ a.s.-converges to zero and 
\begin{itemize}
\item[a)] For all $k\in\mathbb{N}$, 
$
\esp\left[\dist(\widehat x^k,X^*)\right]\le
\mathsf{E}_0^k[\dist(x^0,X^*),1/2\rho,1]/\mathsf{S}_0^k.
$
\item[b)] If $X_0$ is compact, then for all $\ell,k\in\mathbb{N}$ with $\ell<k$, 
$
\esp\left[\dist(\widehat x_\ell^k,X^*)\right]\le
\mathsf{E}_\ell^k[\diam(X_0),1/2\rho,1]/\mathsf{S}_\ell^k.
$
\item[c)] for all $k\in\mathbb{N}$, 
$
\esp\left[\dist(\widetilde x^k,X)^2\right]\le
\mathsf{E}_0^k[\dist(x^0,X^*),2\mathsf{G}_\tau,2]/\mathsf{Z}_0^k.
$
\item[d)] If $X_0$ is compact, then for all $\ell,k\in\mathbb{N}$ with $\ell<k$, 
$
\esp\left[\dist(\widetilde x_\ell^k,X)^2\right]\le
\mathsf{E}_\ell^k[\diam(X_0),2\mathsf{G}_\tau,2]/\mathsf{Z}_\ell^k.
$
\end{itemize}
\end{theorem}
\begin{proof}
Fix $\tau>1$. We will invoke Lemma \ref{lema:recursion}. We take the total expectation in \eqref{lemma:recursion2} and sum 
from $\ell$ to $k$, obtaining
\begin{eqnarray}
2\rho\sum_{i=\ell}^{k}\alpha_i\esp\left[\dist(x^i,X^*)\right]&\le &
\esp\left[\dist(x^\ell,X^*)^2\right]+
\sum_{i=\ell}^{k}\left[\frac{\left(\rho+\sqrt{2C_F}\right)^2}{\mathsf{A}_{i,\tau}}+
2\left(1+\mathsf{B}_{i}\tau\right)C_F^2\right]\alpha_i^2\nonumber\\
&\le &\esp\left[\dist(x^\ell,X^*)^2\right]+
\mathsf{G}_\tau\left(\rho+\sqrt{2C_F}\right)^2\mathsf{b}_\ell^k+\mathsf{H}_\tau C_F^2\mathsf{a}_\ell^k,	\label{cor:rate:conv:bounded:eq1}
\end{eqnarray}
using the fact that $\beta_i(2-\beta_i)\in(0,1]$ and the definitions of $\mathsf{A}_{i,\tau}$, $\mathsf{B}_{i}\tau$, 
$\mathsf{G}_{\tau}$, $\mathsf{H}_{\tau}$, $\mathsf{a}_\ell^k$ and $\mathsf{b}_\ell^k$ in last inequality. 
From \eqref{cor:rate:conv:bounded:eq1} on, the proofs of items (a)-(b) are similar to the proof of Theorem \ref{convergence.rate}. 
We omit the details, but make the following remarks: differently to the proofs of items (a)-(b) in Theorem \ref{convergence.rate}, 
the proofs of items (a)-(b) of Theorem \ref{thm:rate:conv:bounded} do not require Proposition \ref{prop:boundedness:L1}. 
In the proof of item (b), we use the bound $\esp[\dist(x^\ell,X^*)^2]\le\diam(X_0)^2$ in \eqref{cor:rate:conv:bounded:eq1}. 
The proofs of items (c)-(d) follow  a similar proofline, using \eqref{lemma:recursion2:feas}.
\end{proof}

\begin{corollary}[Solvability and feasibility rates with robust stepsizes: bounded case]\label{cor:rate:conv:bounded}
Assume that the hypotheses of Theorem \ref{thm:rate:conv:bounded} hold. 
Given $\theta>0$ and $\lambda>0$, define $\{\alpha_k\}$ as: $\alpha_0=\alpha_1=\theta$ and for $k\ge2$,
\begin{equation}\label{cor:rate:conv:step:bounded}
\alpha_k:=\frac{\theta}{\sqrt{k\left(\ln k\right)^{1+\lambda}}},
\end{equation} 
and choose $\beta_k\equiv\beta\in(0,2)$, $\tau>1$. Define
\begin{eqnarray*}
\mathsf{\widehat J}_\beta[g]&:=&\mathsf{H}_\tau C_F^2+g\mathsf{G}_\tau\left(\rho+\sqrt{2C_F}\right)^2\beta^{-1}(2-\beta)^{-1},\\
\mathsf{\widehat Q}_{\beta,\lambda}[x^0,g]&:=&\dist(x^0,X^*)^2+
\mathsf{\widehat J}_\beta[g]\left[2+\frac{1}{2(\ln 2)^{1+\lambda}}+\frac{1}{\lambda(\ln 2)^\lambda}\right].
\end{eqnarray*}
Then $\dist(x^k,X^*)$ a.s.-converges to $0$ and
\begin{itemize}
\item[a)] for all $k\ge2$,
$$
\esp\left[\dist(\widehat x^k,X^*)\right]\le
\frac{\max\{\theta,\theta^{-1}\}}{2\rho}\cdot\frac{(\ln k)^{\frac{1+\lambda}{2}}}{\sqrt{k+1}}\mathsf{\widehat Q}_{\beta,\lambda}[x^0,1],
$$
\item[b)] if $X_0$ is compact, then given $r\in(0,1)$, for all $k\ge2r^{-1}$, it holds that
$$
\esp\left[\dist(\widehat x_{\lceil rk\rceil}^k,X^*)\right]\le
\frac{\max\{\theta,\theta^{-1}\}}{2\rho}\cdot\frac{(\ln k)^{\frac{1+\lambda}{2}}}{\sqrt{k}}
\cdot\left\{(1-r)^{-1}\diam(X_0)^2+\frac{r^{-1}\mathsf{\widehat J}_\beta[1]}{\left[\ln k-\ln(r^{-1})\right]^{1+\lambda}}\right\}.
$$ 
\item[c)] For all $k\in\mathbb{N}_0$,
$$
\esp\left[\dist(\widetilde x^k,X)^2\right]\le2\mathsf{G}_\tau
\frac{\max\{1,\theta^2\}}{\beta(2-\beta)}\cdot\frac{\mathsf{\widehat Q}_{\beta,\lambda}[x^0,2]}{k+1}.
$$
\end{itemize}
\end{corollary}
\begin{proof}

Item (a) follows from \eqref{cor:rate:conv:eq1}-\eqref{cor:rate:conv:eq1:2} with $\ell:=0$, 
Theorem \ref{thm:rate:conv:bounded}(a), the definition of $\mathsf{\widehat J}_\beta[1]$, $\mathsf{E}_0^k[\dist(x^0,X^*),1/2\rho,1]$ 
and the facts that $\mathsf{b}^k_0\le\mathsf{b}_0^\infty$ and $\mathsf{a}_0^\infty=\beta(2-\beta)\mathsf{b}_0^\infty$.

The proof of item (c) follows  a similar proofline,  using Theorem \ref{thm:rate:conv:bounded}(c) and $\mathsf{Z}_{0}^k=\beta(2-\beta)(k+1)$.
	
We now prove item (b). Let $r\in(0,1)$, $k\ge2r^{-1}$ and set $\ell:=\lceil rk\rceil$. 
We have $\ell\ge2$ and $rk\le\ell\le rk+1$.  We estimate 
\begin{equation}\label{cor:rate:conv:eq2}
\mathsf{a}_\ell^k=\sum_{i=\ell}^{k}\alpha_i^2 =
\sum_{i=\ell}^k\frac{\theta^2}{i(\ln i)^{1+\lambda}}
\le \frac{\theta^2(k-\ell+1)}{\ell(\ln\ell)^{1+\lambda}}.	
\end{equation}
From \eqref{cor:rate:conv:eq1} and \eqref{cor:rate:conv:eq2} we have
\begin{equation}\label{cor:rate:conv:eq3}
\frac{\mathsf{a}_\ell^k}{\mathsf{S}_\ell^k}\le
\frac{\theta\sqrt{k(\ln k)^{1+\lambda}}}{\ell(\ln\ell)^{1+\lambda}}
\le\frac{\theta r^{-1}}{\sqrt{k}}\cdot\frac{\sqrt{(\ln k)^{1+\lambda}}}{(\ln(rk))^{1+\lambda}}=
\theta r^{-1}\frac{(\ln k)^{\frac{1+\lambda}{2}}}{\sqrt{k}\left[\ln k -\ln(r^{-1})\right]^{1+\lambda}},
\end{equation}
\begin{equation}\label{cor:rate:conv:eq4}
\frac{1}{\mathsf{S}_\ell^k}\le
\frac{\theta^{-1}\sqrt{k(\ln k)^{1+\lambda}}}{k-\ell+1}
\le\theta^{-1}(1-r)^{-1}\frac{(\ln k)^{\frac{1+\lambda}{2}}}{\sqrt{k}},
\end{equation}
using the inequality $\ell\ge rk$ in the second inequality of \eqref{cor:rate:conv:eq3} and  
$k-\ell+1\ge(1-r)k$ in the second inequality of \eqref{cor:rate:conv:eq4}. Item (b) follows 
from \eqref{cor:rate:conv:eq3}-\eqref{cor:rate:conv:eq4}, Theorem \ref{thm:rate:conv:bounded}(b), 
the definition of $\mathsf{\widehat J}_\beta[1]$ and $\mathsf{E}^k_\ell[\diam(X_0),1/2\rho,1]$ 
and the fact that $\beta(2-\beta)\mathsf{b}^{k}_{\ell}=\mathsf{a}^k_\ell$. 
\end{proof}
\begin{rem}
Corollary \ref{cor:rate:conv:bounded}(b) implies that, if $X_0$ is compact, then $\dist(\widehat x^k_{\lceil rk\rceil},X^*)$ 
has a better performance than $\dist(x^k,X^*)$ and $\dist(\widehat x^k,X^*)$ when stepsizes as in \eqref{cor:rate:conv:step:bounded} are used. 
Indeed, in Corollary \ref{cor:rate:conv:bounded}(c), $\lambda>0$ can be arbitrarily small, without affecting the constant in the 
convergence rate, and the ``stochastic error'' $r^{-1}\mathsf{\widehat J}_\beta[1]\left[\ln k-\ln(1/r)\right]^{-(1+\lambda)}$ decays 
to zero. For unbounded operators, \eqref{cor:rate:conv:k0} in Corollary \ref{cor:rate:conv} suggests the use of $\lambda>1$ 
and $\theta\sim L$ so that $k_0$ does not become too large. As an example, if $\tau=1.5$, $\theta=L$, $\beta=1$, $\phi=0.5$ and 
$\lambda=2$, we have $k_0=11$. For simplicity we do not state the analogous result for $\dist(\widetilde x^k_{\lceil rk\rceil},X)^2$.
\end{rem}

In Corollaries \ref{cor:rate:conv}-\ref{cor:rate:conv:bounded}, stepsizes $\{\alpha_k\}$ of $O(1)k^{-1/2}(\ln k)^{-(1+\lambda)/2}$ 
are small enough to guarantee asymptotic a.s.-convergence and large enough as to ensure a rate of 
$O(1)k^{-1/2}(\ln k)^{(1+\lambda)/2}$. If asymptotic a.s.-convergence of the whole sequence is not the main concern, 
we show next that one may use larger stepsizes of $O(1)k^{-1/2}$ for ensuring convergence in $L^1$ (hence convergence in 
probability and a.s.-convergence of a subsequence) with a convergence rate of $O(1)k^{-1/2}$. When a constant stepsize $\alpha$ 
is used in method \eqref{algo.extra.approx.eq1}-\eqref{algo.extra.approx.eq2}, we can also give an error bound on the performance 
proportional to $\alpha$. 
Precisely, for fixed $\beta\in(0,2)$, we have $\esp[\dist(\widehat x^k,X^*)]\lesssim k^{-1}+O(\alpha)$ 
and $\esp[\dist(\widetilde x^k,X)^2]\lesssim k^{-1}+O(\alpha^2)$. Such error bounds rigorously justify 
the practical use of constant stepsizes in incremental methods for machine learning, where only an inexact solution is required. 

\begin{corollary}[Solvability and feasibility rates for larger stepsizes: bounded case]\label{cor:rate:conv:const:step}
Assume that the hypotheses of Theorem \ref{thm:rate:conv:bounded} hold. Recall the definition of $\mathsf{\widehat J}_\beta[\cdot]$ 
in Corollary \ref{cor:rate:conv:bounded}. Choose $\theta>0$, $\beta_k\equiv\beta\in(0,2)$ 
and $\tau>1$.  
\begin{itemize}
\item[a)] If we choose a constant stepsize $\alpha_k\equiv\theta\alpha$, then for all $k\ge1$, 
$$
\esp\left[\dist(\widehat x^k,X^*)\right]\le
\frac{\max\{\theta,\theta^{-1}\}}{2\rho}\left\{\frac{\dist(x^0,X^*)^2}{\alpha(k+1)}+\mathsf{\widehat J}_\beta[1]\alpha\right\},
$$
$$
\esp\left[\dist(\widetilde x^k,X)^2\right]\le
\frac{2\mathsf{G}_\tau\max\{1,\theta^2\}}{\beta(2-\beta)}\left\{\frac{\dist(x^0,X^*)^2}{k+1}+\mathsf{\widehat J}_\beta[2]\alpha^2\right\}.
$$
\item[b)] If the total number of iterations $\mathsf{K}\ge1$ is given a priori and for all $k\in[\mathsf{K}]$,
$
\alpha_k\equiv\frac{\theta}{\sqrt{\mathsf{K}+1}},
$
then
$$
\esp\left[\dist(\widehat x^{\mathsf{K}},X^*)\right]\le
\frac{\max\{\theta,\theta^{-1}\}}{2\rho}\cdot\frac{\dist(x^0,X^*)^2+\mathsf{\widehat J}_\beta[1]}{\sqrt{\mathsf{K}+1}},
$$
$$
\esp\left[\dist(\widetilde x^\mathsf{K},X)^2\right]\le
\frac{2\mathsf{G}_\tau\max\{1,\theta^2\}}{\beta(2-\beta)}\cdot\frac{\dist(x^0,X^*)^2+\mathsf{\widehat J}_\beta[2]}{\mathsf{K}+1}.
$$

\item[c)] If $X_0$ is compact and we choose $\alpha_0:=\theta$ and for $k\ge1$,
$
\alpha_k:=\frac{\theta}{\sqrt{k}},
$
then, given $r\in(0,1)$, for all $k\ge r^{-1}$,
$$
\esp\left[\dist(\widehat x_{\lceil rk\rceil}^k,X^*)\right]\le\frac{\max\{\theta,\theta^{-1}\}}{2\rho}\cdot\frac{1}{\sqrt{k}}
\cdot\left\{(1-r)^{-1}\diam(X_0)^2+r^{-1}\mathsf{\widehat J}_\beta[1]\right\},
$$ 
$$
\esp\left[\dist(\widetilde x_{\lceil rk\rceil}^k,X)^2\right]\le\frac{2\mathsf{G}_\tau\max\{1,\theta^2\}}{\beta(2-\beta)}\cdot\frac{1}{k}
\cdot\left\{(1-r)^{-1}\diam(X_0)^2+r^{-1}\mathsf{\widehat J}_\beta[2]\right\}.
$$
\end{itemize}
\end{corollary}
\begin{proof}
Item (a) follows from Theorem \ref{thm:rate:conv:bounded}(a) and (c) and the definitions 
of $\mathsf{\widehat J}_\beta[\cdot]$, $\mathsf{E}^k_0[\cdot]$, $\mathsf{S}_0^k$, $\mathsf{Z}_0^k$, $\mathsf{a}_0^k$ and $\mathsf{b}_0^k$.
Item (b) follows  from item (a). 
We prove now item (c).
Take $r\in(0,1)$, $k\ge r^{-1}$ and set $\ell:=\lceil rk\rceil$. We have $\ell\ge1$ and $rk\le\ell\le rk+1$. We estimate 
\begin{equation}\label{cor:rate:conv:const:eq1}	
\mathsf{S}_\ell^k=\sum_{i=\ell}^{k}\alpha_i 
\ge\frac{\theta(k-\ell+1)}{\sqrt{k}},\quad\quad\mathsf{Z}_\ell^k=\beta(2-\beta)(k-\ell+1),
\end{equation}
using the fact that the minimum stepsize between $\ell$ and $k\ge2$ is $\theta k^{-\frac{1}{2}}$. We also estimate
\begin{equation}\label{cor:rate:conv:const:eq2}	
\mathsf{a}_\ell^k=\sum_{i=\ell}^{k}\alpha_i^2 =
\sum_{i=\ell}^k\frac{\theta^2}{i}
\le \frac{\theta^2(k-\ell+1)}{\ell}.
\end{equation}
From \eqref{cor:rate:conv:const:eq1}-\eqref{cor:rate:conv:const:eq2} we have
\begin{equation}\label{cor:rate:conv:const:eq3}
\frac{\mathsf{a}_\ell^k}{\mathsf{S}_\ell^k}\le
\frac{\theta\sqrt{k}}{\ell}
\le\frac{\theta r^{-1}}{\sqrt{k}},\quad\quad
\frac{\mathsf{a}_\ell^k}{\mathsf{Z}_\ell^k}\le
\frac{\theta^2}{\beta(2-\beta)\ell}
\le\frac{\theta^2 r^{-1}}{\beta(2-\beta)k},
\end{equation}
\begin{equation}\label{cor:rate:conv:const:eq4}
\frac{1}{\mathsf{S}_\ell^k}\le
\frac{\theta^{-1}\sqrt{k}}{k-\ell+1}
\le\frac{\theta^{-1}(1-r)^{-1}}{\sqrt{k}},\quad\quad
\frac{1}{\mathsf{Z}_\ell^k}
\le\frac{(1-r)^{-1}}{\beta(2-\beta)k},
\end{equation}
using $\ell\ge rk$ and $k-\ell+1\ge(1-r)k$. Item (c) follows from 
\eqref{cor:rate:conv:const:eq3}-\eqref{cor:rate:conv:const:eq4}, Theorem \ref{thm:rate:conv:bounded}(b) and (d),
the definitions of $\mathsf{\widehat J}_\beta[\cdot]$ and
$\mathsf{E}_\ell^k[\cdot]$ and the fact that $\beta(2-\beta)\mathsf{b}^{k}_{\ell}=\mathsf{a}^k_\ell$.
\end{proof}

We make a remark concerning the \emph{robustness} of the stepsize sequence in Corollaries \ref{cor:rate:conv}, 
\ref{cor:rate:conv:bounded} and \ref{cor:rate:conv:const:step} in the spirit of Nemirovski et al. \cite{nem3}. The stepsizes presented 
above are robust in the sense that the knowledge of $L$ is not required 
and does not interrupt the advance of the method. Also, a scaling of $\theta$ in the stepsize implies a scaling 
in the  convergence rate which is linear in $\max\{\theta,\theta^{-1}\}$ or $\max\{\theta^2,1\}$. 
Note that these properties hold true in the case of an unbounded operator with approximate projections.

We close this section by showing that, in the case of stochastic approximation, the weak sharpness property 
implies that after a finite number of iterations an auxiliary stochastic program with linear objective 
solves the original variational inequality. This recovers a similar property satisfied in the deterministic setting 
(see Marcotte and Zhu \cite{marcotte}, Theorem 5.1). We estimate the minimum number of iterations  
in terms of the condition number $L/\rho^2$, the variance and the distance of $x^0$ to the solution set, 
when $T$ is $L$-Lipschitz continuous.

We emphasize that the auxiliary problem is still stochastic, an hence, even when $X$ is a polyhedron, 
we cannot conclude that a finite number of steps of a linear programming algorithm  will be enough for finding 
a solution. It is not clear that switching to an SAA method for stochastics LP's will be
computationally more eficcient than continuing with our algorithm. Such issue requires extensive computational
experimentation, which we intend to perform in a future work. Thus, for the time being we look at the next corollary as a possibly interesting theoretical property of weak-sharp SVI's, i.e. an extension to the stochastic setting 
of Theorem 4.2 of \cite{marcotte}.
  
\begin{corollary}[A stochastic optimization problem]\label{cor:aux:problem}
Suppose that $T$ is $(L,\delta)$-H\"older continuous with $\delta\in(0,1]$ and
\begin{enumerate}
\item the assumptions of Corollary \ref{cor:rate:conv} hold with $\delta=1$ (unbounded case), or
\item the assumptions of Corollary \ref{cor:rate:conv:bounded} hold (bounded case).
\end{enumerate}
Then, there exists $\mathsf{V}>0$, such that for all $k\ge2$ with
$
\frac{k}{(\ln k)^{1+\lambda}}>\left(\frac{\mathsf{V}L^{1/\delta}}{\rho^{1+1/\delta}}\right)^2,
$
we have 
$$
\argmin_{x\in X}\langle \esp\left[F(v,\widehat x^k)\right],x\rangle\subset X^*.
$$
Moreover, under condition 1,  
$$
\mathsf{V}:=\frac{\max\{\theta,\theta^{-1}\}}{2}\cdot
\inf_{x^*\in X^*}\left\{\Vert x^0-x^*\Vert^2+
\mathsf{J}_\beta(x^*,k_0,1)\left[2+\frac{1}{2(\ln 2)^{1+\lambda}}+\frac{1}{\lambda(\ln 2)^\lambda}\right]\right\},
$$
while, under condition 2,
$$
\mathsf{V}:=\frac{\max\{\theta,\theta^{-1}\}}{2}\cdot
\left\{\dist(x^0,X^*)^2+
\mathsf{\widehat J}_\beta[1]\left[2+\frac{1}{2(\ln 2)^{1+\lambda}}+\frac{1}{\lambda(\ln 2)^\lambda}\right]\right\}.
$$
\end{corollary}
\begin{proof}
Call $\bar x^k:=\Pi_{X^*}(\widehat x^k)$. By the choice of $k$, the definition of $\mathsf{V}$ 
and Corollaries \ref{cor:rate:conv}(b) and \ref{cor:rate:conv:bounded}(a), we have 
\begin{equation}\label{cor:aux:problem:eq1}
\esp\left[\Vert \widehat x^k-\bar x^k\Vert\right]=\esp\left[\dist(\widehat x^k,X^*)\right]<(\rho/L)^{\frac{1}{\delta}}.
\end{equation}
From the H\"older-continuity of $T$,
\begin{equation}\label{lipschitz.aux}
\left\Vert \esp[T(\widehat x^k)]-\esp[T(\bar x^k)]\right\Vert \le
\esp\left[\left\Vert T(\widehat x^k)-T(\bar x^k)\right\Vert\right] \le  
L\esp\left[\Vert \widehat x^k-\bar x^k\Vert^\delta\right] \le L\esp\left[\Vert \widehat x^k-\bar x^k\Vert\right]^{\delta}<\rho,
\end{equation}
using Jensen's inequality in the first inequality, H\"older's inequality in third inequality and
\eqref{cor:aux:problem:eq1} in last inequality. 

From Proposition \ref{weak.sharp.equivalence}, Assumption \ref{weakly.sharp.hyphotesis} 
and the equivalence between \eqref{weak.sharp.geom2} and \eqref{weak.sharp.aux}, we get
that the Euclidean ball of center $-T(\bar x^k)$ and radius $\rho$ is contained in 
$
\bigcap_{x\in X^*}[\tang_{X}(x)\cap\polar_{X^*}(x)]^\circ.
$
By the convexity of the ball and Jensen's inequality, we have
\begin{equation}\label{cor:aux:problem:eq2}
-\esp\left[T(\bar x^k)\right]+\rho B(0,1)\subset\bigcap_{x\in X^*}[\tang_{X}(x)\cap\polar_{X^*}(x)]^\circ.
\end{equation}
From \eqref{lipschitz.aux} and \eqref{cor:aux:problem:eq2}  we get that
$-\esp[T(\widehat x^k)]\in\interior\big(\bigcap_{x\in X^*}[\tang_{X}(x)\cap\polar_{X^*}(x)]^\circ\big)$. 
Hence 
we conclude  
from Theorem \ref{marcotte.thm4.2} that
\begin{equation}\label{cor:aux:problem:eq3}
\argmin_{x\in X}\langle \esp[T(\widehat x^k)],x\rangle\subset X^*.
\end{equation}
Finally, we observe that
$
\esp\left[T(\widehat x^k)\right]=\esp\left[\esp\left[F(v,\widehat x^k)\big|\alg_k\right]\right]=\esp[F(v,\widehat x^k)],
$
using Assumption \ref{unbiased}, $\widehat x^k\in\alg_k$ and  
the property $\esp[\esp[\cdot|\alg_k]]=\esp[\cdot]$.
The results follows from $\esp\left[T(\widehat x^k)\right]=\esp[F(v,\widehat x^k)]$ and \eqref{cor:aux:problem:eq3}.
\end{proof}

\section{An incremental projection method with regularization for Cartesian SVI}\label{s3}

In this section we shall study incremental projections, dropping the weak sharpness property of Section \ref{s2} and assuming 
only monotonicity of the operator. Additionally, we analyze the distributed version of the method, which includes 
the centralized case ($m=1$) in particular. For the sake of clarity, we present next the Cartesian and constraint structures in such framework.

\subsection{Cartesian structure}\label{ss3.0.1}

We assume in this section that the stochastic variational inequality \eqref{VI}-\eqref{expected} 
has a Cartesian structure. We consider 
the decomposition 
$
\re^n=\re^{n_1}\times\cdots\times\re^{n_m},
$ 
with  $n=n_1+\ldots+n_m$ and furnish this Cartesian space 
with the standard inner product 
$
\langle x,y\rangle=\sum_{j=1}^m\langle x_j,y_j\rangle,
$ 
for $x=(x_1,\ldots,x_m)$ and $y=(y_1,\ldots,y_m)$. We suppose that the feasible set $X\subset\re^n$ has the form
$
X=X^1\times\cdots\times X^m,
$
where each component $X^j\subset\re^{n_j}$ is a closed and convex set for $j\in[m]$. We emphasize that 
the orthogonal projection under a Cartesian structure is simple: 
for $x=(x_1,\ldots,x_m)\in\re^n$ and $Y=Y^1\times\ldots\times Y^m\subset\re^n$ with $x_j\in\re^{n_j}$ and $Y^j\subset\re^{n_j}$, 
we have
$
\Pi_Y(x)=(\Pi_{Y^1}(x_1),\ldots,\Pi_{Y^m}(x_m)).
$

We assume the random variable takes the form $v=(v_1,\ldots,v_m):\Omega\rightarrow\Xi$, where $v_j$ corresponds 
to the randomness of agent $j$, the random operator $F:\Xi\times\re^n\rightarrow\re^n$ has the form 
$
F(v,x)=(F_1(v_1,x),\ldots,F_m(v_m,x)),
$
with $F_j(v_j,\cdot):\re^n\rightarrow\re^{n_j}$ for $j\in[m]$. 
From \eqref{expected}, the mean operator has the form 
$T=(T_1,\ldots,T_m)$ with $T_j(x)=\esp[F_j(v_j,x)]$ for $j\in[m]$. Such framework includes stochastic multi-agent optimization 
and stochastic Nash equilibrium problems as special cases.

\subsection{Constraint structure}\label{ss3.0.2}

In order to exploit the use of incremental projections (as in Section \ref{s2}) in the Cartesian framework, 
we assume from now on that for $j\in[m]$, each Cartesian component $X^j$ of $X=X^1\times\ldots\times X^m$ has the following form:
\begin{equation}\label{constraint:set:cart}
X^j=X_0^j\cap\left(\cap_{i\in\constr_j}X_i^j\right),
\end{equation}
where $\{X_0^j\}\cup\{X_i^j:i\in\constr_j\}$ is a collection of closed and convex subsets of $\re^{n_j}$. 
Given $j\in[m]$, we assume that the projection operator onto $X_0^j$ is computationally easy to evaluate, 
and that for every $i\in\constr_j$, $X_i^j$ is representable in $\re^{n_j}$ as 
\begin{equation}\label{assump:represent:cart}
X_i^j=\{x\in\re^{n_j}:g_{i}(j|x)\le0\},
\end{equation}
for some convex function $g_{i}(j|\cdot):\re^{n_j}\rightarrow\re\cup\{\infty\}$ with domain $\dom g_{i}(j|\cdot)\subset X_0^j$. 
We denote the positive part of $g_{i}(j|\cdot)$ as 
$
g^+_{i}(j|x):=\max\{g_{i}(j|x),0\},
$
for $x\in\re^{n_j}$. We also assume that, for every $i\in\constr_j$, the subgradients of $g^+_{i}(j|\cdot)$ 
at points $x\in X_0^j-X_i^j$ are easily computable and that $\{\partial g^+_{i}(j|\cdot):i\in\constr_j\}$ 
is uniformly bounded over $X_0^j$, i.e., there exists $C_g^j>0$ such that 
\begin{equation}\label{assump:subgrad:bound:cart}
\Vert d\Vert\le C_g^j. 
\end{equation}
for all $x\in X_0^j$, all $i\in\constr_j$ and all $d\in\partial g^+_{i}(j|x)$. 

\subsection{Statement of the algorithm}\label{ss3.1}

For problems endowed with the Cartesian structure and the constraint structure of Sections \ref{ss3.0.1} and \ref{ss3.0.2}, 
our method advances in a distributed fashion for each Cartesian component $j\in[m]$, as in the incremental projection method 
\eqref{algo.extra.approx.eq1}-\eqref{algo.extra.approx.eq2} with an additional 
Tykhonov regularization (in order to cope with the plainly monotone case). Precisely, fix the Cartesian component $j\in[m]$. 
In a first stage, given the current iterate $x^k$, the method advances in the direction $-F_j(v^k_j,x^k)-\epsilon_{k,j}x^k_j$ 
with stepsize $\alpha_{k,j}$, after taking the sample $v^k_j$ and choosing the regularization parameter $\epsilon_{k,j}>0$, 
producing an auxiliary iterate $y^k_j$. In the second stage, a soft constraint $X_{\omega_{k,j}}^j$ is randomly chosen with the random 
control $\omega_{k,j}\in\constr_j$, and the method advances in the direction opposite to a subgradient of $g^+_{\omega_{k,j}}(j|\cdot)$ 
at the point $y^k_j$ with a stepsize $\beta_{k,j}$, producing the next iterate $x^{k+1}_j$. The iterates are collected in $x^{k+1}$ 
and the method continues. Formally, the method takes the form:

\begin{algo}[Regularized incremental projection method: distributed case]
\quad
\begin{enumerate}
\item{\bf Initialization:} Choose the initial iterate $x^0\in\mathbb{R}^n$, 
the stepsize sequences $\alpha^k=(\alpha_{k,1},\ldots,\alpha_{k,m})\in(0,\infty)^m$ and 
$\beta^k=(\beta_{k,1},\ldots,\beta_{k,m})\in(0,2)^m$, the regularization 
sequence $\epsilon^k=(\epsilon_{k,1},\ldots,\epsilon_{k,m})\in(0,\infty)^m$,
the random control sequence $\omega^k=(\omega_{k,1},\ldots,\omega_{k,m})\in\constr_1\times\ldots\times\constr_m$ 
and the operator sample sequence $v^k=(v^k_1\,\ldots,v^k_m)$.
\item {\bf Iterative step:} Given $x^k=(x^k_1,\ldots,x^k_m)$, define, for each $j\in[m]$, 
\begin{eqnarray}
y_j^{k}&=&\Pi_{X_0^j}\left[x_j^k-\alpha_{k,j}\left(F_j(v^k_j,x^k)+\epsilon_{k,j} x_j^k\right)\right]\label{algo.tikhonov.approx},\\
x^{k+1}_j&=&\Pi_{X_0^j}\left[y^k_j-\beta_{k,j}\frac{g^+_{\omega_{k,j}}(j|y^k_j)}{\Vert d^k_j\Vert^2}d^k_j\right],\label{algo.tikhonov.approx:eq2}
\end{eqnarray}
where $d^{k}_j\in\partial g^+_{\omega_{k,j}}(j|y^k_j)-\{0\}$ if $g_{\omega_{k,j}}(j|y^k_j)>0$, and 
$d^k_j=d$ for any $d\in\re^{n_j}-\{0\}$ if $g_{\omega_{k,j}}(j|y^k_j)\le0$.
\end{enumerate}
\end{algo}

The first stage \eqref{algo.tikhonov.approx} of the iterative step can be written compactly as
$$
y^{k}=\Pi_{X_0}\left[x^k-D(\alpha_{k})\cdot\left(T(x^k)+\varsigma^k+D(\epsilon_{k}) x^k\right)\right],
$$
where
$
X_0:=X_0^1\times\ldots\times X_0^m,
$
\begin{equation}\label{def:varsigma}
\varsigma^k_j:=F_j(v^k_j,x^k)-T_j(x^k),\quad j\in[m],
\end{equation}
with $\varsigma^k:=(\varsigma_j^k)_{j=1}^m$ and $D(\alpha)$ denotes the block-diagonal matrix in $\re^{n\times n}$ defined as
$$
D(\alpha):=
\begin{bmatrix}
\alpha_1 I_{n_1} & 0 & 0 \\
 & \ddots & \\
0 & 0 & \alpha_m I_{n_m}
\end{bmatrix},
$$
with $\alpha=(\alpha_1,\ldots,\alpha_m)\in\re^m_{>0}$, 
and $I_{n_j}\in\re^{n_j\times n_j}$ denoting the identity matrix  
for each $j\in[m]$. 

\subsection{Discussion of the assumptions}\label{ss3.2}

We consider the natural filtration
$$
\alg_k=\sigma(\omega^0,\ldots,\omega^{k-1},v^0,\ldots,v^{k-1}).
$$

\begin{assump}\label{mono.lips.unbiased.tik}
We request  Assumptions \ref{existence}-\ref{unbiased} and Assumption \ref{lipschitz}(i).
\end{assump}
 
In this section we avoid the weak sharpness property assumed in Section \ref{s2}. We now state the assumptions concerning 
the approximate projections which accommodate the Cartesian structure. 
In simple terms, we require each Cartesian component $X^j$ given by \eqref{constraint:set:cart} to satisfy 
Assumption \ref{assump:constraint:reg} of Section \ref{s2}. This is formally stated in Assumption \ref{assump:constraint:reg:cart}. 
Also, the agents' stepsizes and regularization sequences require a partial 
coordination specified in Assumption \ref{approx.step.tik}. 
\begin{assump}[Constraint sampling and regularity]\label{assump:constraint:reg:cart}
For each $j\in[m]$, there exists $c^j>0$, such that a.s. for all $k\in\mathbb{N}$ and all $x\in X^j_0$,
$$
\dist(x,X^j)^2\le c^j\cdot\esp\left[g^+_{\omega_{k,j}}(j|x)\Big|\alg_k\right].
$$
\end{assump}
We observe that Assumption \ref{assump:constraint:reg:cart} requires a sampling coordination between 
the control sequences $\{\omega_{k,j}\}_{k=0}^\infty$ for $j\in[m]$, since the filtration $\alg_k$ accumulates the history 
of the control sequence of every Cartesian component. The next lemma shows that this assumption is immediately satisfied 
if each agent has a metric regular decision set and the constraint sampling is independent between agents and uniform i.i.d. for each agent.

\begin{lemma}[Sufficient condition for Assumption \ref{assump:constraint:reg:cart}]\label{lemma:sufficience:reg:cart}
Suppose that $\{v^k\}$ and $\{\omega_k\}$ are independent sequences, $\omega_{k,1},\ldots,\omega_{k,m}$ 
are independent for each $k$ and the following conditions hold: for each $j\in[m]$, $|\mathcal{I}^j|<\infty$ and
\begin{itemize}
\item[(i)] The sequence $\{\omega_{k,j}\}_{k=0}^\infty$ is an i.i.d. sample of a random variable $\omega^j$ 
taking values on $\mathcal{I}^j$ such that for some $\lambda^j>0$,
$$
\prob\left(\omega^j=i\right)\ge\frac{\lambda^j}{|\constr^j|},\quad\forall i\in\constr^j,
$$
\item[(ii)] The set $X^j$ is \emph{metric regular}: there is $\eta^j>0$ such that for all $x\in X_0^j$,
$$
\dist(x,X^j)^2\le\eta^j\max_{i\in\constr^j}[g_i^+(j|x)]^2.
$$
\end{itemize}
Then Assumption \ref{assump:constraint:reg:cart} holds with $c^j=\frac{\eta^j|\constr^j|}{\lambda^j}$ for $j\in[m]$.
\end{lemma}
\begin{proof}
Since $\{v^k\}$ and $\{\omega_k\}$ are independent, $\{\omega_k\}_{k=0}^\infty$ is independent and $\omega_{k,1},\ldots,\omega_{k,m}$ 
are independent for each $k$, it follows that for all $k\in\mathbb{N}_0$ and $j\in[m]$, $\omega_{k,j}$ is independent of $\alg_k$. 
The remainder of the proof follows the proof line of Lemma \ref{lemma:sufficience:reg}.
\end{proof}

\begin{assump}[Partial coordination of stepsizes and regularization sequences] \label{approx.step.tik}
For $j\in[m]$, consider the stepsize sequences $\{\alpha_{k,j}\}_{k=0}^\infty$ and $\{\beta_{k,j}\}_{k=0}^\infty$ 
and the regularization sequence $\{\epsilon_{k,j}\}_{k=0}^\infty$ in Algorithm \eqref{algo.tikhonov.approx}-\eqref{algo.tikhonov.approx:eq2}. 
Without loss of generality, for $j\in[m]$ we add the term $\epsilon_{-1,j}$ to the regularization sequence. 
We use the notation $u_{k,\min}:=\min_{j\in[m]}u_{k,i}$, $u_{k,\max}:=
\max_{j\in[m]}u_{k,j}$ for $u\in\{\alpha,\beta,\epsilon\}$, $\Delta_k:=\alpha_{k,\max}-\alpha_{k,\min}$, 
$\Gamma_k:=\epsilon_{k-1,\max}-\epsilon_{k,\min}$ and $\mathsf{B}_k:=\beta_{k,\min}(2-\beta_{k,\max})$.
We then assume that $0<\beta_{k,\min}\le\beta_{k,\max}<2$ and 
\begin{itemize}
\item[(i)] For each $j\in[m]$, $\{\epsilon_{k,j}\}_{k=-1}^\infty$ is a decreasing positive sequence converging to zero.
\item[(ii)] 
$	
\lim_{k\rightarrow\infty}\frac{\alpha_{k,\max}^2}{\alpha_{k,\min}\epsilon_{k,\min}}=0,
$ 
$
\lim_{k\rightarrow\infty}\frac{\alpha_{k,\max}^2}{\mathsf{B}_k\alpha_{k,\min}\epsilon_{k,\min}}=0,
$
$
\lim_{k\rightarrow\infty}\frac{\Delta_k}{\alpha_{k,\min}\epsilon_{k,\min}}=0
$
and 
$
\lim_{k\rightarrow\infty}\alpha_{k,\min}\epsilon_{k,\min}=0.
$
\item[(iii)] $\sum_{k=0}^\infty\alpha_{k,\min}\epsilon_{k,\min}=\infty$.
\item[(iv)] 
$
\sum_{k=0}^\infty\alpha_{k,\max}^2<\infty,
$
$
\sum_{k=0}^\infty\frac{\alpha_{k,\max}^2}{\mathsf{B}_k}<\infty,
$
$
\sum_{k=0}^\infty\left(\frac{\Gamma_k}
{\epsilon_{k,\min}}\right)^2\left(1+\alpha_{k,\min}^{-1}\epsilon_{k,\min}^{-1}\right)<\infty
$
and
\begin{equation}\label{extra.condition}
\sum_{k=0}^\infty\frac{\Delta_k^2}{\alpha_{k,\min}\epsilon_{k,\min}}<\infty.
\end{equation}
\item[(v)]
$
\lim_{k\rightarrow\infty}\frac{\Gamma_k^2}
{\epsilon_{k,\min}^3\alpha_{k,\min}}\left(1+\alpha_{k,\min}^{-1}\epsilon_{k,\min}^{-1}\right)=0.
$ 
\end{itemize} 
\end{assump}

Assumption \ref {approx.step.tik} contains usual conditions on the regularization parameters of Tykhonov algorithms and 
on the stepsize for SA algorithms, with certain coordination across stepsizes 
and regularization parameters.  Assumption \ref{approx.step.tik} includes 
Assumption 2 in \cite{Uday} with the addition of \eqref{extra.condition}, due to the use of approximate projections 
(in addition to asynchronous stepsizes).\footnote{We observe that this condition is trivially satisfied with synchronous stepsizes, 
i.e., $\alpha_{k,j}=\alpha_{k,\ell}$ for all $k,j,\ell$.} Next we show that Assumption \ref{approx.step.tik} 
is satisfied by explicit stepsizes and regularization parameters.
\begin{corollary}[Asynchronous stepsizes and regularization parameters]\label{cor:step:reg}
Take $\delta\in(0,\frac{1}{2})$ and real numbers $\underline{C}\le\overline{C}$, $\underline{D}\le\overline{D}$. 
The following stepsizes and regularization parameters satisfy Assumption \ref{approx.step.tik}: 
for any $j\in[m]$ and $k\in\mathbb{N}_0$, take $C_j\in[\underline{C},\overline{C}]$, $D_j\in[\underline{D},\overline{D}]$, $\beta_j\in(0,2)$ and
$$
\alpha_{k,j}=O(1)\frac{1}{(k+C_j)^{\frac{1}{2}+\delta}},\quad\quad
\epsilon_{k,j}=O(1)\frac{1}{(k+D_j)^{\frac{1}{2}-\delta}},\quad\quad\beta_{k,j}\equiv\beta_j.
$$
\end{corollary}
\begin{proof}
Except for condition \eqref{extra.condition}, all other conditions in Assumption \ref{approx.step.tik} 
are proved in Lemma 4 of \cite{Uday}. We proceed with the proof of \eqref{extra.condition}. Set $u_{\max}:=\max_{1\le i\le m}{u_i}$, 
$u_{\min}=\min_{1\le i\le m}{u_i}$ for $u\in\{C,D\}$ and $a:=1/2+\delta$, $b:=1/2-\delta$. The claim is proved by showing that
$$
\frac{(\alpha_{k,\max}-\alpha_{k,\min})^2}{\alpha_{k,\min}\epsilon_{k,\min}}\sim
\frac{\left[\left(k+C_{\min}\right)^{-a}-\left(k+C_{\max}\right)^{-a}\right]^2}{\left(k+C_{\max}\right)^{-a}\left(k+D_{\max}\right)^{-b}}
=\frac{\left[\left(1-\frac{C_{\max}-C_{\min}}{k+C_{\max}}\right)^{-a}-1\right]^2}{\left(k+C_{\max}\right)^{a}\left(k+D_{\max}\right)^{-b}}
$$
$$
=\frac{\left[1+a\frac{C_{\max}-C_{\min}}{k+C_{\max}}+O\left(\frac{1}{k^2}\right)-1\right]^2}{k^{a-b}\left(1
+\frac{C_{\max}}{k}\right)^{a}\left(1+\frac{D_{\max}}{k}\right)^{-b}}\\
=O\left(\frac{1}{k^{2+2\delta}}\right).
$$
\end{proof}

\subsection{Convergence analysis}\label{ss3.4}
We present next our convergence result for method \eqref{algo.tikhonov.approx}-\eqref{algo.tikhonov.approx:eq2}. We shall need two lemmas.

\begin{lemma}[Eventual strong-monotonicity]\label{lema1.tik}
Consider Assumption \ref{mono.lips.unbiased.tik}. Define $H_k:=D(\alpha_k)\cdot(T+D(\epsilon_k))$ and 
$
\sigma_k=\alpha_{k,\min}\epsilon_{k,\min}-
L(\alpha_{k,\max}-\alpha_{k,\min}).
$
Then for all $y,x\in\re^n$ and $k\in\mathbb{N}$,
$
\langle H_k(y)-H_k(x),y-x\rangle\ge\sigma_k\Vert y-x\Vert^2.
$
\end{lemma}
\begin{proof}
We consider the decomposition
\begin{equation}
\label{lema1.tik.eq1}
\langle H_k(y)-H_k(x),y-x\rangle =
\langle D(\alpha_k)\cdot(T(y)-T(x)),y-x\rangle+\langle D(\alpha_k)D(\epsilon_k)(y-x),y-x\rangle.
\end{equation}

Concerning the second term  in the right hand side of 
\eqref{lema1.tik.eq1},
if $D_k$ is the diagonal matrix with entries $(\alpha_1\epsilon_1,\ldots,\alpha_m\epsilon_m)$, then
\begin{equation}\label{lema1.tik.eq2}
\langle D(\alpha_k)D(\epsilon_k)(y-x),y-x\rangle=
\langle D_k(y-x),y-x\rangle\ge\alpha_{k,\min}\epsilon_{k,\min}\Vert y-x\Vert^2.
\end{equation}

The first term in the right hand side of \eqref{lema1.tik.eq1} is equal to
\begin{eqnarray}\label{lema1.tik.eq3}
\sum_{i=1}^m\alpha_{k,i}\langle T_i(y)-T_i(x),y_i-x_i\rangle &=&
\alpha_{k,\min}\sum_{i=1}^m\langle T_i(y)-T_i(x),y_i-x_i\rangle\nonumber\\
&&+\sum_{i=1}^m(\alpha_{k,i}-\alpha_{k,\min})\langle T_i(y)-T_i(x),y_i-x_i\rangle.
\end{eqnarray}
The first term in the right hand side of \eqref{lema1.tik.eq3} 
is nonnegative by  monotonicity of $T$. For the second term 
in the right hand side of \eqref{lema1.tik.eq3}, 
we have
\begin{eqnarray}
\sum_{i=1}^m(\alpha_{k,i}-\alpha_{k,\min})\langle T_i(y)-T_i(x),y_i-x_i\rangle &\ge &-\sum_{i=1}^m(\alpha_{k,i}-\alpha_{k,\min})\Vert T_i(y)-T_i(x)\Vert\Vert y_i-x_i\Vert\nonumber\\
&\ge &-(\alpha_{k,\max}-\alpha_{k,\min})\sum_{i=1}^m\Vert T_i(y)-T_i(x)\Vert\Vert y_i-x_i\Vert\nonumber\\
&\ge &-(\alpha_{k,\max}-\alpha_{k,\min})\Vert T(y)-T(x)\Vert\Vert y-x\Vert\nonumber\\
&\ge &-(\alpha_{k,\max}-\alpha_{k,\min})L\Vert y-x\Vert^2,\label{lema1.tik.eq4}
\end{eqnarray}
using Cauchy-Schwartz inequality in the first inequality, H\"older-inequality 
in the third one and Lipschitz continuity of $T$ in the last one. 
The result follows from \eqref{lema1.tik.eq1}-\eqref{lema1.tik.eq4}.
\end{proof}

We will use the following result, proved in Lemma 3 of Koshal et al. \cite{Uday}: 
\begin{lemma}[Properties of the Tykhonov sequence]\label{tykhonov.sequence}
Assume that $X\subset\re^n$ is convex and closed, that the operator $T:\re^n\rightarrow\re^n$ is continuous and monotone over $X$ and  
that Assumption \ref{existence} hold. Assume also that the positive sequences $\{\epsilon_{k,j}\}_{k=-1}^{\infty}$ for $j\in[m]$ 
decrease to $0$ and satisfy $\limsup_{k\rightarrow\infty}\frac{\epsilon_{k,\max}}{\epsilon_{k,\min}}<\infty$, 
with $\epsilon_{k,\max}:=\max_{j\in[m]}\epsilon_{k,j}$ and $\epsilon_{k,\min}:=\min_{j\in[m]}\epsilon_{k,j}$. 
Denote by $t^k$ the solution of \emph{VI}$(T+D(\epsilon_k),X)$. Then 
\begin{itemize}
\item[(i)] $\{t^k\}$ is bounded and all cluster points of $\{t^k\}$ belong to $X^*$.
\item[(ii)] The following inequality holds for all $k\ge1$:
$$
\Vert t^k-t^{k-1}\Vert\le\frac{\epsilon_{k-1,\max}-\epsilon_{k,\min}}{\epsilon_{k,\min}}M_t,
$$
where 
\begin{equation}\label{def:Mt}
M_t:=\sup_{k\in\mathbb{N}_0}\Vert t^k\Vert.
\end{equation}
\item[(iii)] If $\limsup_{k\rightarrow\infty}\frac{\epsilon_{k,\max}}{\epsilon_{k,\min}}\le1$ 
then $\{t^k\}$ converges to the least-norm solution in $X^*$.
\end{itemize}
\end{lemma}

Recalling \eqref{def:B}, we define 
\begin{equation}\label{def:bt}
B_t:=\sup_{k\in\mathbb{N}_0}B(t^k),
\end{equation}
which is a finite quantity, because $\{t^k\}$ is bounded and $B(\cdot)$ is a locally bounded function. 
We also define the following constants for given $\tau>1$:
\begin{equation}\label{convergence:tik:hag}
H_{k,\tau}:=4\left[1+\tau\beta_{k,\max}(2-\beta_{k,\min})\right],\mbox{ }C:=\max_{j\in [m]}c^j,\mbox{ } 
C_g:=\min_{j\in[m]}C_g^j,\mbox{ } G_\tau:=\frac{C C_g^2\tau}{(\tau-1)},\mbox{ } A_{k,\tau}:=\frac{\mathsf{B}_k}{G_\tau}.
\end{equation}
Next we prove the asymptotic convergence of method \eqref{algo.tikhonov.approx}-\eqref{algo.tikhonov.approx:eq2}.

\begin{theorem}[Asymptotic convergence]\label{convergence.tik}
If Assumptions \ref{mono.lips.unbiased.tik}-\ref{approx.step.tik} hold, then the 
method \eqref{algo.tikhonov.approx}-\eqref{algo.tikhonov.approx:eq2} generates a sequence $\{x^k\}$ such that:
\begin{itemize}
\item[(i)] if $\limsup_{k\rightarrow\infty}\frac{\epsilon_{k,\max}}{\epsilon_{k,\min}}<\infty$, 
then almost surely $\{x^k\}$ is bounded and all cluster points of $\{x^k\}$ belong to the solution set $X^*$,
\item[(ii)] if $\limsup_{k\rightarrow\infty}\frac{\epsilon_{k,\max}}{\epsilon_{k,\min}}\le1$, 
then almost surely $\{x^k\}$ converges to the least-norm solution in $X^*$.
\end{itemize}
\end{theorem}
\begin{proof}
In the sequel we denote by $\{t^k\}$ the Tykhonov sequence of Lemma \ref{tykhonov.sequence}. Let $x=(x_j)_{j=1}^m\in X$. 
We claim that for all $\tau>1$, $j\in[m]$, $k\in\mathbb{N}$, 
$$	
\Vert x^{k+1}_j-x_j\Vert^2\le\Vert x^k_j-x_j\Vert^2-2\alpha_{k,j}\langle x^k_j-x_j,F_j(v^k_j,x^k)+\epsilon_{k,j}x^k_j\rangle+
$$
\begin{equation}\label{lemma:recursion:tyk:eq1}
\left[1+\tau\beta_{k,j}(2-\beta_{k,j})\right]\alpha_{k,j}^2\left\Vert F_j(v^k_j,x^k)+
\epsilon_{k,j}x^k_j\right\Vert^2-
\frac{\beta_{k,j}(2-\beta_{k,j})}{(C_g^j)^2}\left(1-\frac{1}{\tau}\right)\left(g^+_{\omega_{k,j}}(j|x^k_j)\right)^2.
\end{equation}	
Indeed, in view of \eqref{algo.tikhonov.approx}-\eqref{algo.tikhonov.approx:eq2} and  
$x_j\in X^j\subset X_0^j\cap X_{\omega_{k,j}}^j$, we can invoke Lemma \ref{lemma:feas:step} 
with $g:=g_{\omega_{k,j}}(j|\cdot)$, $x_1:=x^k_j$, $x_2:=x^{k+1}_j$, $x_0:=x_j$, 
$\alpha:=\alpha_{k,j}$, $u:=F_j(v^k_j,x^k)+\epsilon_{k,j}x^k_j$, $y:=y^k_j$, $\beta:=\beta_{k,j}$ 
and $d:=d^k_j$ obtaining \eqref{lemma:recursion:tyk:eq1}.

We define for $j\in[m]$,
\begin{equation}\label{equation:def:varsigma:z}
z^k_j:=x^k_j-\alpha_{k,j}(F_j(v^k_j,x^k)+\epsilon_{k,j}x^k_j),
\end{equation}
with $z^k:=(z_j^k)_{j=1}^m$. We use the definitions in \eqref{convergence:tik:hag} and sum the inequalities 
in \eqref{lemma:recursion:tyk:eq1}  with $j$ between $1$ and $m$, getting
$$
\Vert x^{k+1}-x\Vert^2\le\Vert x^k-x\Vert^2+2\sum_{j=1}^m\langle x_j-x^k_j,x^k_j-z^k_j\rangle+
$$
\begin{equation}\label{eq3.2.tik}
\left[1+\tau\beta_{k,\max}(2-\beta_{k,\min})\right]\left\Vert z^k-
x^k\right\Vert^2-
\frac{\beta_{k,\min}(2-\beta_{k,\max})}{C_g^2}\left(1-\frac{1}{\tau}\right)\sum_{j=1}^m\left(g^+_{\omega_{k,j}}(j|x^k_j)\right)^2.
\end{equation}

Concerning the second term in the right hand side of \eqref{eq3.2.tik}, we have
\begin{equation}\label{lema.eq4.tik}
\langle x_j-x_j^k,x_j^k-z_j^k\rangle=
\alpha_{k,j}\langle x_j-x_j^k,F_j(v^k_j,x^k)+\epsilon_{k,j} x_j^k\rangle=
\alpha_{k,j}\langle x_j-x_j^k,T_j(x^k)+\varsigma_j^k+\epsilon_{k,j} x_j^k\rangle,
\end{equation}
using the definitions in \eqref{def:varsigma} and \eqref{equation:def:varsigma:z}.

We now analyze the third term in the right hand side of \eqref{eq3.2.tik}. 
The triangular inequality and the inequality $(\sum_{i=1}^4 a_i)^2\le4\sum_{i=1}^4a_i^2$ imply that
\begin{eqnarray}
\Vert z_j^k-x_j^k\Vert^2 &=&\alpha_{k,j}^2\Vert F_j(v^k_j,x^k)+\epsilon_{k,j}x_j^k\Vert^2\nonumber\\
&= &\alpha_{k,j}^2\Vert F_j(v^k_j,x^k)-F_j(v^k_j,x)+\epsilon_{k,j}(x_j^k-x_j)+F_j(v^k_j,x)+\epsilon_{k,j} x_j\Vert^2\nonumber\\
&\le &4\alpha_{k,j}^2\Vert F_j(v^k_j,x^k)-F_j(v^k_j,x)\Vert^2+4\alpha_{k,j}^2\epsilon_{k,j}^2\Vert x_j^k-x_j\Vert^2\nonumber\\
&&+4\alpha_{k,j}^2\Vert F_j(v^k_j,x)\Vert^2+4\alpha_{k,j}^2\epsilon_{k,j}^2\Vert x_j\Vert^2\nonumber\\
&\le &4\alpha_{k,\max}^2\Vert F_j(v^k_j,x^k)-F_j(v^k_j,x)\Vert^2+4\alpha_{k,\max}^2\epsilon_{k,\max}^2\Vert x_j^k-x_j\Vert^2\nonumber\\
&&+4\alpha_{k,\max}^2\Vert F_j(v^k_j,x)\Vert^2+4\alpha_{k,\max}^2\epsilon_{k,\max}^2\Vert x_j\Vert^2.\label{lema.eq5.tik}
\end{eqnarray}
Summing the inequalities in \eqref{lema.eq5.tik} with  $j$ between $1$ and $m$, we get 
from Assumption \ref{lipschitz}(i),
\begin{eqnarray}
\Vert z^k-x^k\Vert^2 &=
&\sum_{j=1}^m\Vert z_j^k-x_j^k\Vert^2\nonumber\\
&\le &4\alpha_{k,\max}^2\Vert F(v^k,x^k)-F(v^k,x)\Vert^2+
4\alpha_{k,\max}^2\epsilon_{k,\max}^2\Vert x^k-x\Vert^2\nonumber\\
&&+4\alpha_{k,\max}^2\Vert F(v^k,x)\Vert^2+
4\alpha_{k,\max}^2\epsilon_{k,\max}^2\Vert x\Vert^2\nonumber\\
&\le &4L(v^k)^2\alpha_{k,\max}^2\Vert x^k-x\Vert^2+
4\alpha_{k,\max}^2\epsilon_{k,\max}^2\Vert x^k-x\Vert^2\nonumber\\
&&+4\alpha_{k,\max}^2\Vert F(v^k,x)\Vert^2+
4\alpha_{k,\max}^2\epsilon_{k,\max}^2\Vert x\Vert^2.\label{lema.eq6.tik}	
\end{eqnarray}

Now we combine \eqref{eq3.2.tik}-\eqref{lema.eq6.tik}, in order to obtain
$$
\Vert x^{k+1}-x\Vert^2\le 
\big[1+H_{k,\tau}(L(v^k)^2+\epsilon_{k,\max}^2)\alpha_{k,\max}^2\big]\Vert x^{k}-x\Vert^2+
2\sum_{j=1}^m\alpha_{k,j}\langle x_j-x_j^k,T_j(x^k)+\epsilon_{k,j} x_j^k\rangle
$$
\begin{equation}\label{eq7.tik}
+2\sum_{j=1}^m\alpha_{k,j}\langle x_j-x_j^k,\varsigma_j^k\rangle+
H_{k,\tau}(\Vert F(v^k,x)\Vert^2+
\Vert x\Vert^2\epsilon_{k,\max}^2)\alpha_{k,\max}^2-CA_{k,\tau}\sum_{j=1}^m\left(g^+_{\omega_{k,j}}(j|x^k_j)\right)^2,
\end{equation}
where $H_{k,\tau}, A_{k,\tau}$ and $G_\tau$ are defined as in \eqref{convergence:tik:hag}.

The sum in the second term of the right hand side of \eqref{eq7.tik} is equal to
$$
\langle D(\alpha_k)\cdot(T+D(\epsilon_k))(x^k),x-x^k\rangle
=\langle D(\alpha_k)\cdot(T+D(\epsilon_k))(x^k)
-D(\alpha_k)\cdot(T+D(\epsilon_k))(x),x-x^k\rangle
$$
\begin{equation}\label{eq8.tik}
+\langle D(\alpha_k)\cdot(T+D(\epsilon_k))(x),x-\Pi(x^k)\rangle
+\langle D(\alpha_k)\cdot(T+D(\epsilon_k))(x),\Pi(x^k)-x^k\rangle.
\end{equation}
Recalling the  definition of $\Delta_k$ in Assumption \ref{approx.step.tik}, it follows from Lemma \ref{lema1.tik} 
that the first term in the right hand side of \eqref{eq8.tik} satisfies
\begin{equation}\label{term1.tik}
\langle D(\alpha_k)\cdot(T+D(\epsilon_k))(x^k)-D(\alpha_k)\cdot(T+D(\epsilon_k))(x),x-x^k\rangle\le
-(\alpha_{k,\min}\epsilon_{k,\min}-L\Delta_k)\Vert x^k-x\Vert^2.
\end{equation}
The second term in the right hand side of \eqref{eq8.tik} is equal to
$$
\sum_{j=1}^m\alpha_{k,j}\langle T_j(x)+\epsilon_{k,j}x_j,x_j-\Pi_{X^j}(x^k_j)\rangle =
\alpha_{k,\max}\sum_{j=1}^m\langle T_j(x)+\epsilon_{k,j}x_j,x_j-\Pi_{X^j}(x^k_j)\rangle
$$
\begin{equation}\label{term2.tik}
+\sum_{j=1}^m(\alpha_{k,j}-\alpha_{k,\max})\langle T_j(x)+\epsilon_{k,j}x_j,x_j-\Pi_{X^j}(x^k_j)\rangle.
\end{equation}
The first term in the right hand side of \eqref{term2.tik} equals 
\begin{equation}\label{term2.tik.eq1}
\alpha_{k,\max}\sum_{j=1}^m\langle T_j(x)+\epsilon_{k,j}x_j,x_j-\Pi_{X^j}(x^k_j)\rangle=\alpha_{k,\max}
\langle(T+D(\epsilon_k))(x),x-\Pi(x^k)\rangle.
\end{equation}

Regarding the second term in the right hand side of \eqref{term2.tik}, we 
use $\Pi_{X^j}(x_j)=x_j$, so that for each $\mu\in(0,1)$ we have
\begin{eqnarray}
&&\sum_{j=1}^m(\alpha_{k,j}-\alpha_{k,\max})\langle T_j(x)+\epsilon_{k,j}x_j,x_j-\Pi_{X^j}(x^k_j)\rangle\nonumber\\ 
&\le &\sum_{j=1}^m(\alpha_{k,\max}-\alpha_{k,j})\Vert T_j(x)+
\epsilon_{k,j}x_j\Vert\Vert \Pi_{X^j}(x_j)-\Pi_{X^j}(x^k_j)\Vert\nonumber\\
&\le &\Delta_k\sum_{j=1}^m(\Vert T_j(x)\Vert+\epsilon_{k,j}\Vert x_j\Vert)\Vert x_j-x^k_j\Vert\nonumber\\
&\le &\Delta_k(B(x)+\epsilon_{k,\max}\Vert x\Vert)\Vert x-x^k\Vert\nonumber\\
&=&2\frac{(B(x)+\epsilon_{k,\max}\Vert x\Vert)\Delta_k}{2\sqrt{\mu\alpha_{k,\min}\epsilon_{k,\min}}}
\cdot\sqrt{\mu\alpha_{k,\min}\epsilon_{k,\min}}\Vert x-x^k\Vert\nonumber\\
&\le &\frac{(B(x)+\epsilon_{k,\max}\Vert x\Vert)^2\Delta_k^2}{4\mu\alpha_{k,\min}\epsilon_{k,\min}}+
\mu\alpha_{k,\min}\epsilon_{k,\min}\Vert x-x^k\Vert^2,\label{term2.tik.eq2}
\end{eqnarray}
using Cauchy-Schwartz inequality in the first inequality, Lemma \ref{proj}(ii) for $\Pi_{X^i}$ in the second one, 
the fact that $\Vert T(x)\Vert\le B(x)$ in the third one, and the relation $2ab=-(a-b)^2+a^2+b^2$ in the fourth one. 
Putting together \eqref{term2.tik}-\eqref{term2.tik.eq2}, we finally get that the second term in 
the right hand side of \eqref{eq8.tik} is bounded by
\begin{eqnarray}
\langle D(\alpha_k)\cdot(T+D(\epsilon_k))(x),x-\Pi(x^k)\rangle &\le & \alpha_{k,\max}
\langle(T+D(\epsilon_k))(x),x-\Pi(x^k)\rangle\nonumber\\
&&+\frac{(B(x)+\epsilon_{k,\max}\Vert x\Vert)^2\Delta_k^2}{4\mu\alpha_{k,\min}\epsilon_{k,\min}}+
\mu\alpha_{k,\min}\epsilon_{k,\min}\Vert x^k-x\Vert^2.\label{term2.tik.eq3}
\end{eqnarray}
For  the third term in the right hand side of \eqref{eq8.tik}, we have
\begin{eqnarray}\label{term3.tik}
\langle D(\alpha_k)\cdot(T+D(\epsilon_k))(x),\Pi(x^k)-x^k\rangle &\le &
\Vert D(\alpha_k)\Vert\Vert T(x)+\epsilon_k x\Vert\Vert\Pi(x^k)-x^k\Vert\nonumber\\
&\le &\alpha_{k,\max}(B(x)+\epsilon_{k,\max}\Vert x\Vert)\dist(x^k,X).
\end{eqnarray}
Combining \eqref{term1.tik}, \eqref{term2.tik.eq3} and \eqref{term3.tik} with \eqref{eq8.tik}, we obtain
$$
2\langle D(\alpha_k)\cdot(T+D(\epsilon_k))(x^k),x-x^k\rangle\le
\Big[-2(1-\mu)\alpha_{k,\min}\epsilon_{k,\min}+2L\Delta_k\Big]\Vert x^k-x\Vert^2
$$
$$
+2\alpha_{k,\max}\langle(T+D(\epsilon_k))(x),x-\Pi(x^k)\rangle
$$
\begin{equation}\label{eq9.tik}
+\frac{(B(x)+\epsilon_{k,\max}\Vert x\Vert)^2\Delta_k^2}{2\mu\alpha_{k,\min}\epsilon_{k,\min}}
+2\alpha_{k,\max}(B(x)+\epsilon_{k,\max}\Vert x\Vert)\dist(x^k,X).
\end{equation}
We use \eqref{eq9.tik} in \eqref{eq7.tik} and finally get the following recursive relation: for all $k\in\mathbb{N}_0$ and $x\in X$,
$$
\Vert x^{k+1}-x\Vert^2 \le 
q_{k,\tau,\mu}(L(v^k))\Vert x^{k}-x\Vert^2
+H_{k,\tau}(\Vert F(v^k,x)\Vert^2+\Vert x\Vert^2\epsilon_{k,\max}^2)\alpha_{k,\max}^2
$$
$$
+\frac{(B(x)+\epsilon_{k,\max}\Vert x\Vert)^2\Delta_k^2}{2\mu\alpha_{k,\min}\epsilon_{k,\min}}
+2\alpha_{k,\max}\langle(T+D(\epsilon_k))(x),x-\Pi(x^k)\rangle
+2\langle x-x^k,D(\alpha_k)\varsigma^k\rangle
$$
\begin{equation}\label{eq10.tik}
+2\alpha_{k,\max}(B(x)+\epsilon_{k,\max}\Vert x\Vert)\dist(x^k,X)-CA_{k,\tau}\sum_{j=1}^m\left(g^+_{\omega_{k,j}}(j|x^k_j)\right)^2,
\end{equation}
where for $R>0$ we define:
\begin{equation}\label{qk}
q_{k,\tau,\mu}(R):=1-2(1-\mu)\alpha_{k,\min}\epsilon_{k,\min}+H_{k,\tau}(R^2+\epsilon_{k,\max}^2)\alpha_{k,\max}^2+2L\Delta_k. 
\end{equation}
In the sequel we specify $x:=t^k$ and take $\esp[\cdot|\alg_k]$ in \eqref{qk}, getting
$$
\esp\left[\Vert x^{k+1}-t^k\Vert^2|\alg_k\right] \le 
q_{k,\tau,\mu}(L)\Vert x^{k}-t^k\Vert^2
+H_{k,\tau}(2B_t^2+M_t^2\epsilon_{k,\max}^2)\alpha_{k,\max}^2
$$
$$
+\frac{(B_t+\epsilon_{k,\max}M_t)^2\Delta_k^2}{2\mu\alpha_{k,\min}\epsilon_{k,\min}}
+2\alpha_{k,\max}\langle(T+D(\epsilon_k))(t^k),t^k-\Pi(x^k)\rangle
$$
\begin{equation}\label{equation:recursion:tyk}
+2\alpha_{k,\max}(B_t+
\epsilon_{k,\max}M_t)\dist(x^k,X)-C A_{k,\tau}\sum_{j=1}^m\esp\left[\left(g^+_{\omega_{k,j}}(j|x^k_j)\right)^2\Big|\alg_k\right].
\end{equation}
For deriving \eqref{equation:recursion:tyk}, we use the facts that 
$x^k\in\alg_k$, $\esp[L(v^k)^2|\alg_k]=L^2$ and $\esp[\Vert F(v^k,t^k)\Vert^2|\alg_k]=2B(t^k)^2\le 2B_t^2$, 
which hold because $\{v^k\}$ is independent of $\alg_k$ and identically distributed to $v$, $\Vert t^k\Vert\le M_t$ and 
$$
\esp\left[2\langle t^k-x^k,D(\alpha_k)\varsigma^k\rangle\Big|\alg_k\right]=2\langle t^k-x^k,D(\alpha_k)\esp\left[\varsigma^k|\alg_k\right]\rangle=0,
$$
in view of the fct that $\esp\left[\varsigma^k|\alg_k\right]=0$.

Using the definition of $C$ in \eqref{convergence:tik:hag}, we get from Assumption \ref{assump:constraint:reg:cart} and 
the fact that $x_j^k\in\alg_k$:  
\begin{equation}\label{equation:reg:mean}
\sum_{j=1}^m\esp\left[\left(g^+_{\omega_{k,j}}(j|x^k_j)\right)^2\Big|\alg_k\right]
\ge \sum_{j=1}^m \frac{1}{c^j}\Vert\Pi_{X^j}(x_j^k)-x_j^k\Vert^2
\ge \frac{1}{C}\sum_{j=1}^m\Vert\Pi_{X^j}(x_j^k)-x_j^k\Vert^2
=\frac{1}{C}\dist(x^k,X)^2.
\end{equation}
By \eqref{equation:reg:mean}, the last term in the right hand side of \eqref{equation:recursion:tyk} is bounded by
\begin{equation}\label{eee2}
-A_{k,\tau}\dist(x^k,X)^2+2(B_t+\epsilon_{k,\max} M_t)\alpha_{k,\max}\dist(x^k,X)\le 
\frac{(B_t+\epsilon_{k,\max} M_t)^2\alpha_{k,\max}^2}{A_{k,\tau}},
\end{equation}
using the fact that $2ab\le\lambda a^2+\frac{b^2}{\lambda}$ with $\lambda:=A_{k,\tau}$, 
$a:=\dist(x^k,X)$ and $b:=(B_t+\epsilon_{k,\max} M_t)\alpha_{k,\max}$.

Since $t^k$ solves VI$(T+D(\epsilon_k),X)$, we have
\begin{equation}\label{equation:def:tyk:seq}
2\alpha_{k,\max}\langle(T+D(\epsilon_k))(t^k),t^k-\Pi(t^k)\rangle\le0.
\end{equation}

Next we relate $\Vert x^k-t^k\Vert^2$ to $\Vert x^k-t^{k-1}\Vert^2$, using the properties of $t^k$ (Lemma \ref{tykhonov.sequence}). We have
$$
\Vert x^k-t^k\Vert^2 \le (\Vert x^k-t^{k-1}\Vert+\Vert t^k-t^{k-1}\Vert)^2
=\Vert x^k-t^{k-1}\Vert^2+\Vert t^k-t^{k-1}\Vert^2+2\Vert x^k-t^{k-1}\Vert\Vert t^k-t^{k-1}\Vert
$$
\begin{equation}\label{eee3}
\le \Vert x^k-t^{k-1}\Vert^2+\Big(M_t\frac{\epsilon_{k-1,\max}-\epsilon_{k,\min}}{\epsilon_{k,\min}}\Big)^2+
2M_t\frac{\epsilon_{k-1,\max}-\epsilon_{k,\min}}{\epsilon_{k,\min}}\Vert x^k-t^{k-1}\Vert.
\end{equation}
Using the  relation $2ab\le\lambda a^2+\frac{b^2}{\lambda}$ for any $\lambda>0$, the last term in the rightmost expression in \eqref{eee3} 
can be estimated as
\begin{eqnarray}\label{eee4}
2M_t\frac{\epsilon_{k-1,\max}-\epsilon_{k,\min}}{\epsilon_{k,\min}}\Vert x^k-t^{k-1}\Vert &=&
2\sqrt{\alpha_{k,\min}\epsilon_{k,\min}}\Vert x^k-t^{k-1}\Vert\cdot M_t\frac{\epsilon_{k-1,\max}-
\epsilon_{k,\min}}{\sqrt{\alpha_{k,\min}\epsilon_{k,\min}}\epsilon_{k,\min}}\nonumber\\
&\le &\alpha_{k,\min}\epsilon_{k,\min}\Vert x^k-t^{k-1}\Vert^2+
M_t^2\frac{\big(\epsilon_{k-1,\max}-\epsilon_{k,\min})^2}
{\alpha_{k,\min}\epsilon_{k,\min}^3}.
\end{eqnarray}
Putting \eqref{eee4} in \eqref{eee3} yields 
\begin{equation}\label{eq12.tik}
\Vert x^k-t^k\Vert^2 \le (1+\alpha_{k,\min}\epsilon_{k,\min})\Vert x^k-t^{k-1}\Vert^2+
\Big(M_t\frac{\epsilon_{k-1,\max}-\epsilon_{k,\min}}{\epsilon_{k,\min}}\Big)^2\Big(1+\frac{1}{\alpha_{k,\min}\epsilon_{k,\min}}\Big).
\end{equation}

Combining \eqref{equation:recursion:tyk}, \eqref{equation:reg:mean}-\eqref{eee2}, \eqref{equation:def:tyk:seq} and \eqref{eq12.tik} we get
$$
\esp\big[\Vert x^{k+1}-t^k\Vert^2\big|\alg_k\big] \le 
q_{k,\tau,\mu}(L)(1+\alpha_{k,\min}\epsilon_{k,\min})\Vert x^{k}-t^{k-1}\Vert^2
$$
$$
+\Bigg[H_{k,\tau}(2B_t^2+M_t^2\epsilon_{k,\max}^2)+\frac{(B_t+M_t\epsilon_{k,\max})^2}{A_{k,\tau}}\Bigg]\alpha_{k,\max}^2
+\frac{(B_t+\epsilon_{k,\max}M_t)^2\Delta_k^2}{2\mu\alpha_{k,\min}\epsilon_{k,\min}}
$$
\begin{equation}\label{eq13.tik}
+q_{k,\tau,\mu}(L)\Big(M_t\frac{\epsilon_{k-1,\max}-\epsilon_{k,\min}}{\epsilon_{k,\min}}\Big)^2\Big(1+\frac{1}{\alpha_{k,\min}\epsilon_{k,\min}}\Big).
\end{equation}

We now estimate the coefficient $q_{k,\tau,\mu}(L)(1+\alpha_{k,\min}\epsilon_{k,\min})$ in \eqref{eq13.tik}. 
In view of \eqref{qk}, we have
\begin{equation}\label{qk.eq2}
q_{k,\tau,\mu}(L)=1-\alpha_{k,\min}\epsilon_{k,\min}\Bigg( 2-2\mu-\frac{H_{k,\tau}(L^2+
\epsilon_{k,\max}^2)\alpha_{k,\max}^2}{\alpha_{k,\min}\epsilon_{k,\min}}
-\frac{2L\Delta_k}{\alpha_{k,\min}\epsilon_{k,\min}}\Bigg).
\end{equation}
Assumption \ref{approx.step.tik}(ii) and 
$0<H_{k,\tau}=4[1+\beta_{k,\min}(2-\beta_{k,\max})\tau]\le4(1+\tau)$ guarantee that
$$
\frac{H_{k,\tau}(L^2+\epsilon_{k,\max}^2)\alpha_{k,\max}^2}{\alpha_{k,\min}\epsilon_{k,\min}}
+\frac{2L\Delta_k}{\alpha_{k,\min}\epsilon_{k,\min}}
\rightarrow 0.
$$
Since $\mu\in(0,1)$ is arbitrary, we can ensure the existence of $d\in(0,1)$ such that 
\begin{equation}\label{c}
c_k:=2\mu+\frac{H_{k,\tau}(L^2+\epsilon_{k,\max}^2)\alpha_{k,\max}^2}{\alpha_{k,\min}\epsilon_{k,\min}}
+\frac{2L\Delta_k}{\alpha_{k,\min}\epsilon_{k,\min}}<d
\end{equation}
for all sufficiently large $k$. 
Next we show that $q_{k,\tau,\mu}(L)\in(0,1)$ for large $k$. 
Indeed, from \eqref{c} and $d\in(0,1)$ we have 
that $1<2-c_k<2$ 
for large enough $k$, 
so that we obtain from \eqref{qk.eq2}, 
\begin{equation}\label{eee5}
1-2\alpha_{k,\min}\epsilon_{k,\min}<q_{k,\tau,\mu}(L)<1-\alpha_{k,\min}\epsilon_{k,\min}.
\end{equation}
Finally, 
$\lim_{k\rightarrow\infty}\alpha_{k,\min}\epsilon_{k,\min}=0$ 
by Assumption \ref{approx.step.tik}(ii), 
so that \eqref{eee5} implies that 
$q_{k,\tau,\mu}(L)\in(0,1)$ 
for sufficiently large $k$. 
Using this fact and \eqref{c} we get the following estimate:
\begin{eqnarray}
0<q_{k,\tau,\mu}(L)(1+\alpha_{k,\min}\epsilon_{k,\min})&\le &q_{k,\tau,\mu}(L)+\alpha_{k,\min}\epsilon_{k,\min}\nonumber\\
&=&1-\alpha_{k,\min}\epsilon_{k,\min}(2-c_k)+\alpha_{k,\min}\epsilon_{k,\min}\nonumber\\
&=&1-\alpha_{k,\min}\epsilon_{k,\min}(1-c_k)\nonumber\\
&\le &1-\alpha_{k,\min}\epsilon_{k,\min}(1-d),\label{qk.eq3}
\end{eqnarray}
using \eqref{c} in the last inequality. 

Combining \eqref{eq13.tik}, \eqref{qk.eq3} and $A_{k,\tau}=\beta_{k,\min}(2-\beta_{k,\max})G_\tau^{-1}$, we obtain 
\begin{equation}\label{eq14.tik}
\esp\big[\Vert x^{k+1}-t^k\Vert^2\big|\alg_k\big] \le 
(1-a_k)\Vert x^{k}-t^{k-1}\Vert^2+b_k,
\end{equation}
for all sufficiently large $k$,
with $a_k:=\alpha_{k,\min}\epsilon_{k,\min}(1-d)$ and
$$
b_k:=\Bigg[H_{k,\tau}(2B_t^2+M_t^2\epsilon_{k,\max}^2)+
\frac{G_\tau(B_t+M_t\epsilon_{k,\max})^2}{\beta_{k,\min}(2-\beta_{k,\max})}\Bigg]\alpha_{k,\max}^2
$$
\begin{equation}\label{bk}
+\frac{(B_t+\epsilon_{k,\max}M_t)^2\Delta_k^2}{2\mu\alpha_{k,\min}\epsilon_{k,\min}}
+q_{k,\tau,\mu}(L)\Big(M_t\frac{\epsilon_{k-1,\max}-\epsilon_{k,\min}}{\epsilon_{k,\min}}\Big)^2\Big(1+\frac{1}{\alpha_{k,\min}\epsilon_{k,\min}}\Big).
\end{equation}
From \eqref{qk.eq3} and $d\in(0,1)$, 
we conclude that $a_k\in[0,1]$, while from Assumption \ref{approx.step.tik}(iii) we have that $\sum_k a_k=\infty$. 
From Assumption \ref{approx.step.tik}(iv) and \eqref{bk}, we also get that $\sum_k b_k<\infty$. 
Finally, using the definitions of $\Gamma_k$ and $\mathsf{B}_k$, we obtain from \eqref{bk}: 
$$	
0\le\frac{b_k}{a_k}=C_1\frac{\alpha_{k,\max}^2}{\alpha_{k,\min}\epsilon_{k,\min}}
+C_2\frac{\alpha_{k,\max}^2}{\mathsf{B}_k\alpha_{k,\min}\epsilon_{k,\min}}
+C_3\Bigg(\frac{\Delta_k}{\alpha_{k,\min}\epsilon_{k,\min}}\Bigg)^2
+C_4\frac{\Gamma_k^2}{\epsilon_{k,\min}^3\alpha_{k,\min}}\Big(1+\frac{1}{\alpha_{k,\min}\epsilon_{k,\min}}\Big)
$$
for some positive constants $C_1$, $C_2$, $C_3$ and $C_4$.
Therefore, we get $\lim_{k\rightarrow\infty}b_k/a_k=0$
from Assumption \ref{approx.step.tik}(ii) and (v). 
These conditions, 
Theorem \ref{rob2} 
and \eqref{eq14.tik} 
imply 
that $\lim_{k\rightarrow\infty}\Vert x^k-t^{k-1}\Vert=0$ almost surely. 
The result  follows from this fact and Lemma \ref{tykhonov.sequence}.
\end{proof}	

\subsection{Convergence rate analysis}\label{subsection:rate:analyis:tyk}
Next we give feasibility and solvability convergence rates. The feasibility rate will be given in terms of 
the metric $\dist(\cdot):=\dist(\cdot,X)^2$ evaluated at  
$$
\widetilde x^k:=\frac{\sum_{i=0}^k\mathsf{B}_{i}x^i}{\sum_{i=0}^k\mathsf{B}_{i}},
$$
i.e., the ergodic average of the iterates with weights $\mathsf{B}_k=\beta_{k,\min}(2-\beta_{k,\max})$. 
Assuming that $X$ is compact (but allowing the hard constraint $X_0$ to be unbounded), the solvability convergence rate 
will be given in terms of the dual gap function $\mathsf{G}$, defined in \eqref{def:gap:function}, evaluated at
$$
\widehat x^k:=\frac{\sum_{i=0}^k\alpha_{i,\max}\Pi(x^i)}{\sum_{i=0}^k\alpha_{i,\max}},
$$
which is the ergodic average of the feasible projections of the iterates with weights $\alpha_{k,\max}$. 
We shall use the notation $\overline x^k:=\Pi(x^k)$ for $k\in\mathbb{N}_0$. 

In the remainder of this subsection we recall definitions \eqref{def:sigma}, \eqref{def:B}, \eqref{def:Mt}-\eqref{convergence:tik:hag} 
and the ones given in Assumption \ref{approx.step.tik}. We first present the feasibility rate. In order to facilitate the presentation,
we define some constants. Given $\tau,\mu\in(0,1)$ and $R>0$ we set
\begin{eqnarray}
q_{k,\tau,\mu}(R):=1-2(1-\mu)\alpha_{k,\min}\epsilon_{k,\min}+
H_{k,\tau}(R^2+\epsilon_{k,\max}^2)\alpha_{k,\max}^2+2L\Delta_k,\label{def:qk:tilde:tau:mu}\\
f_{0,\tau,\mu}:=q_{0,\tau,\mu}(L)(1+\alpha_{0,\min}\epsilon_{0,\min}),\quad f_{k,\tau,\mu}:=
q_{k,\tau,\mu}(L)(1+\alpha_{k,\min}\epsilon_{k,\min})-1\mbox{ for }k\ge1,\label{def:fk:tilde}\\
\overline{H}_\tau:=\max_{k\in\mathbb{N}_0}H_{k,\tau},\quad I_t:=B_t+
\epsilon_{0,\max}M_t,\quad J_{t,\tau}:=\overline{H}_\tau\cdot\left(2B_t^2+\epsilon_{0,\max}^2M_t^2\right).\label{thm:feas:rate:tik:eq1}
\end{eqnarray}

\begin{theorem}[Feasibility rate]\label{thm:feas:rate:tik}
Suppose Assumptions \ref{mono.lips.unbiased.tik}-\ref{approx.step.tik} hold. Then given $\tau,\mu\in(0,1)$, for all $k\in\mathbb{N}_0$,
\begin{eqnarray*}
\esp[\dist(\widetilde x^k)^2]\sum_{i=0}^k\mathsf{B}_i&\le &
2G_\tau\sum_{i=0}^k f_{i,\tau,\mu}\esp[\Vert x^i-t^i\Vert^2]
+2G_\tau J_{t,\tau}\sum_{i=0}^k\alpha_{i,\max}^2+
\frac{G_\tau I_t^2}{\mu}\sum_{i=0}^k\frac{\Delta_{i}^2}{\alpha_{i,\min}\epsilon_{i,\min}}\nonumber\\
&+&4G_\tau^2 I_t^2\sum_{i=0}^k
\frac{\alpha_{i,\max}^2}{\mathsf{B}_i}+
2G_\tau\sum_{i=0}^k q_{i,\tau,\mu}(L)\left(\frac{M_t\Gamma_i}{\epsilon_{i,\min}}\right)^2\left(1+\frac{1}{\alpha_{i,\min}\epsilon_{i,\min}}\right).
\end{eqnarray*}
\end{theorem}
\begin{proof}
We recall relation \eqref{equation:recursion:tyk} 
in the proof of Theorem \ref{convergence.tik}. 
Instead of using \eqref{eee2}, we bound the left hand side of \eqref{eee2} by
\begin{equation}\label{thm:feas:rate:tik:eq2}
-A_{k,\tau}\dist(x^k)^2+2I_t\alpha_{k,\max}\dist(x^k)\le
-\frac{A_{k,\tau}}{2}\dist(x^k)^2+\frac{2I_t^2\alpha_{k,\max}^2}{A_{k,\tau}}=
-\frac{\mathsf{B}_k}{2G_\tau}\dist(x^k)^2+\frac{2 G_\tau I_t^2\alpha_{k,\max}^2}{\mathsf{B}_k},
\end{equation}
using the facts that $2ab\le \lambda a^2+\lambda^{-1}b^2$ with $\lambda=A_{k,\tau}/2$, $a=\dist(x^k)$ and $b=I_t\alpha_{k,\max}$.

We combine \eqref{equation:recursion:tyk}, \eqref{equation:reg:mean}, \eqref{equation:def:tyk:seq} 
and \eqref{eq12.tik} with \eqref{thm:feas:rate:tik:eq2}, take total expectation and sum from $0$ to $k$ in order to get 
\begin{eqnarray}
\sum_{i=0}^k\mathsf{B}_i\esp[\dist(x^i)^2]&\le &2G_\tau\sum_{i=0}^k f_{i,\tau,\mu}\esp[\Vert x^i-t^i\Vert^2]+
2G_\tau J_{t,\tau}\sum_{i=0}^k\alpha_{i,\max}^2+
\frac{G_\tau I_t^2}{\mu}\sum_{i=0}^k\frac{\Delta_{i}^2}{\alpha_{i,\min}\epsilon_{i,\min}}\nonumber\\
&+&4G_\tau^2 I_t^2\sum_{i=0}^k\frac{\alpha_{i,\max}^2}{\mathsf{B}_i}+
2G_\tau\sum_{i=0}^k q_{i,\tau,\mu}(L)\left(\frac{M_t\Gamma_i}{\epsilon_{i,\min}}\right)^2\left(1+\frac{1}{\alpha_{i,\min}\epsilon_{i,\min}}\right).
\label{thm:feas:rate:tik:eq3}
\end{eqnarray}
In view of the convexity of $y\mapsto\dist(y)^2$ and the linearity of the expectation operator, we have
\begin{equation}\label{thm:feas:rate:tik:eq4}
\esp[\dist(\widetilde{x}^k)^2]\le\frac{\sum_{i=0}^k\mathsf{B}_i\esp[\dist(x^i)^2]}{\sum_{i=0}^k\mathsf{B}_i}.
\end{equation}
Relations \eqref{thm:feas:rate:tik:eq3}-\eqref{thm:feas:rate:tik:eq4} prove the required claim.
\end{proof}

Next we present the solvability rate assuming that $X$ is compact. We will need the following definitions: for $\tau,\mu\in(0,1)$ and $R>0$,
\begin{equation}
B_X:=\sup_{x\in X}B(x),\quad M_X:=\sup_{x\in X}\Vert x\Vert,\quad I_X:=B_X+\epsilon_{0,\max}M_X,
\end{equation}
\begin{equation}
J_{X,\tau}(x):=\overline{H}_\tau\left[2L^2\diam(X)^2+2B(x)^2+\epsilon_{0,\max}^2M_X^2\right],\quad K_X(x):=6L^2\diam(X)^2+3\sigma(x)^2,
\end{equation}
\begin{equation}\label{thm:gap:rate:tyk:eq1}
h_{0,\tau,\mu}(R):=q_{0,\tau,\mu}(R),\quad h_{k,\tau,\mu}(R):=q_{k,\tau,\mu}(R)-1,\mbox{ for }k\ge1.
\end{equation}

We start with an intermediate lemma.
\begin{lemma}[Feasibility error control]\label{lemma:feas:error}
For any $I>0$ and $k\in\mathbb{N}_0$,
$$
\sum_{i=0}^k\esp\left\{2I\alpha_{i,\max}\dist(x^i)-
\frac{C\mathsf{B}_i}{G_\tau}\sum_{j=1}^m\left[g^+_{\omega_{i,j}}(j|x^i_j)\right]^2\right\}\le 
G_\tau I^2\sum_{i=0}^k\frac{\alpha_{i,\max}^2}{\mathsf{B}_i}.
$$
\end{lemma}
\begin{proof}
For $0\le \ell\le k$, define
$$
Q_\ell:=\sum_{i=0}^\ell\left\{2I\alpha_{i,\max}\dist(x^i)-
\frac{C\mathsf{B}_i}{G_\tau}\sum_{j=1}^m\left[g^+_{\omega_{i,j}}(j|x^i_j)\right]^2\right\}.
$$
We have
\begin{eqnarray}
\esp\left\{Q_k\Big|\alg_k\right\}&=&Q_{k-1}+
2I\alpha_{k,\max}\dist(x^k)-
\frac{C\mathsf{B}_k}{G_\tau}\esp\left\{\sum_{j=1}^m\left[g^+_{\omega_{k,j}}(j|x^k_j)\right]^2\Bigg|\alg_k\right\}\nonumber\\
&\le &Q_{k-1}+2I\alpha_{k,\max}\dist(x^k)-\frac{\mathsf{B}_k}{G_\tau}\dist(x^k)^2\nonumber\\
&\le &Q_{k-1}+\frac{G_\tau I^2\alpha_{k,\max}^2}{\mathsf{B}_k},\label{lemma:feas:error:eq1}
\end{eqnarray}
using the fact that $Q_{k-1},x^k\in\alg_k$ 
in the equality,  
\eqref{equation:reg:mean} in the first inequality and 
the fact that $2ab\le\lambda a^2+\lambda^{-1} b^2$ with $a:=I\alpha_{k,\max}$, $b:=\dist(x^k)$ and $\lambda:=G_\tau/\mathsf{B}_k$
in the second inequality. 
We then take $\esp[\cdot|\alg_{k-1}]$ in \eqref{lemma:feas:error:eq1} and use the fact that 
$\esp[\esp[\cdot|\alg_k]|\alg_{k-1}]=\esp[\cdot|\alg_{k-1}]$ in order to obtain
\begin{eqnarray}
\esp\left\{Q_k\Big|\alg_{k-1}\right\}&\le &\esp\left\{Q_{k-1}\Big|\alg_{k-1}\right\}+
\frac{G_\tau I^2\alpha_{k,\max}^2}{\mathsf{B}_k}.\label{lemma:feas:error:eq2}
\end{eqnarray}
Proceeding by induction as in \eqref{lemma:feas:error:eq1}-\eqref{lemma:feas:error:eq2}, we get 
\begin{eqnarray}
\esp\left\{Q_k\Big|\alg_0\right\}&\le &\esp\left\{Q_0\Big|\alg_0\right\}+
\sum_{i=1}^k\frac{G_\tau I^2\alpha_{i,\max}^2}{\mathsf{B}_i}\\
&\le &\frac{G_\tau I^2\alpha_{0,\max}^2}{\mathsf{B}_0}+
\sum_{i=1}^k\frac{G_\tau I^2\alpha_{i,\max}^2}{\mathsf{B}_i}=\sum_{i=0}^k\frac{G_\tau I^2\alpha_{i,\max}^2}{\mathsf{B}_i}.
\label{mumu}
\end{eqnarray}
Taking total expectation in \eqref{mumu} and using the fact that $\esp[\esp[\cdot|\alg_0]]=\esp[\cdot]$, we prove the claim.
\end{proof}

\begin{theorem}[Solvability rate]\label{thm:gap:rate:tyk}
Suppose that Assumptions \ref{mono.lips.unbiased.tik}-\ref{approx.step.tik} hold. Then, given $\tau,\mu\in(0,1)$, for all $k\in\mathbb{N}_0$,
\begin{eqnarray}
2\esp[\mathsf{G}(\widehat{x}^k)]\sum_{i=0}^k\alpha_{i,\max}&\le&\diam(X)^2+
2\sum_{i=0}^k h_{i,\tau,\mu}(L)\left\{\esp\left[\dist(x^i)^2\right]+\diam(X)^2\right\}\nonumber\\
&&+\left[J_{X,\tau}(\overline x^0)+K_X(\overline x^0)\right]\sum_{i=0}^k\alpha_{i,\max}^2
+G_\tau [I_X+2L\diam(X)]^2\sum_{i=0}^k\frac{\alpha_{i,\max}^2}{\mathsf{B}_i}\nonumber\\
&&+\frac{I_X^2}{2\mu}\sum_{i=0}^k\frac{\Delta_{i}^2}{\alpha_{i,\min}\epsilon_{i,\min}}+
2\diam(X)M_X\sum_{i=0}^k\alpha_{i,\max}\epsilon_{i,\max}.\label{thm:gap:rate:tyk:statement}
\end{eqnarray}
\end{theorem}
\begin{proof}
We recall relation \eqref{eq10.tik} 
in the proof of Theorem \ref{convergence.tik}, 
where $\varsigma^k$ is defined in \eqref{def:varsigma}. Regarding the second line of \eqref{eq10.tik}, we have for any $x\in X$,
\begin{equation}\label{thm:gap:rate:tyk:eq2}
2\alpha_{k,\max}\langle (T+D(\epsilon_k))(x),x-\Pi(x^k)\rangle\le2\alpha_{k,\max}\langle T(x),x-\Pi(x^k)\rangle+
2\alpha_{k,\max}\epsilon_{k,\max} M_X\diam(X),
\end{equation}
using Cauchy-Schwartz inequality and the definitions of $M_X$ and $\diam(X)$.

We set $Q(x,y):=\langle T(x),y-x\rangle$ so that $\mathsf{G}(y):=\sup_{x\in X}Q(x,y)$ as in \eqref{def:gap:function}. 
Using \eqref{thm:gap:rate:tyk:eq2} in \eqref{eq10.tik} and then summing from $0$ to $k$, we get for all $x\in X$,
\begin{eqnarray}
2\sum_{i=0}^k\alpha_{i,\max}Q(x,\Pi(x^i))&\le 
&\sum_{i=0}^k h_{i,\tau,\mu}(L(v^i))\Vert x^i-x\Vert^2+
\overline{H}_\tau\sum_{i=0}^k\left[\Vert F(v^i,x)\Vert^2+
\Vert x\Vert^2\epsilon_{i,\max}^2\right]\alpha_{i,\max}^2\nonumber\\
&&+\frac{I_X^2}{2\mu}\sum_{i=0}^k
\frac{\Delta_{i}^2}{\alpha_{i,\min}\epsilon_{i,\min}}+2\diam(X)M_X\sum_{i=0}^k\alpha_{i,\max}\epsilon_{i,\max}\nonumber\\
&&+\sum_{i=0}^k\left\{2I_X\alpha_{i,\max}\dist(x^i)-
\frac{C\mathsf{B}_i}{G_\tau}\sum_{j=1}^m\left[g^+_{\omega_{i,j}}(j|x^i_j)\right]^2\right\}\nonumber\\
&&+2\sum_{i=0}^k\langle x-x^i,D(\alpha_i)\varsigma^i\rangle, \label{thm:gap:rate:tyk:eq3}
\end{eqnarray}
where the last line of \eqref{eq10.tik} has been bounded using the definition of $I_X$.

The total expectation of the term in the first line of \eqref{thm:gap:rate:tyk:eq3} is bounded above by
\begin{eqnarray}
&&2\sum_{i=0}^k\esp\left\{h_{i,\tau,\mu}(L(v^i))\left[\dist(x^i)^2+\diam(X)^2\right]\right\}+
\overline{H}_\tau\sum_{i=0}^k\left[2L^2\diam(X)^2+2B(\overline x^0)^2+
M_X^2\epsilon_{i,\max}^2\right]\alpha_{i,\max}^2,\nonumber\\
&\le &\sum_{i=0}^k\esp\left\{\esp\left[h_{i,\tau,\mu}(L(v^i))\big|\alg_i\right]\cdot\left[\dist(x^i)^2+\diam(X)^2\right]\right\}+J_{X,\tau}(\overline x_0)\sum_{i=0}^k\alpha_{i,\max}^2\nonumber\\
&=&\sum_{i=0}^k h_{i,\tau,\mu}(L)\esp\left[\dist(x^i)^2+\diam(X)^2\right]+J_{X,\tau}(\overline x_0)\sum_{i=0}^k\alpha_{i,\max}^2,
\label{thm:gap:rate:tyk:eq4}
\end{eqnarray}
where in first line we used Lemma \ref{l1}, $\Vert x^i-x\Vert^2\le2\dist(x^i)^2+2\diam(X)^2$, 
$\Vert x-\overline x^i\Vert\le\diam(X)$ and $\Vert x\Vert\le M_X$ for all $x\in X$ and $0\le i\le k$, in second line we used the property $\esp\{\esp\{\cdot|\alg_i\}\}=\esp\{\cdot\}$ and $x^i\in\alg_i$ and in third line we used $\esp\left[h_{i,\tau,\mu}(L(v^i))\big|\alg_i\right]=\esp\left[h_{i,\tau,\mu}(L(v^i))\right]=h_{i,\tau,\mu}(L)$ (using Assumption \ref{unbiased}).

We will now bound the last term in the right hand side of \eqref{thm:gap:rate:tyk:eq3}. We define
$$
\overline{\varsigma}^k:=F(v^k,\overline x^k)-T(\overline x^k),\quad \Delta\varsigma^k:=\varsigma^k-\overline \varsigma^k.
$$
We define $\{u^k\}$ recursively as follows. Take any $u^0\in X$ and set, for $k\in\mathbb{N}_0$,
$$
u^{k+1}:=\Pi\left[u^k+D(\alpha_k)\overline\varsigma^k\right].
$$
Note that $u^k\in\alg_k$. We write, for all $k\in\mathbb{N}_0$,
\begin{eqnarray}
\Delta M_{k}&:=&\langle u^k-x^k,D(\alpha_k)\varsigma^k\rangle,\nonumber\\
2\sum_{i=0}^k\langle x-x^i,D(\alpha_i)\varsigma^i\rangle &=& 
2\sum_{i=0}^k\langle x-u^i,D(\alpha_i)\overline\varsigma^i\rangle+
2\sum_{i=0}^k\langle x-u^i,D(\alpha_i)\Delta\varsigma^i\rangle+2\sum_{i=0}^k\Delta M_{i}.\label{thm:gap:rate:tyk:eq5}
\end{eqnarray}
Note that for all $k\in\mathbb{N}_0$, 
\begin{equation}\label{thm:gap:rate:tyk:zero:mean:M}
\esp\left[\Delta M_{k}\right]=0,
\end{equation}
which follows from $u^k,x^k\in\alg_k$ and $\esp[\varsigma^k]=0$ (Assumption \ref{unbiased}).

Concerning the first term in the right hand side of \eqref{thm:gap:rate:tyk:eq5}, we have
\begin{eqnarray}
2\langle x-u^i,D(\alpha_i)\overline\varsigma^i\rangle &=& 2\langle x-u^{i+1},D(\alpha_i)\overline\varsigma^i\rangle+
2\langle u^{i+1}-u^i,D(\alpha_i)\overline\varsigma^i\rangle\nonumber\\
&\le &\Vert u^i-x\Vert^2-\Vert u^{i+1}-x\Vert^2-\Vert u^{i+1}-u^i\Vert^2\nonumber\\
&+&\Vert u^{i+1}-u^i\Vert^2+\Vert D(\alpha_i)\overline\varsigma^i\Vert^2\nonumber\\
&\le &\Vert u^i-x\Vert^2-\Vert u^{i+1}-x\Vert^2+\alpha_{i,\max}^2\Vert\overline\varsigma^i\Vert^2,\label{thm:gap:rate:tyk:eq6}
\end{eqnarray}
using Lemma \ref{proj}(iii) with the definition of $u^{i+1}$ and $2ab\le a^2+b^2$ with $a:=\Vert u^{i+1}-u^i\Vert$ 
and $b:=D(\alpha_i)\overline\varsigma^i$ 
in the first inequality. 
Summing \eqref{thm:gap:rate:tyk:eq6} from $0$ to $k$ and then taking total expectation in \eqref{thm:gap:rate:tyk:eq5} we get 
\begin{eqnarray}
\esp\left[2\sum_{i=0}^k\langle x-x^i,D(\alpha_i)\varsigma^i\rangle\right] 
&\le &\diam(X)^2+\sum_{i=0}^k\alpha_{i,\max}^2\esp\left[\Vert\overline\varsigma^i\Vert^2\right]+
2\sum_{i=0}^k\esp\left[\langle x-u^i,D(\alpha_i)\Delta\varsigma^i\rangle\right],
\label{thm:gap:rate:tyk:eq7}
\end{eqnarray}
using the fact that $\Vert u^0-x\Vert\le\diam(X)$ and \eqref{thm:gap:rate:tyk:zero:mean:M}.
Regarding the second term in the right hand side of   
\eqref{thm:gap:rate:tyk:eq7},
we have
\begin{eqnarray}
\esp\left\{\Vert\overline\varsigma^i\Vert^2\right\}=\esp\left\{\esp\{\Vert\overline\varsigma^i\Vert^2|\alg_i\}\right\}
&\le &\esp\left\{\esp\left\{3[L(v^i)^2+L^2]\Vert\overline x^i-\overline x^0\Vert^2+3\sigma(\overline x^0)^2|\alg_i\right\}\right\},\nonumber\\
&\le &6L^2\diam(X)^2+3\sigma(\overline x^0)^2,
\end{eqnarray}
using the Lipschitz continuity of $F(v^k,\cdot)$ and $T$,
$\overline\varsigma^i=F(v^i,\overline x^i)-F(v^i,\overline x^0)+T(\overline x^0)-T(\overline x^i)+F(v^i,\overline x^0)-T(\overline x^0)$,  
$(a+b+c)^2\le3a^2+3b^2+3c^2$ and $\overline x^i\in\alg_i$ 
in the first inequality 
and that $\esp[L(v^i)^2|\alg_i]=L^2$ and $\Vert\overline x^i-\overline x^0\Vert\le\diam(X)$
in the second inequality.  
The third term in the right hand side of \eqref{thm:gap:rate:tyk:eq7} is equal to	
\begin{eqnarray}
2\sum_{i=0}^k\esp\left\{\esp\left\{\langle x-u^i,D(\alpha_i)\Delta\varsigma^i\rangle|\alg_i\right\}\right\}&\le &
2\diam(X)\sum_{i=0}^k\alpha_{i,\max}\esp\left\{\esp\left\{\Vert\Delta\varsigma^i\Vert|\alg_i\right\}\right\}\nonumber\\
&\le &2\diam(X)\sum_{i=0}^k\alpha_{i,\max}\esp\left\{\esp\left\{[L(v^i)+L]\Vert x^i-\overline x^i\Vert|\alg_i\right\}\right\}\nonumber\\
&\le &4\diam(X)L\sum_{i=0}^k\alpha_{i,\max}\esp\left\{\dist(x^i)\right\},\label{thm:gap:rate:tyk:eq7:2}
\end{eqnarray}
using Cauchy-Schwartz inequality and the fact that $\Vert x-u^i\Vert\le\diam(X)$, 
in the first inequality, 
the Lipschitz continuity of $F(v^k,\cdot)$ 
and $T$ 
in the second inequality,
and 
that $\esp[L(v^i)|\alg_i]\le L$ and $x^i\in\alg_i$
in the third inequality. 

From the convexity of $y\mapsto Q(x,y)$, we  get 
\begin{equation}\label{thm:gap:rate:tyk:eq8}
\esp\left[\mathsf{G}(\widehat x^k)\right]=
\esp\left[\sup_{x\in X}Q(x,\widehat x^k)\right]\le
\esp\left[\sup_{x\in X}\frac{\sum_{i=0}^k\alpha_{i,\max}Q(x,\Pi(x^i))}{\sum_{i=0}^k\alpha_{i,\max}}\right].
\end{equation}

We are now ready to prove the claim. We take total expectation in \eqref{thm:gap:rate:tyk:eq3} 
and combine it with \eqref{thm:gap:rate:tyk:eq4} and \eqref{thm:gap:rate:tyk:eq7}-\eqref{thm:gap:rate:tyk:eq8}. 
In order to complete the proof, we use the obtained relation, combine the expectation of the fifth term 
$$
\sum_{i=0}^k\left\{2I_X\alpha_{i,\max}\dist(x^i)-
\frac{C\mathsf{B}_i}{G_\tau}\sum_{j=1}^m\left[g^+_{\omega_{i,j}}(j|x^i_j)\right]^2\right\}
$$
in the right hand side of \eqref{thm:gap:rate:tyk:eq3} with \eqref{thm:gap:rate:tyk:eq7:2} and 
use Lemma \ref{lemma:feas:error} with $I:=I_X+2L\diam(X)$ in order to obtain the final bound
$$
\sum_{i=0}^k\esp\left\{2[I_X+2L\diam(X)]\alpha_{i,\max}\dist(x^i)-\frac{C\mathsf{B}_i}{G_\tau}
\sum_{j=1}^m\left[g^+_{\omega_{i,j}}(j|x^i_j)\right]^2\right\}\le G_\tau [I_X+2L\diam(X)]^2\sum_{i=0}^k\frac{\alpha_{i,\max}^2}{\mathsf{B}_i}.
$$
\end{proof}

\begin{corollary}[Solvability and feasibility rates: asynchronous parameters]
Suppose\newline
that Assumptions \ref{mono.lips.unbiased.tik}-\ref{approx.step.tik} hold. 
Take stepsizes and regularization parameters as specified in Corollary \ref{cor:step:reg}. 
Then Theorem \ref{convergence.tik} and the following feasibility rate hold:
$$
\esp\left[\dist(\widetilde x^k,X)^2\right]\lesssim\frac{1}{k}.
$$
If additionally $X$ is compact, the following solvability rate holds: for any $\delta\in(0,\frac{1}{2})$,
$$
\esp\left[\mathsf{G}(\widehat x^k)\right]\lesssim\frac{k^\delta\ln k}{\sqrt{k}}.
$$
\end{corollary}
\begin{proof}
The stated stepsizes and regularization parameters of Corollary \ref{cor:step:reg} satisfy Assumption \ref{approx.step.tik}, 
so that a.s.-convergence follows from Theorem \ref{convergence.tik}. In the sequel we fix $\mu,\tau\in(0,1)$.

We first establish the feasibility rate. We have
\begin{eqnarray}
\sum_{i=0}^\infty f_{i,\tau,\mu}\esp[\Vert x^i-t^i\Vert^2]<\infty,\quad
\sum_{i=0}^\infty \alpha_{i,\max}^2<\infty, \quad\sum_{i=0}^\infty\frac{\alpha_{i,\max}^2}{\mathsf{B}_i}<\infty,\label{cor:rate:tyk:eq1}\\
\sum_{i=0}^\infty\frac{\Delta_i^2}{\alpha_{i,\min}\epsilon_{i,\min}}<\infty, 
\quad\sum_{i=0}^\infty q_{i,\tau,\mu}(L)
\left(\frac{M_t\Gamma_i}{\epsilon_{i,\min}}\right)^2\left(1+\frac{1}{\alpha_{i,\min}\epsilon_{i,\min}}\right)<\infty.
\label{cor:rate:tyk:eq2}
\end{eqnarray}
The first inequality in \eqref{cor:rate:tyk:eq1} follows from \eqref{qk.eq3}, which implies that $f_{k,\tau,\mu}$ is negative 
for all sufficiently large $k$. The remaining inequalities in \eqref{cor:rate:tyk:eq1}-\eqref{cor:rate:tyk:eq2} 
follow from Corollary \ref{cor:step:reg} and from the boundedness of $q_{k,\tau,\mu}(L)$ (see \eqref{eee5} 
in Theorem \ref{convergence.tik}). The claimed feasibility rate follows from  
\eqref{cor:rate:tyk:eq1}-\eqref{cor:rate:tyk:eq2},
Theorem \ref{thm:feas:rate:tik} and the fact that $\sum_{i=0}^k\mathsf{B}_i=\min_{j\in[m]}\beta_j(2-\max_{j\in[m]}\beta_j)k$.

We now establish the solvability rate. We have
\begin{eqnarray}
\sum_{i=0}^k\alpha_{i,\max}\epsilon_{i,\max}\sim\sum_{i=0}^k\frac{1}{i}\sim\ln k,
\sum_{i=0}^k\alpha_{i,\max}\sim\sum_{i=0}^k\frac{1}{i^{\frac{1}{2}+\delta}}\sim k^{\frac{1}{2}-\delta}.
\label{cor:rate:tyk:eq3}
\end{eqnarray}
Also, $h_{k,\tau,\mu}(L)$ is negative for sufficiently large $k$ (as shown by relation \eqref{eee5}) so $\sum_{i=0}^\infty h_{i,\tau,\mu}(L)\left\{\esp\left[\dist(x^i)^2\right]+\diam(X)^2\right\}<\infty$. This, \eqref{cor:rate:tyk:eq1}-\eqref{cor:rate:tyk:eq3} and Theorem \ref{thm:gap:rate:tyk} prove the claim on the solvability rate.
\end{proof}

\section*{Appendix}
Proof of Proposition \ref{weak.sharp.equivalence}:
\begin{proof}
Suppose that \eqref{weakly.sharp1} holds and take $x^*\in X^*$. If $\tang_{X}(x^*)\cap\polar_{X^*}(x^*)=\{0\}$, 
then \eqref{weak.sharp.aux} holds trivially. Otherwise, take $d\in\tang_{X}(x^*)\cap\polar_{X^*}(x^*)$ with $d\ne0$. 
Since $d\in\polar_{X^*}(x^*)$, the definition of $\polar_{X^*}(x^*)$ implies that $X^*$ is a subset of the halfspace 
$H_d^-:=\{y:\langle d,y-x^*\rangle\le0\}$.
In view of 
\eqref{tangent:cone:def} 
and $d\in\tang_{X}(x^*)$, 
there exist sequences $d^k\in\re^n$, $t_k>0$ 
such that $x^*+t_k d^k\in X$, $d^k\rightarrow d$ and $t^k\rightarrow0$. We claim that, taking a subsequence if needed, 
\begin{equation}\label{prop.weak.sharp:eq1}
x^*+t_k d^k\in X-H_d^-.
\end{equation}
for all $k$.
Indeed, otherwise we would have 
\begin{equation}\label{e104}
0\ge\langle d,x^*+t_kd^k-x^*\rangle=t_k\langle d,d^k\rangle
\end{equation}
for large enough $k$.  
Dividing \eqref{e104}  by $t_k$ and letting $k\rightarrow\infty$ we get $d=0$ which entails a contradiction. 
Hence, \eqref{prop.weak.sharp:eq1} holds. From \eqref{weakly.sharp1}, $x^*\in X^*$ and $x^*+t_kd^k\in X$ we get
\begin{equation}\label{e105}
\langle T(x^*),x^*+t_k d^k-x^*\rangle\ge\rho\dist(x^*+t_k d^k,X^*)
\ge\rho\dist(x^*+t_k d^k,H_d^-)=\rho t_k\frac{\langle d,d^k\rangle}{\Vert d\Vert},
\end{equation}
using 
\eqref{prop.weak.sharp:eq1} 
and the fact that $X^*\subset H_d^-$ in the second inequality. 
Dividing \eqref{e105} by $t_k$ and letting $k\rightarrow\infty$, we conclude that \eqref{weak.sharp.aux} holds for $d$.

Now suppose that \eqref{weak.sharp.aux} holds and that $T$ is constant on $X^*$. Take $x\in X$, $x^*\in X^*$ and 
let $\bar x:=\Pi_{X^*}(x)$. Since $x,\bar x\in X$ and $X$ is closed and convex, we have that $x-\bar x\in\tang_{X}(\bar x)$, 
using the first equality in \eqref{tangent:cone:charac}. Since $T$ is monotone and $X$ is closed and convex, 
$X^*$ is closed and convex (see e.g. Facchinei and Pang \cite{facchinei}, Theorem 2.3.5). 
From this fact, the fact that $\bar x=\Pi_{X^*}(x)$ and Lemma \ref{proj}(i), we obtain that $x-\bar x\in\polar_{X^*}(\bar x)$, 
using the definition of the polar cone. Thus, $x-\bar x\in\tang_{X}(\bar x)\cap\polar_{X^*}(\bar x)$. We conclude 
from  \eqref{weak.sharp.aux} that 
\begin{equation}\label{prop.weak.sharp:eq2}
\langle T(\bar x),x-\bar x\rangle\ge\rho\Vert x-\bar x\Vert=\rho\dist(x,X^*).
\end{equation}
Since $T$ is constant on $X^*$, we have 
\begin{equation}\label{e107}
\langle T(\bar x),x-\bar x\rangle=\langle T(x^*),x-\bar x\rangle=
\langle T(x^*),x-x^*\rangle+\langle T(x^*),x^*-\bar x\rangle\le
\langle T(x^*),x-x^*\rangle,
\end{equation}
using the fact that $\langle T(x^*),x^*-\bar x\rangle\le0$, which  holds because $x^*\in X^*$ and $\bar x\in X$. 
The desired claim \eqref{weakly.sharp1} follows from \eqref{e107} and \eqref{prop.weak.sharp:eq2}.
\end{proof}


\begin{thebibliography}{99}
\bibitem{auslender:teboulle} Auslender, A. and Teboulle, M. (2005) Interior projection-like methods for monotone 
variational inequalities, \emph{Mathematical Programming, Ser. A} , Vol. 104, pp. 39--68.

\bibitem{bach} Bach, F. and Moulines, E. (2011) Non-Asymptotic Analysis of Stochastic Approximation 
Algorithms for Machine Learning, \emph{Advances in Neural Information Processing Systems} (NIPS).

\bibitem{bauschke} Bauschke, H.H. (2001) Projection algorithms: results and open problems. In: Butnariu, D., Censor, Y., 
Reich, Y. (eds.) \emph{Inherently Parallel Algorithms in Feasibility and Optimization and their Applications}, Elsevier, Amsterdam, pp. 11--22.

\bibitem{bau.bor} Bauschke, H.H. and Borwein, J.M. (1996) On projection algorithms for solving convex feasibility problems, 
\emph{SIAM Review}, Vol. 38, pp. 367--426.

\bibitem{luke} Bauschke, H.H., Combettes, H.H., Luke, D.R. (2003) Hybrid projection-reflection method for phase retrieval, 
\emph{Journal of the  Optical Socety of America}, Vol. A20, pp. 1025--1034.

\bibitem{yunier} Bello Cruz, J.Y. and Iusem, A.N. (2012) An explicit algorithm for monotone variational inequalities, 
\emph{Optimization}, Vol. 61, pp. 855--871.

\bibitem{iusbel1} Bello Cruz, J.Y. and Iusem, A.N., (2010) Convergence of direct methods for paramonotone variational 
inequalities, \emph{Computational Optimization and Applications}, Vol. 46, 
pp. 247--263.

\bibitem{iusbel2} Bello Cruz, J.Y. and Iusem A.N. (2015) Full convergence of an
approximate projections method for nonsmooth variational inequalities, \emph{Mathematics and Computers in Simulation}, Vol.  114, pp. 2--13.

\bibitem{bertsekas.2} Bertsekas, D.P. (2011) Incremental proximal methods for large scale convex optimization, 
\emph{Mathematical Programming}, Vol. 129, pp.  163--195.

\bibitem{ferris}  Burke, J.V. and Ferris, M.C. (1993) Weak sharp minima in mathematical programming, 
\emph{SIAM Journal on Control and Optimization}, Vol. 31, pp. 1340--1359.

\bibitem{suchocka} Cegielski, A., Suchocka, A. (2008) Relaxed alternating projection methods. \emph{SIAM Journal on  Optimization}, 
Vol. 19, pp.  1093-1106.

\bibitem{cen} Censor, Y. (1981) Row-action methods for huge and sparse systems and its
applications, \emph{SIAM Review}, Vol. 23, pp. 444-464.

\bibitem{censor} Censor, Y. and Gibali, A. (2008) Projections onto super-half-spaces for monotone variational inequality problems 
in finite-dimensional spaces, \emph{Journal of Nonlinear and Convex Analysis}, Vol. 9, pp. 461--474.

\bibitem{lan}  Chen, Y., Lan, G. and  Ouyang, Y. Accelerated schemes for a 
class of variational inequalities, pre-print. http://www.ise.ufl.edu/glan/files/2014/03/AMP3-17-14.pdf.

\bibitem{chen&wets} Chen, X., Wets, R.J-B and Zhang, Y. (2012) Stochastic variational inequalities: residual minimization smoothing 
sample average approximations, \emph{SIAM Journal on Optimization}, Vol. 22, pp. 649--673.

\bibitem{deu.hun} Deutsch, F. and Hundal, H., (2008) The rate of convergence for the cyclic 
projections algorithm III: regularity of convex sets, \emph{Journal of Approximation Theory}, Vol. 155, pp. 155--184.

\bibitem{facchinei} Facchinei, F. and Pang, J.-S. (2003) \emph{Finite-Dimensional Variational Inequalities and Complementarity
 Problems}, Springer, New York.

\bibitem{fukushima} Fukushima, M., (1986) A relaxed projection method for variational inequalities,
\emph{Mathematical Programming}, Vol. 35, pp. 58--70.

\bibitem{robinson2} G\"urkan, G.,  \"Ozge, A.Y. and Robinson, S.M. (1999) Sample-path solution of stochastic variational 
inequalities, \emph{Mathematical Programmming}, Vol. 84, pp. 313--333.

\bibitem{huang} Huang, J., Subramanian, V.G., Agrawal, R., Berry, R. (2009) Joint scheduling and resource allocation in uplink 
OFDM systems for broadband wireless access networks. \emph{IEEE Journal on Selected Areas in Communications}, Vol. 27, pp.  226--234.

\bibitem{iusem1} Iusem, A.N. (1998) On some properties of paramonotone operators, \emph{Journal of Convex Analysis},  Vol. 5, pp. 269--278.

\bibitem{iusem} Iusem, A.N. and Svaiter, B.F. (1997) A variant of Kopelevich's method 
for variational inequalities with a new search strategy, \emph{Optimization}, Vol. 42, pp.309--321.

\bibitem{Xu} Jiang, H. and Xu, H., (2008) Stochastic approximation approaches to the stochastic variational inequality problem, 
\emph{IEEE Transactions on Automatic Control}, Vol. 53, pp. 1462--1475.

\bibitem{nem} Juditsky, A., Nemirovski, A.  and Tauvel, C. (2011) Solving variational 
inequalities with stochastic mirror-prox algorithm, \emph{Stochastic Systems}, Vol. 1, pp. 17--58.

\bibitem{uday4} Kannan, A. and Shanbhag, U.V. The pseudomonotone stochastic variational 
inequality problems: analysis and optimal stochastic approximation schemes, preprint.
http://arxiv.org/pdf/1410.1628.pdf.

\bibitem{uday1} Kannan, A. and Shanbhag, U.V.  (2012) Distributed computation of equilibria in monotone Nash games via 
iterative regularization techniques, \emph{SIAM Journal on Optimization}, 22, pp. 1177--1205.

\bibitem{kibardin} Kibardin, V.M (1980) Decomposition into functions in the minimization problem, \emph{Automatic and Remote Control},
Vol. 40, pp. 1311--1323.

\bibitem{korpelevich} Korpelevich, G.M. (1976) The extragradient method for finding saddle
points and other problems, \emph{Ekonomika i Matematcheskie Metody}, Vol. 12, pp. 747--756.

\bibitem{Uday} Koshal, J., Nedi\'c, A. and Shanbhag, U.V.  (2013) Regularized iterative stochastic approximation methods 
for stochastic variational inequality problems, \emph{IEEE Transactions on Automatic Control}, Vol 58, pp. 594--608.

\bibitem{luo.tseng} Luo, Z.-Q., Tseng, P.  (1994) Analysis of an approximate gradient projection method with applications 
to the backpropagation algorithm, \emph{Optimization Methods and Softawre}, Vol. 4, pp. 85--101.

\bibitem{marcotte} Marcotte, P. and Zhu, D. (1998) Weak sharp solutions of variational inequalities,
\emph{SIAM Journal on Optimization}, Vol. 9, pp. 179--189.

\bibitem{nedich} Nedi\'c, A. (2011) Random algorithms for convex minimization problems, \emph{Mathematical Programming}, Vol. 
129, pp. 225--253.

\bibitem{nedic.bertsekas} Nedi\'c, A., Bertsekas, D.P. (2001) Incremental subgradient method for nondifferentiable 
optimization, \emph{SIAM Journal on Optimization}, Vol. 12, pp. 109--138.

\bibitem{nem3} Nemirovski, A., Juditsky, A., Lan, G. and Shapiro, A. (2009) Robust stochastic 
approximation approach to stochastic programming, \emph{SIAM Journal on Optimization,} Vol. 19, pp. 1574--1609.

\bibitem{nem2} Nemirovski, A. (2004) Prox-method with rate of convergence O(1/t) for variational 
inequalities with Lipschitz continuous monotone operators and smooth convex-concave 
saddle point problems, \emph{SIAM Journal on  Optimization}, Vol. 15, pp. 229--251.

\bibitem{pang} Pang, J.-S. (1997) Error bounds in mathematical programming, \emph{Mathematical Programming}, Vol. 79, pp. 299--332.

\bibitem{polyak} Polyak, B. (1987) \emph{Introduction to Optimization}, Optimization Software, New York.

\bibitem{polyak.2} Polyak, B.T. (1969) Minimization of unsmooth functionals, \emph{USSR Computational Matematics and 
Mathematical Physics},  Vol. 9, pp. 14--29.

\bibitem{polyak.3} Polyak, B.T. (2001) Random algorithms for solving convex inequalities, 
In: Butnariu, D., Censor, Y., Reich, S. (eds.) \emph{Inherently Parallel Algorithms in Feasibility and Optimization 
and their Applications}, pp. 409--422. Elsevier, Amsterdam.

\bibitem{rob.monro} Robbins, H. and Monro, S., (1951) A Stochastic Approximation Method, 
\emph{The Annals of Mathematical Statistics}, Vol. 22, pp. 400--407.

\bibitem{rob} Robbins, H. and Siegmund, D.O., (1971) A convergence theorem for nonnegative almost super-martingales and 
some applications, \emph{Optimizing Methods in Statistics}, pp. 233--257.

\bibitem{terry&wets} Rockafellar, R.T. and Wets, R.J-B. (1998) \emph{Variational Analysis}, Springer, Berlin.

\bibitem{wang} Wang, M. and Bertsekas, D. (2015) Incremental Constraint Projection Methods for Variational 
Inequalities, \emph{Mathematical Programmming},  Vol. 150, pp. 321--363. 

\bibitem{wang2} Wang, M.  and Bertsekas, D. (2016) Stochastic first-order methods with random constraint projection, 
\emph{SIAM Journal on Optimization}, Vol. 26, pp. 681-717.

\bibitem{uday5} Yousefian, F.,  Nedi\'c, A. and Shanbhag, U.V. On Smoothing, Regularization and Averaging in Stochastic 
Approximation Methods for Stochastic Variational Inequalities, preprint. http://arxiv.org/abs/1411.0209.

\bibitem{uday0} Yousefian, F., Nedi\'c, A. and Shanbhag, U.V. Distributed adaptive steplength stochastic approximation schemes 
for Cartesian stochastic variational inequality problems, preprint. http://arxiv.org/abs/1301.1711.

\bibitem{uday2} Yousefian, F.,  Nedi\'c, A. and Shanbhag, U.V. (2014) Optimal robust smoothing extragradient algorithms for 
stochastic variational inequality problems, \emph{IEEE Conference on Decision and Control}, http://arxiv.org/pdf/1403.5591.pdf. 

\end{thebibliography}
\end{document}